\documentclass[english,11pt,english,a4paper, leqno]{article}
\usepackage[T1]{fontenc}
\usepackage[latin9]{inputenc}
\usepackage{textcomp}
\usepackage{amsthm}
\usepackage{amsmath}
\usepackage{amssymb}
\usepackage{esint}

\makeatletter
\theoremstyle{plain}
\newtheorem{thm}{\protect\theoremname}[section]
  \theoremstyle{plain}
  \newtheorem{lem}[thm]{\protect\lemmaname}
  \theoremstyle{remark}
  \newtheorem{rem}[thm]{\protect\remarkname}
  \theoremstyle{plain}
  \newtheorem{cor}[thm]{\protect\corollaryname}
  \theoremstyle{plain}
  \newtheorem{prop}[thm]{\protect\propositionname}

\usepackage{amssymb,amsfonts,amsmath,amsthm,amscd}
\usepackage[T1]{fontenc}
\usepackage[latin9]{inputenc}
\usepackage{geometry}
\usepackage{xcolor}
\usepackage{pdfcolmk}
\usepackage{refstyle}
\usepackage{textcomp}
\usepackage{esint}
\PassOptionsToPackage{normalem}{ulem}
\usepackage{ulem}

\makeatletter

\AtBeginDocument{}
\RS@ifundefined{subref}
  {\def\RSsubtxt{section~}\newref{sub}{name = \RSsubtxt}}
  {}
\RS@ifundefined{thmref}
  {\def\RSthmtxt{theorem~}\newref{thm}{name = \RSthmtxt}}
  {}
\RS@ifundefined{lemref}
  {\def\RSlemtxt{lemma~}\newref{lem}{name = \RSlemtxt}}
  {}

\providecolor{lyxadded}{rgb}{0,0,1}
\providecolor{lyxdeleted}{rgb}{1,0,0}

\numberwithin{equation}{section}
\numberwithin{figure}{section}
\theoremstyle{plain}

\addtocounter{section}{-1}
\date{}

\makeatother

\usepackage{babel}
  \providecommand{\corollaryname}{Corollary}
  \providecommand{\lemmaname}{Lemma}
  \providecommand{\propositionname}{Proposition}
  \providecommand{\remarkname}{Remark}
\providecommand{\theoremname}{Theorem}

\begin{document}

\title{\textbf{\normalsize{A LOWER BOUND FOR DISCONNECTION BY SIMPLE RANDOM
WALK}}}

\maketitle
\vspace{-0.7cm}

\begin{center}
{\large{Xinyi Li}}
\par\end{center}{\large \par}

\vspace{1cm}

\begin{center}
Preliminary draft
\par\end{center}
\begin{abstract}
We consider simple random walk on $\mathbb{Z}^{d}$, $d\geq3$. Motivated
by the work of A.-S. Sznitman and the author in \cite{Li-Szn ld}
and \cite{Li-Szn lb}, we investigate the asymptotic behaviour of
the probability that a large body gets disconnected from infinity
by the set of points visited by a simple random walk. We derive asymptotic
lower bounds that bring into play random interlacements. Although
open at the moment, some of the lower bounds we obtain possibly match
the asymptotic upper bounds recently obtained in \cite{dscn2-Szn}.
This potentially yields special significance to the tilted walks that
we use in this work, and to the strategy that we employ to implement
disconnection.
\end{abstract}
\vspace{8cm}

\date{\begin{flushright}{\small December 2014}\end{flushright}}

\vfill \noindent {\small -------------------------------- \\ Departement Mathematik, ETH Z\"urich, CH-8092 Z\"urich, Switzerland.} \thispagestyle{empty} 

\pagebreak{}\newpage{}

\pagebreak{}

\newpage{}

\mbox{} \thispagestyle{empty} \newpage

\section{Introduction}

How hard is it to disconnect a macroscopic body from infinity by the
trace of a simple random walk in $\mathbb{Z}^{d}$, when $d\geq3$?
In this work we partially answer this question by deriving an asymptotic
lower bound on the probability of such a disconnection. Remarkably,
our bounds bring into play random interlacements as well as a suitable
strategy to implement disconnection. Although open at the moment,
some of the lower bounds we obtain in this work may be sharp, and
match the recent upper bounds from \cite{dscn2-Szn}.

\vspace{0.3cm}

We now describe the model and our results in a more precise fashion.
We refer to Section 1 for precise definitions. We consider the continuous-time
simple random walk on $\mathbb{Z}^{d}$, $d\geq3$. and we denote
by $P_{0}$ the (canonical) law of the walk starting from the origin.
We denote by $\mathcal{V}=\mathbb{Z}^{d}\backslash X_{[0,\infty)}$
the complement of the set of points visited by the walk.

We consider $K$, a non-empty compact subset of $\mathbb{R}^{d}$
 and for $N\geq1$ its discrete blow-up:
\begin{equation}
K_{N}=\{x\in\mathbb{Z}^{d};\: d_{\infty}(x,NK)\leq1\},\label{eq:KNdef0}
\end{equation}
where $NK$ stands for the homothetic of ratio $N$ of the set $K$,
and 
\begin{equation}
d_{\infty}(z,NK)=\inf_{y\in NB}|z-y|_{\infty}
\end{equation}
stands for the sup-norm distance of $z$ to $NK$. Of central interest
for us is the event specifying that $K_{N}$ is not connected to infinity
in $\mathcal{V}$, which we denote by
\begin{equation}
\{K_{N}\overset{\mathcal{V}}{\nleftrightarrow}\infty\}.\label{eq:ANdef}
\end{equation}

Our main result brings into play the model of random interlacements.
Informally, random interlacements in $\mathbb{Z}^{d}$ are a Poissonian
cloud of doubly-infinite nearest-neighbour paths, with a positive
parameter $u$, which is a multiplicative factor of the intensity
of the cloud (we refer to \cite{CerTeilecture} and \cite{Mousquetaires}
for further details and references). We denote by $\mathcal{I}^{u}$
the trace of random interlacements of level $u$ on $\mathbb{Z}^{d}$,
and by $\mathcal{V}^{u}=\mbox{\ensuremath{\mathbb{Z}}}^{d}\backslash\mathcal{I}^{u}$
the corresponding vacant set. It is known that there is a critical
value $u_{**}\in(0,\infty)$, which can be characterized as the infimum
of the levels $u>0$ for which the probability that the vacant cluster
at the origin reaches distance $N$ from the origin has a stretched
exponential decay in $N$, see \cite{sido10} or \cite{Mousquetaires}.

The main result of this article is the following asymptotic lower
bound.
\begin{thm}
\label{thm:mainTheorem} 
\begin{equation}
\liminf_{N\to\infty}\frac{1}{N^{d-2}}\log(P_{0}[K_{N}\overset{\mathcal{V}}{\nleftrightarrow}\infty])\geq-\frac{u_{**}}{d}\mathrm{cap}_{\mathbb{R}^{d}}(K),\label{eq:mainlb}
\end{equation}
where $\mathrm{cap}_{\mathbb{R}^{d}}(K)$ stands for the Brownian
capacity of $K$. 
\end{thm}
Actually, the proof of Theorem \ref{thm:mainTheorem} (after minor
changes) also shows that for any $M>1$,
\begin{equation}
\liminf_{N\to\infty}\frac{1}{N^{d-2}}\log(P_{0}[B_{N}\overset{\mathcal{V}}{\nleftrightarrow}S_{N}])\geq-\frac{u_{**}}{d}\mathrm{cap}_{\mathbb{R}^{d}}([-1,1]^{d}),\label{eq:mainlbvariant}
\end{equation}
where $B_{N}=\{x\in\mathbb{Z}^{d};\,|x|_{\infty}\leq N\}$ and $S_{N}=\{x\in\mathbb{Z}^{d};\,|x|_{\infty}=[MN]\}$
with $[MN]$ the integer part of $MN$, see Remark \ref{finalremark}
1). 

On the other hand, the recent article \cite{dscn2-Szn} improves on
\cite{dscn-Szn}, and shows that for any $M>1$, the following asymptotic
upper bound holds
\begin{equation}
\limsup_{N\to\infty}\frac{1}{N^{d-2}}\log(P_{0}[B_{N}\overset{\mathcal{V}}{\nleftrightarrow}S_{N}])\leq-\frac{\overline{u}}{d}\mathrm{cap}_{\mathbb{R}^{d}}([-1,1]^{d}),\label{eq:mainubnew}
\end{equation}
where $\overline{u}$ is a certain critical level introduced in \cite{dscn2-Szn},
such that $0<u<\overline{u}$ corresponds to the \textit{strongly
percolative} regime of $\mathcal{V}^{u}.$ Precisely, one knows that
$0<\overline{u}\leq u_{*}\leq u_{**}<\infty$, where $u_{*}$ stands
for the critical level for the percolation of $\mathcal{V}^{u}$ (the
positivity of $\overline{u}$, for all $d\geq3$, actually stems from
\cite{ubarsource} as explained in Section 2 of \cite{dscn2-Szn}).
It is plausible, but unproven at the moment, that actually $\overline{u}=u_{*}=u_{**}$.
If this is the case, the asymptotic lower bound (\ref{eq:mainlbvariant})
from the present article matches the asymptotic upper bound (\ref{eq:mainubnew})
from \cite{dscn2-Szn}.

In the case of (\ref{eq:mainlb}), one can also wonder whether one
actually has the following asymptotics (possibly with some regularity
assumption on $K$) 
\begin{equation}
\lim_{N\to\infty}\frac{1}{N^{d-2}}\log(P_{0}[K_{N}\overset{\mathcal{V}}{\nleftrightarrow}\infty])=-\frac{u_{*}}{d}\mathrm{cap}_{\mathbb{R}^{d}}(K).\label{eq:conjecture}
\end{equation}

Our proof of Theorem \ref{thm:mainTheorem} (and of (\ref{eq:mainlbvariant}))
relies on the change of probability method. The feature that the asymptotic
lower bounds, which we derive in this article, are potentially sharp,
yields special significance to the strategy that we employ to implement
disconnection.

Let us give some comments about the strategy and the proof. We construct
through fine-tuned Radon-Nikodym derivatives new measures $\widetilde{P}_{N}$,
corresponding to the ``tilted walks''. In essence, these walks evolve
as recurrent walks with generator $\widetilde{L}g(x)\!=\!\frac{1}{2d}\sum_{|x'-x|=1}\frac{h_{N}(x')}{h_{N}(x)}(g(x')\!\!-\!\! g(x))$,
up to a deterministic time $T_{N}$, and then as the simple random
walk afterwards, with $h_{N}(x)=h(\frac{x}{N})$, where $h$ is the
solution of (assuming that $K$ is regular) 
\begin{equation}
\begin{cases}
\Delta h=0 & {\rm \mbox{ on }}\mathbb{R}^{d}\backslash K,\\
h=1 & \mbox{ on }K\mbox{, and }h\mbox{ tends to }0\mbox{ at }\infty,
\end{cases}\label{eq:Dirichletproblem}
\end{equation}
and $T_{N}$ is chosen so that the expected time spent by the tilted
walk up to $T_{N}$ at any $x$ in $K_{N}$ is $u_{**}h_{N}^{2}(x)=u_{**}$
(by the choice of $h$). Informally, $\widetilde{P}_{N}$ achieves
this at a ``low entropic cost''. Quite remarkably, this constraint
on the \textit{time} spent at points and low entropic cost, induces
a local behaviour of the trace of the tilted walk which \textit{geometrically}
behaves as random interlacements with a slowly space-modulated parameter
$u_{**}h_{N}^{2}(x)$, at least close to $K_{N}$. This creates a
``fence'' around $K_{N}$, where the vacant set left by the tilted
walk is locally in a strongly non-percolative mode, so that
\begin{equation}
\lim_{N\to\infty}\widetilde{P}_{N}[K_{N}\overset{\mathcal{V}}{\nleftrightarrow}\infty]=1.
\end{equation}
On the other hand, we show that 
\begin{equation}
\widetilde{\lim}\frac{1}{N^{d-2}}H(\widetilde{P}_{N}|P_{0})\leq\frac{u_{**}}{d}\mathrm{cap}_{\mathbb{R}^{d}}(K),\label{eq:ubproc}
\end{equation}
where $\widetilde{\lim}$ refers to a certain limiting procedure,
in which $N$ goes first to infinity, and $H(\widetilde{P}_{N}|P_{0})$
stands for the relative entropy of $\widetilde{P}_{N}$ with respect
to $P_{0}$, (see (\ref{eq:relentrodef})). The main claim (\ref{eq:mainlb}),
or (\ref{eq:mainlbvariant}) then quickly follow by a classical inequality,
see (\ref{eq:Entropychange}).

The above lines are of course mainly heuristic, and the actual proof
involves several mollifications of the above strategy: $K$ is slightly
enlarged, $h$ is replaced by a compactly supported function smoothed
near $K$, we work with $u_{**}(1+\epsilon)$ in place of $u_{**}$,
and the tilted walk lives in a ball of radius $RN$ up to time $T_{N}$,
$\ldots$. These various mollifications naturally enter the limiting
procedure alluded to above in (\ref{eq:ubproc}). 

Clearly, a substantial part of this work is to make sense of the above
heuristics. Observe that unlike what happened in \cite{Li-Szn lb}, 
where an asymptotic lower bound was derived for the disconnection of a macroscopic
body by random interlacements, in the present set-up, we only have one single trajectory at our disposal. So the titled walk behaves as a recurrent walk up to time $T_N$ in order to implement disconnection. This makes the extraction of the necessary independence implicit to comparison with random interlacements more delicate.
This is achieved by several sorts of analysis on the mesoscopic
level. More precisely, on all mesoscopic boxes $A_{1}^{x}$ with the
center $x$ varying in a ``fence'' around $K_{N}$, we bound from
above the tilted probability that there is a path in $\mathcal{V}$
that connects $x$ to the (inner) boundary of $A_{1}^{x}$ by the
probability that there is such a path in the vacant set of random
interlacements with level slightly higher than $u_{**}$ (which is
itself small due to the the strong non-percolative character of $\mathcal{V}^{u}$
when $u>u_{**}$) and a correction term: 
\begin{equation}
\widetilde{P}_{N}[x\overset{\mathcal{V}}{\longleftrightarrow}\partial_{i}A_{1}^{x}]\leq\mathbb{P}[x\overset{\mathcal{V}^{u_{**}(1+\epsilon/8)}}{\longleftrightarrow}\partial_{i}A_{1}^{x}]+e^{-c\log^{2}N}\leq e^{-c'\log^{2}N},
\end{equation}
where $\mathbb{P}$ stands for the law of random interlacements, and
$\partial_{i}A_{1}^{x}$ for the inner boundary of the box $A_{1}^{x}$.
To prove the above claim, we conduct a local comparison at mesoscopic
scale between the trace of the the tilted walk, and the occupied set
of random interlacements, with a level slightly exceeding $u_{**}$,
via a chain of couplings.

There are two crucial steps in this ``chain of couplings'', namely
Propositions \ref{prop:coupling0} and \ref{prop:coupling3}. In Proposition
\ref{prop:coupling0} we call upon the estimates on hitting times
proved in Section 3 and on the results concerning the quasi-stationary
measure from Section 4. We construct a coupling between the trace
in $A_{1}$ of excursions of the confined walk up to time $T_{N}$,
and the trace in $A_{1}$ of the excursions of many independent confined
walks from $A_{1}$ to the boundary of a larger mesoscopic box. This
proposition enables us to cut the confined walk into ``almost'' independent
sections, and compare it to the trace of a suitable Poisson point
process of excursions. On the other hand, Proposition \ref{prop:coupling3}
uses a result proved in \cite{Gumbel}, coupling the above mentioned
Poisson point process of excursions and the trace of random interlacements.
Some of the arguments used in this work are similar to those in \cite{TeiWin}.
However, in our set-up, special care is needed due to the fact that
the stationary measure of the tilted walk is massively non-uniform. 

\vspace{0.3cm}

We will now explain how this article is organized. In Section 1 we
introduce notation and make a brief review of results concerning continuous-time
random walk, continuous-time random interlacements, Markov chains,
as well as other useful facts and tools. Section 2 is devoted to the
construction of the tilted random walk and the confined walk, as well
as the proof of various properties concerning them. Most important
are a lower bound of the spectral gap of the confined walk in Proposition
\ref{prop:spectralgap}, and an asymptotic upper bound on the relative
entropy between the tilted walk and the simple random walk, in Proposition
\ref{prop:limsup}. In Section 3 we prove some estimates on the hitting
times of some mesoscopic objects, namely Propositions \ref{prop:entranceinvup}
and \ref{prop:entranceinvlb} that will be useful in Section 5. In
Section 4 we prove some controls (namely Proposition \ref{prop:qsdem})
on the quasi-stationary measure that will be crucial for the construction
of couplings in Section 5. In Section 5 we develop the chain of couplings
and prove that the tilted disconnection probability $\widetilde{P}_{N}[A_{N}]$
tends to $1$, as $N$ tends to infinity. In the short Section 6 we
assemble the various pieces and prove the main Theorem \ref{thm:mainTheorem}.

Finally, we explain the convention we use concerning constants. We
denote by $c$, $c'$, $c''$, $\bar{c}\ldots$ positive constants
with values changing from place to place. Throughout the article the
constants depend on the dimension $d$. Dependence on additional constants
is stated at the beginning of each section.\vspace{0.3cm}

\textbf{Acknowledgement.} The author wishes to thank his advisor A.-S.
Sznitman for suggesting the problem and for useful discussions.

\section{Some useful facts}

Throughout the article we assume $d\geq3$ unless otherwise stated.
In this section we will introduce further notation and recall useful
facts concerning continuous-time random walk on $\mathbb{Z}^{d}$
and its potential theory. We also recall the definition of and some
results about continuous-time random interlacements. At the end of
the section we state an inequality on relative entropy and review
various results about Markov chain.

We start with some notation. We write $|\cdot|$ and $|\cdot|_{\infty}$
for the Euclidean and $l^{\infty}$-norms on $\mathbb{R}^{d}$. We
denote by $B(x,r)=\{y\in\mathbb{Z}^{d};\:|x-y|\leq r\}$ the closed
Euclidean ball of center $x$ and radius $r\geq0$ intersected with
$\mathbb{Z}^{d}$ and by $B_{\infty}(x,r)=\{y\in\mathbb{Z}^{d},\:|x-y|_{\infty}\leq r\}$
the closed $l^{\infty}$-ball of center $x$ and radius $r$ intersected
with $\mathbb{Z}^{d}$. When $U$ is a subset of $\mathbb{Z}^{d}$,
we write $|U|$ for the cardinality of $U$, and $U\subset\subset\mathbb{Z}^{d}$
means that $U$ is a finite subset of $\mathbb{Z}^{d}$. We denote
by $\partial U$ (resp. $\partial_{i}U$) the boundary (resp. internal
boundary) of $U$, and by $\overline{U}$ its ``closure''
\begin{align}
\begin{split}\partial U & =\{x\in U^{c};\:\exists y\in U,\ |x-y|=1\}\\
\partial_{i}U & =\{y\in U;\:\exists U^{c},\:|x-y|=1\}\quad\textrm{ and }\quad\overline{U}=U\cup\partial U.
\end{split}
\label{eq:boundarydef}
\end{align}
When $U\subset\mathbb{\mathbb{R}}^{d}$ , and $\delta>0$ , we write
$U^{\delta}=\{z\in\mathbb{R}^{d};\: d(z,U)\leq\delta\}$ for the closed
$\delta$-neibourhood of $U$, where $d(x,A)=\inf_{y\in A}|x-y|$
is the Euclidean distance of $x$ to $A$. We define $d_{\infty}(x,A)$
in a similar fashion, with $\big|\cdot\big|_{\infty}$ in place of
$\big|\cdot\big|$. To distinguish balls in $\mathbb{R}^{d}$ from
balls in $\mathbb{Z}^{d}$, we write $B_{\mathbb{R}^{d}}(x,r)\{z\in\mathbb{R}^{d};\:|x-z|\leq r\}$
for the closed Euclidean ball of center $x$ and radius $r$ in $\mathbb{R}^{d}$
and $B_{\mathbb{R}^{d}}^{\circ}(x,r)=\{z\in\mathbb{R}^{d};\:|x-z|<r\}$
for the corresponding open Euclidean ball. We also write the $N$-discrete
blow-up of $U\subseteq\mathbb{R}^{d}$ as 
\begin{equation}
U_{N}=\{x\in\mathbb{Z}^{d};\, d_{\infty}(x,NU)\le1\},\label{eq:blowupdef-1}
\end{equation}
where we denote by $NU=\{Nz;\: z\in U\}$  the homothetic of $U$. 

We will now collect some notation concerning connectivity properties.
We write $x\sim y$ if for $x,y\in\mathbb{Z}^{d}$, $|x-y|=1$. We
call $\pi:\{1,\ldots n\}\to\mathbb{Z}^{d}$ with $n\geq1$ a nearest-neighbour
path, when $\pi(i)\sim\pi(i-1)$ for $1<i\leq n$. Given $K,L,U$
subsets of $\mathbb{Z}^{d}$, we say that $K$ and $L$ are connected
by $U$ and write $K\overset{U}{\longleftrightarrow}L$, if there
exists a finite nearest-neighbour path $\pi$ in $\mathbb{Z}^{d}$
such that $\pi(1)$ belongs to $K$ and $\pi(n)$ belongs to $L$,
and for all $k\in\{1,\cdots,n\}$, $\pi(k)$ belongs to $U$. Otherwise,
we say that $K$ and $L$ are not connected by $U$, and write $K\overset{U}{\nleftrightarrow}L$.
Similarly, we say that $K$ is connected to infinity by $U$, if for
$K,U$ subsets of $\mathbb{Z}^{d}$, $K\overset{U}{\longleftrightarrow}B(0,N)^{c}$
for all $N$, and write $K\overset{U}{\longleftrightarrow}\infty$.
Otherwise we say $K$ is not connected to infinity by $U$, and write
$K\overset{U}{\nleftrightarrow}\infty$. 

\vspace{0.3cm}

We now turn to the definition of some path spaces and of the continuous-time
simple random walk. We consider $\widehat{W}_{+}$ the spaces of infinite
$(\mathbb{Z}^{d})\times(0,\infty)$-valued sequences such that the
first coordinate of the sequence forms an infinite nearest neighbour
path in $\mathbb{Z}^{d}$, spending finite time in any finite set
of $\mathbb{Z}^{d}$, and the sequence of the second coordinate has
an infinite sum. The second coordinate describes the duration at each
step corresponding to the first coordinate. We denote by $\widehat{\mathcal{W}}_{+}$
the respective $\sigma$-algebra generated by the coordinate maps,
$Z_{n}$, $\zeta_{n}$, $n\geq0$ (where $Z_{n}$ is $\mathbb{Z}^{d}$-valued
and $\zeta_{n}$ is $(0,\infty)$-valued). We denote by $P_{x}$ the
law on $\widehat{W}_{+}$ under which $Z_{n}$, $n\geq0$, has the
law of the simple random walk on $\mathbb{Z}^{d}$, starting from
$x$, and $\zeta_{n}$, $n\geq0$, are i.i.d. exponential variables
with parameter $1$, independent from $Z_{n}$, $n\geq0$. We denote
by $E_{x}$ the corresponding expectation. Moreover, if $\alpha$
is a measure on $\mathbb{Z}^{d}$, we denote by $P_{\alpha}$ and
$E_{\alpha}$ the measure $\sum_{x\in\mathbb{Z}^{d}}\alpha(x)P_{x}$
(not necessarily a probability measure) and its corresponding ``expectation''
(i.e. the integral with respect to the measure $P_{\alpha}$). 

We attach to $\widehat{w}\in\widehat{W}_{+}$ a continuous-time process
$(X_{t})_{t\geq0}$ and call it the random walk on $\mathbb{Z}^{d}$
with constant jump rate $1$ under $P_{x}$, through the following
relations
\begin{equation}
X_{t}(\widehat{w})=Z_{k}(\widehat{w})\textrm{ for }t\geq0,\textrm{ when }\tau_{k}\leq t<\tau_{k+1},\label{eq:skeleton chain}
\end{equation}
where for $l$ in $\mathbb{Z}^{+}$, we set (if $l=0$, the right
sum term is understood as 0), 
\begin{equation}
\tau_{l}=\sum_{i=0}^{l-1}\zeta_{i}.\label{eq:taudef}
\end{equation}
We also introduce the filtration
\begin{equation}
\mathcal{F}_{t}=\sigma(X_{s},\: s\leq t).\label{eq:filtration}
\end{equation}
For $I$ a Borel subset of $\mathbb{R}^{+}$, we record the set of
points visited by $(X_{t})_{t\geq0}$ during the time set $I$ as
$X_{I}$. Importantly, we denote by $\mathcal{V}$ the vacant set,
namely the complement of the entire trace $X_{[0,\infty)}$ of $X$.

Given a subset $U$ of $\mathbb{Z}^{d}$, and $\widehat{w}\in\widehat{W}_{+}$,
we write $H_{U}(\widehat{w})=\inf\{t\geq0;\: X_{t}(\widehat{w})\in U\}$
and $T_{U}=\inf\{t\geq0;\: X_{t}(\widehat{w})\notin U\}$ for the
entrance time in $U$ and exit time from $U$. Moreover, we write
$\widetilde{H}_{U}=\inf\{s\geq\zeta_{1};\: X_{s}\in U$\} for the
hitting time of $U$. If $U=\{x\}$ we then write $H_{x}$, $T_{x}$
and $\widetilde{H}_{x}$.

Given a subset $U$ of $\mathbb{Z}^{d}$, we write $\Gamma(U)$ for
the space of all right-continuous, piecewise constant functions from
$[0,\infty)$ to $U$, with finitely many jumps on any compact interval.
We will also denote by $(X_{t})_{t\geq0}$ the canonical coordinate
process on $\Gamma(U)$, and when an ambiguity arises, we will specify
on which space we are working. For $\gamma\in\Gamma(U)$, we denote
by $\mathrm{Range}(\gamma)$ the trace of $\gamma$. 

\vspace{0.3cm}

Now, we recall some facts concerning equilibrium measure and capacity,
and refer to Section 2, Chapter 2 of \cite{Intersection} for more
details. Given $M\subset\subset\mathbb{Z}^{d}$, we write $e_{M}$
for the equilibrium measure of $M$:
\begin{equation}
e_{M}(x)=P_{x}[\widetilde{H}_{M}=\infty]1_{M}(x),\: x\in\mathbb{Z}^{d},\label{eq:emdef}
\end{equation}
and $\mathrm{cap}(M)$ for the capacity of $M$, which is the total
mass of $e_{M}$:
\begin{equation}
\mathrm{cap}(M)=\sum_{x\in M}e_{M}(x).
\end{equation}
We denote the normalized equilibrium measure of $M$ by
\begin{equation}
\widetilde{e}_{M}(x)=\frac{e_{M}(x)}{\mathrm{cap}(M)}.\label{eq:nemdef}
\end{equation}
There is also an equivalent definition of capacity through the Dirichlet
form:
\begin{equation}
\mathrm{cap}(M)=\inf_{f}\mathcal{E}_{\mathbb{Z}^{d}}(f,f)
\end{equation}
where $f:\mathbb{Z}^{d}\to\mathbb{R}$ is finitely supported and $f\geq1$
on $M$, and
\begin{equation}
\mathcal{E}_{\mathbb{Z}^{d}}(f,f)=\frac{1}{2}\sum_{|x-y|=1}\frac{1}{2d}(f(y)-f(x))^{2}\label{eq:Dirichletdef}
\end{equation}
is the discrete Dirichlet form for simple random walk. 

It is well known that (see e.g., Section 2.2, pp. 52-55 of \cite{Intersection})
\begin{equation}
cN^{d-2}\leq\mathrm{cap}(B_{\infty}(0,N))\leq c'N^{d-2},\label{eq:capacityorder}
\end{equation}
and that
\begin{equation}
e_{B_{\infty}(0,N)}(x)\geq c_{1}N^{-1}\label{eq:emboxlb}
\end{equation}
for $x$ on the inner boundary of $B_{\infty}(0,N)$.

\vspace{0.3cm}

Now, we turn to random interlacements. We refer to \cite{CerTeilecture},
\cite{Mousquetaires}, \cite{vacant} and \cite{Isomorphism} for
more details. Random interlacements are random subsets of $\mathbb{Z}^{d}$,
governed by a non-negative parameter $u$ (referred to as the ``level''),
and denoted by $\mathcal{I}^{u}$. We write $\mathbb{P}$ for the
law of $\mathcal{I}^{u}$. Although the construction of random interlacements
is involved, the law of $\mathcal{I}^{u}$ can be simply characterized
by the following relation: 
\begin{equation}
\mathbb{P}[\mathcal{I}^{u}\cap K=\emptyset]=e^{-u\mathrm{cap}(K)}\qquad\textrm{ for all }K\subset\subset\mathbb{Z}^{d}.\label{eq:RIdef}
\end{equation}
We denote by $\mathcal{V}^{u}=\mathbb{Z}^{d}\backslash\mathcal{I}^{u}$
the vacant set of random interlacements at level $u$.

The connectivity function of the vacant set of random interlacements
is known to have a stretched-exponential decay when the level exceeds
a certain critical value (see Theorem 4.1 of \cite{decoupling}, Theorem
0.1 of \cite{sido10}, or Theorem 3.1 of \cite{softlocaltime} for
recent developments). Namely, there exists a $u_{**}\in(0,\infty)$
which, for our purpose in this article, can be characterized as the
smallest positive number such that for all $u>u_{**}$, 
\begin{equation}
\mathbb{P}[0\overset{\mathcal{V}^{u}}{\longleftrightarrow}\partial B_{\infty}(0,N)]\leq c(u)e^{-c'(u)N^{c'(u)}}\textrm{, for all }N\geq0,\label{eq:supercrit}
\end{equation}
(actually, the exponent of $N$ can be chosen as $1$, when $d\geq4$,
and as an arbitrary number in $(0,1)$ when $d=3$, see \cite{softlocaltime}).

\vspace{0.3cm}

We also wish to recall a classical result about relative entropy,
which is helpful in Section \ref{sec:Tilted walk}. For $\widetilde{P}$
absolutely continuous with respect to $P$, the relative entropy of
$\widetilde{P}$ with respect to $P$ is defined as
\begin{equation}
H(\widetilde{P}|P)=E^{\widetilde{P}}[\log\frac{d\widetilde{P}}{dP}]=E^{P}[\frac{d\widetilde{P}}{dP}\log\frac{d\widetilde{P}}{dP}]\in[0,\infty].\label{eq:relentrodef}
\end{equation}
For an event $A$ with positive $\widetilde{P}-$probability, we have
the following inequality (see p. 76 of \cite{LD}):
\begin{equation}
P[A]\ge\widetilde{P}[A]e^{-(H(\widetilde{P}|P)+1/e)/\widetilde{P}[A]}.\label{eq:Entropychange}
\end{equation}

\vspace{0.3cm}

We end this section with some results regarding continuous-time reversible
finite Markov chain. 

Let $L$ be the generator for an irreducible, reversible continuous-time
Markov chain with possibly non-constant jump rate on a finite set
$V$. Let $\pi$ be the stationary measure of this Markov chain. Then
$-L$ is self-adjoint in $l^{2}(\pi)$ and has non-negative eigenvalues
$0=\lambda_{1}<\lambda_{2}\leq\ldots\leq\lambda_{|V|}$. We denote
by $\lambda=\lambda_{2}$ its spectral gap. For any real function
$f$ on $V$ we define its variance with respect to $\pi$ as $\mathrm{Var}_{\pi}(f)$.
Then the semigroup $H_{t}=e^{tL}$ satisfies
\begin{equation}
||H_{t}f-\pi(f)||_{2}\leq e^{-\lambda t}\sqrt{\mathrm{Var}_{\pi}(f)}.\label{eq:spectralgap}
\end{equation}
One can further show that, for all $x$ and $y$ in $V$, 
\begin{equation}
\Big|P_{x}(X_{t}=y)-\pi(y)\Big|\leq\sqrt{\frac{\pi(y)}{\pi(x)}}e^{-\lambda t},\label{eq:spectralgap2}
\end{equation}
 see pp. 326-328 of \cite{Saloff-Coste} for more detail. 

We also introduce the so-called ``canonical path method'' to give
a lower bound on the spectral gap $\lambda$. We denote by $E$ the
edge set 
\begin{equation}
\big\{\{x,y\};\, x,y\in V,\, L_{x,y}>0\big\}
\end{equation}
where $L_{x,y}$ is the matrix coefficient of $L$. We investigate
the following quantity $A$
\begin{equation}
A=\max_{e\in E}\{\frac{1}{W(e)}\sum_{x,y,\gamma(x,y)\ni e}\mathrm{leng}(\gamma(x,y))\pi(x)\pi(y)\}\label{eq:Adef}
\end{equation}
where $\gamma$ is a map, which sends ordered pairs of vertices $(x,y)\in V\times V$
to a finite path $\gamma(x,y)$ between $x$ and $y$, $\mathrm{leng}(\gamma)$
denotes the length of $\gamma$, and
\begin{equation}
W(e)=\pi(x)L_{x,y}=\pi(y)L_{y,x}=(1_{x},L1_{y})_{l^{2}(\pi)}=(L1_{x},1_{y})_{l^{2}(\pi)}\label{eq:edgedef}
\end{equation}
is the edge-weight of $e=\{x,y\}\in E$. Then the proof of Theorem
3.2.1, p. 369 of \cite{Saloff-Coste} is also valid (note that actually
$e$ in \cite{Saloff-Coste} is an oriented edge) in the present set-up
of possibly non-constant jump rates and shows that
\begin{equation}
\lambda\geq\frac{1}{A}.\label{eq:pathmethod}
\end{equation}

\section{{\normalsize{\label{sec:Tilted walk}}}The tilted random walk}

In this section, we construct the main protagonists of this work:
a time non-homogenous Markov chain on $\mathbb{Z}^{d}$, which we
will refer to as \textit{the tilted walk}, as well as a continuous-time
homogenous Markov chain on a (macroscopic) finite subset of $\mathbb{Z}^{d}$,
which we will refer to as \textit{the confined walk}. The tilted walk
coincides with the confined walk up to a certain finite time, which
is of order $N^{d}$, and then evolves as a simple random walk. We
derive a lower bound on the spectral gap of the confined walk in Proposition
\ref{prop:spectralgap}. In Proposition \ref{prop:limsup}, we prove
that with a suitable limiting procedure, the relative entropy between
the tilted random walk and the simple random walk has an asymptotic
upper bound given by a quantity involving the Brownian capacity of
$K$ that appears in Theorem \ref{thm:mainTheorem}. In this section
the constants tacitly depend on $\delta$, $\eta$, $\epsilon$ and
$R$ (see (\ref{eq:deltadef}) and (\ref{eq:Rr})).

We recall that $K$ is a compact subset of $\mathbb{R}^{d}$ as above
(\ref{eq:KNdef0}). We assume, without loss of generality, that 
\begin{equation}
0\in K.\label{eq:0inK}
\end{equation}
Otherwise, as we now explain, we can replace $K$ by $\widetilde{K}=K\cup\{0\}$:
on the one hand, by the monotonicity and subadditivity of Brownian
capacity (see for example Proposition 1.12, p. 60 of \cite{Portstone}),
one has $\mathrm{cap}_{\mathbb{R}^{d}}(K)=\mathrm{cap}_{\mathbb{R}^{d}}(\widetilde{K})$;
on the other hand, since $K\subseteq\widetilde{K}$, it is more difficult
to disconnect $\widetilde{K}_{N}$ than to disconnect $K_{N}$, hence
$P_{0}[K_{N}\overset{\mathcal{V}}{\nleftrightarrow}\infty]\geq P_{0}[\widetilde{K}_{N}\overset{\mathcal{V}}{\nleftrightarrow}\infty]$.
This means that the lower bound (\ref{eq:mainlb}) with $K$ replaced
by $\widetilde{K}$ implies $(\ref{eq:mainlb})$, justifying our claim.
From now on, for the sake of simplicity, for any $r>0$ we write $B_{(r)}$
for the open ball $B_{\mathbb{R}^{d}}^{\circ}(0,r)$ and $B_{r}$
for the closed ball $B_{\mathbb{R}^{d}}(0,r)$. We introduce the three
parameters 
\begin{equation}
0<\delta,\eta,\epsilon<1,\label{eq:deltadef}
\end{equation}
where $\delta$ will be used as a smoothing radius for $K$, see (\ref{eq:K2dinclusion}),
$\eta$ will be used as a parameter in the construction of $\widetilde{h}$,
the smoothed potential function, see (\ref{eq:thproperty}), and $\epsilon$
will be used as a parameter in the definition of $T_{N}$, the time
length of ``tilting'', see (\ref{eq:TNdef}). We let $R>400$ be a
large integer (see Remark \ref{integer} for explanations on why we
take $R$ to be an integer) such that 
\begin{equation}
K\subset B_{R/100}.\label{eq:Rr}
\end{equation}
By definition of $R$ we always have
\begin{equation}
K^{2\delta}\subset B_{R/50}.\label{eq:K2dinclusion}
\end{equation}

\vspace{0.3cm}
In the next lemma, we show the existence of a function $\widetilde{h}$
that satisfies various properties (among which the most important
is an inequality relating its Dirichlet form to the relative Brownian
capacity of $K^{2\delta}$), which, as we will later show, make it
the right candidate for the main ingredient in the construction of
the tilted walk. 

We denote by $\mathcal{E}_{\mathbb{R}^{d}}(f,f)=\frac{1}{2}\int_{\mathbb{R}^{d}}|\nabla f(x)|^{2}dx$
for $f\in H^{1}(\mathbb{R}^{d})$ the usual Dirichlet form on $\mathbb{R}^{d}$
(see Example 4.2.1, p. 167 and (1.2.12), p. 11 of \cite{Fukushima}),
and by $\mathrm{cap}_{\mathbb{R}^{d},B_{(R)}}(K^{2\delta})$ the relative
Brownian capacity of $K^{2\delta}$ with respect to $B_{(R)}$. We
write $C^{\infty}(B_{R})$ for the set of functions having all derivatives
of every order continuous in $B_{(R)}$, which all have continuous
extensions to $B_{R}$ (see p. 10 of \cite{GilTru} for more details). 
\begin{lem}
There exists a continuous function $\widetilde{h}:\mathbb{R}^{d}\to\mathbb{R}$,
satisfying the following properties
\begin{equation}
\begin{cases}
1. & \widetilde{h}\mbox{ is a }\ensuremath{C^{\infty}(B_{R})}\mbox{ function when restricted to \ensuremath{B_{R}}, and harmonic on }B_{(R)}\backslash B_{R/2};\\
2. & 0\leq\widetilde{h}(z)\leq1\mbox{ for all }z\in\mathbb{R}^{d}\mbox{, }\widetilde{h}=1\mbox{ on }K^{2\delta}\mbox{, and }\widetilde{h}(z)=0\mbox{ outside }B_{(R)};\\
3. & {\cal E}_{\mathbb{R}^{d}}(\widetilde{h},\widetilde{h})\leq(1+\eta)^{2}\mathrm{cap}_{\mathbb{R}^{d},B_{(R)}}(K^{2\delta});\\
4. & cw_{1}\leq\widetilde{h}\leq c'w_{2}\mbox{ where }w_{1},w_{2}\mbox{ are defined respectively in (\ref{eq:w1}) and (\ref{eq:w2});}\\
5. & \widetilde{h}(z_{1})\geq c\widetilde{h}(z_{2})\mbox{ for all }z_{1},z_{2}\in\mathbb{R}^{d}\mbox{ such that }|z_{1}|\leq|z_{2}|\leq R.
\end{cases}\label{eq:thproperty}
\end{equation}
\end{lem}
\begin{proof}
We now construct $\widetilde{h}$. We define, with $\delta$ as in
(\ref{eq:deltadef}), 
\begin{equation}
h(z)=W_{z}[H_{K^{2\delta}}<T_{B_{(R)}}],\qquad\forall z\in\mathbb{R}^{d},\label{eq:defh}
\end{equation}
the Brownian relative equilibrium potential function, where $W_{z}$
stands for the Wiener measure starting from $z\in\mathbb{R}^{d}$,
and $H_{K^{2\delta}}$ and $T_{B_{(R)}}$ respectively stand for the
entrance time of the canonical Brownian motion in $K^{2\delta}$ and
its exit time from $B_{(R)}$. 

We let $\psi:\mathbb{R}\to\mathbb{R}$ be a smooth, non-decreasing
and concave function such that $0\leq(\psi)'(z)\leq1$ for all $z\in\mathbb{R}$,
$\psi(z)=z$ for $z\in(-\infty,\frac{1}{2}]$, and $\psi(z)=1$ for
$z\in[1+\frac{\eta}{2},\infty)$. We consider
\begin{equation}
\widetilde{h}=\psi\circ\big((1+\eta)h\big).\label{eq:tiltedhdef}
\end{equation}

Now we prove the claims. 

We first prove claim 1.~in (\ref{eq:thproperty}). It is classical
that $h$ is $C^{\infty}$ on $B_{(R)}\backslash K^{2\delta}$. In
addition, $h$ is continuous, equal to $1$ on $K^{2\delta}$ and
to $0$ outside $B_{(R)}$ (note that every point in $K^{2\delta}$
is regular for $K^{2\delta}$). In particular, $(1+\eta)h\geq1+\eta/2$
on an open neighbourhood of $K^{2\delta}$, which implies that $\widetilde{h}$
is identically equal to $1$ on this neighbourhood. It follows that
$\widetilde{h}$ is $C^{\infty}$ on $B_{(R)}$. From the representation
of $h$ with the killed Green function in $B_{(R)}$, one also sees
that $\widetilde{h}$ is $C^{\infty}$ on $B_{R}$.

Claim 2.~follows directly from the definition of $\widetilde{h}$:
for all $z\in\mathbb{R}^{d}$, $(1+\eta)h(z)\in[0,1+\eta]$, hence
by the definition of $\psi$, $\widetilde{h}(z)\in[0,1]$; $\widetilde{h}=1$
on $K^{2\delta}$ is already shown in claim 1. of (\ref{eq:thproperty});
moreover, by (\ref{eq:defh}), outside $B_{(R)}$, $h=0$, hence $\widetilde{h}=0$.

We now prove claim 3. By $({\cal E}.4)''$,$\mbox{ p.\,5 }$of \cite{Fukushima},
an equivalent characterization of Markovian Dirichlet form, one knows
that since $\psi$ is a normal contraction, 
\begin{equation}
\mathcal{E}_{\mathbb{R}^{d}}(\widetilde{h},\widetilde{h})\leq\mathcal{E}_{\mathbb{R}^{d}}((1+\eta)h,(1+\eta)h)=(1+\eta)^{2}\mathcal{E}_{\mathbb{R}^{d}}(h,h)=(1+\eta)^{2}\mathrm{cap}_{\mathbb{R}^{d},B_{(R)}}(K^{2\delta}),
\end{equation}
where the last equality follows from \cite{Fukushima}, pp. 152 and
71. 

We now turn to claim 4. Because $B_{\delta}\subset K^{2\delta}\subset B_{R/50}$
by (\ref{eq:0inK}), we know that 
\begin{equation}
w_{1}\leq h\leq w_{2},\mbox{ on }B_{R},\label{eq:w1hw2}
\end{equation}
where 
\begin{equation}
w_{1}(z)=W_{z}[H_{B_{\delta}}<T_{B_{(R)}}]=\begin{cases}
1 & |z|\in[0,\delta)\\
\frac{|z|^{2-d}-R^{2-d}}{\delta{}^{2-d}-R^{2-d}} & |z|\in[\delta,R)\\
0 & |z|\in[R,\infty)
\end{cases}\label{eq:w1}
\end{equation}
and 
\begin{equation}
w_{2}(z)=W_{z}[H_{B_{R/50}}<T_{B_{(R)}}]=\begin{cases}
1 & |z|\in[0,R/50)\\
\frac{|z|^{2-d}-R^{2-d}}{(R/50){}^{2-d}-R^{2-d}} & |z|\in[R/50,R)\\
0 & |z|\in[R,\infty).
\end{cases}\label{eq:w2}
\end{equation}
are respectively the Brownian relative equilibrium potential functions
of $B_{\delta}$ and $B_{R/50}$ (see (4) in Section 1.7, $\mbox{p. 29}$
of \cite{BMMA} for the explicit formula of $w_{1}$ and $w_{2}$).
By the definition of $\psi$, we also know that, $cr\leq\psi(r)\leq c'r$
for $0\leq r\leq1+\eta$. Hence by the definition of $\widetilde{h}$,
we find that 
\begin{equation}
\widetilde{c}w_{1}\overset{(\ref{eq:w1hw2})}{\leq}ch\leq\widetilde{h}\leq c'h\overset{(\ref{eq:w1hw2})}{\leq}\widetilde{c'}w_{2}.\label{eq:hhtilde}
\end{equation}
Claim 4. hence follows.

Finally, claim 5.~follows by claim 4.~and the observation from the
explicit formula of $w_{1}$ and $w_{2}$ that $w_{1}\geq cw_{2}$
uniformly for some positive $c$ on $B_{R}$ and both $w_{1}$ and
$w_{2}$ are radially symmetric and radially non-increasing: 
\begin{equation}
\widetilde{h}(z_{1})\geq cw_{1}(z_{1})\geq c'w_{2}(z_{1})\geq c'w_{2}(z_{2})\geq c''\widetilde{h}(z_{2})\mbox{ for }z_{1},z_{2}\mbox{ such that }|z_{1}|\leq|z_{2}|\leq R.
\end{equation}

\end{proof}
We then introduce the restriction to $\mathbb{Z}^{d}$ of the blow-up
of $\widetilde{h}$ and its $l^{2}(\mathbb{Z}^{d})-$normalization
as 
\begin{equation}
h_{N}(x)=\widetilde{h}(\frac{x}{N})\textrm{ for }x\in\mathbb{Z}^{d};\mbox{ and }f(x)=\frac{h_{N}(x)}{||h_{N}||_{2}},\label{eq:hN+fdef}
\end{equation}
and also set (see (\ref{eq:deltadef}) for the definition of $\epsilon$)
\begin{equation}
T_{N}=u_{**}(1+\epsilon)||h_{N}||_{2}^{2},\label{eq:TNdef}
\end{equation}
(recall that $u_{**}$ is the threshold of random interlacements defined
above (\ref{eq:supercrit})). We define $T_{N}$ in a way such that
the quantity $T_{N}f^{2}$ is bigger than $u_{**}$ on $K_{N}^{\delta}$,
which, roughly speaking, makes the occupational time profile of the
tilted random walk (which we will later define) at time $T_{N}$ on
$K_{N}^{\delta}$ bigger than that of the random interlacement with
intensity $u_{**}$. We also set 
\begin{equation}
U^{N}=B_{(NR)}\cap\mathbb{Z}^{d}.\label{eq:UNdef}
\end{equation}
This will be the state space of the confined random walk that we will
later define.

In the following lemma we record some basic properties of $f$. Intuitively
speaking, $f$ is a volcano-shaped function, with maximal value on
$K_{N}^{\delta}$ that vanishes outside $U^{N}$. Note that $f$ tacitly
depends on $\delta$, $\eta$ and $R$.
\begin{lem}
\label{lem:fproperty}For large $N$, one has
\begin{equation}
\begin{cases}
1. & f\textrm{ is supported on }U^{N}\mbox{ and }f>0\mbox{ on }U^{N};\\
2. & f^{2}\textrm{ is a probability measure on }\mathbb{Z}^{d}\mbox{ supported on }U^{N};\\
3. & T_{N}f^{2}(\cdot)=u_{**}(1+\epsilon)\textrm{ on }K_{N}^{\delta}.
\end{cases}\label{eq:fproperty}
\end{equation}
\end{lem}
\begin{proof}
Claims 1.~and 2.~follow by the definition of $f$ (see (\ref{eq:hN+fdef}))
and $U^{N}$ (see (\ref{eq:UNdef})), note that by (\ref{eq:UNdef})
$x\in U^{N}$ implies $\frac{x}{N}$ belongs to the \textit{open}
ball $B_{(R)}$. Claim 3. follows from the definition of $T_{N}$
(see (\ref{eq:TNdef})) and the fact that $h_{N}=1$ on $K_{N}^{\delta}$
for large $N$.
\end{proof}
We introduce a subset of $U^{N}$ (which will be used in Lemma \ref{lem:vbound})
\begin{equation}
O^{N}=\Big\{ U^{N}\backslash\big(\partial_{i}U^{N}\cup B_{NR/2}\big)\Big\}\bigcup\Big\{ x\in\partial_{i}U^{N};\,|y|=NR\mbox{ for all }y\sim x,\, y\notin U^{N}\Big\}.\label{eq:ONdef}
\end{equation}
Intuitively speaking, $O^{N}$ denotes the set of points in $U^{N}$
which have distance at least $NR/2$ from $0$ such that all their
neighbours outside $U^{N}$ (if they exist) are on the sphere $\partial B_{NR}$.
In the next lemma we collect some properties of $h_{N}$ and $T_{N}$
for later use, in particular in the proofs of Lemmas \ref{lem:pippy},
\ref{lem:vbound} and Propositions \ref{prop:entropyub}, \ref{prop:limsup}. 
\begin{lem}
\label{lem:hNproperty}For large $N$, one has
\begin{equation}
\begin{cases}
1. & cN^{-2}\leq h_{N}(x)\leq1\mbox{ for all }x\in U^{N};\\
2. & h_{N}(x)\leq cN^{-1}\mbox{ for all }x\in\partial_{i}U^{N};\\
3. & h_{N}(x)\geq c'N^{-1}\mbox{ for all }x\in O^{N};\\
4. & c'N^{d}\leq||h_{N}||_{2}^{2}\leq c''N^{d};\\
5. & cN^{d}\leq T_{N}\leq c'N^{d}.
\end{cases}\label{eq:hNproperty}
\end{equation}
\end{lem}
\begin{proof}
We first prove claim 1. The right-hand side inequality follows by
the definition of $h_{N}$ (see (\ref{eq:hN+fdef})) and $\widetilde{h}$
(see (\ref{eq:tiltedhdef})). We now turn to the left-hand side inequality
of claim 1. For all $x\in U^{N}$, one has $|x|^{2}<(NR)^{2}$ by
the definition of $U^{N}$ (see (\ref{eq:UNdef})). Since $x$ has
integer coordinates, this implies $|x|\leq\sqrt{(NR)^{2}-1}$, hence
for all $x\in U^{N}$, 
\begin{equation}
|x|\leq NR-cN^{-1}.\label{eq:integerreason}
\end{equation}
Thus, by claim 4. of (\ref{eq:thproperty}) and (\ref{eq:w1}) one
has 
\begin{equation}
\widetilde{h}(z)\geq c'N^{-2}\mbox{ for all }|z|\leq R-\frac{c}{N^{2}}.
\end{equation}
This implies that for large $N$, for all $x\in U^{N}$, 
\begin{equation}
h_{N}(x)=\widetilde{h}(\frac{x}{N})\geq c''N^{-2}.\label{eq:hNlb}
\end{equation}

Similarly, to prove claims 2., and 3., again by claim 4. of (\ref{eq:thproperty})
and respectively (\ref{eq:w2}) and (\ref{eq:w1}), it suffices to
prove that 
\begin{equation}
|x|\geq NR-1\mbox{ , }\forall x\in\partial_{i}U^{N}\label{eq:UNib}
\end{equation}
and that
\begin{equation}
|x|\leq NR-c'\mbox{, }\forall x\in O^{N}.\label{eq:UNmSN}
\end{equation}
To prove (\ref{eq:UNib}), we observe that, if $x\in\partial_{i}U^{N}$,
there exists $y\notin U^{N}$, such that $x\sim y$. Since $|y|\geq NR$
and $|x-y|=1$, the claim (\ref{eq:UNib}) follows by triangle inequality.
Now we prove (\ref{eq:UNmSN}). We consider $x=(a_{1},\cdots,a_{d})\in O^{N}$.
By definition of $O^{N}$ (see (\ref{eq:ONdef})), $a_{1}^{2}+\cdots a_{d}^{2}\geq c(NR)^{2}$,
hence without loss of generality, we assume that $|a_{1}|\geq cNR$.
By the definition of $O^{N}$, we also know that $(|a_{1}|+1)^{2}+a_{2}^{2}+\cdots+a_{d}^{2}\leq(NR)^{2}$,
which implies that $|x|=\sqrt{a_{1}^{2}+\cdots a_{d}^{2}}\leq NR(1-c'/N)^{1/2}\leq NR-c''$,
and hence (\ref{eq:UNmSN}). 

Claim 4. follows by the observation that by claim 2. of (\ref{eq:thproperty}),
on the one hand $h_{N}\leq1$ on $\mathbb{Z}^{d}$ and $h_{N}$ is
supported on $U^{N}$, and on the other hand $h_{N}=1$ on $(NK^{\delta})\cap\mathbb{Z}^{d}$. 

Claim 5. follows as a consequence of claim 4. and the definition of
$T_{N}$, see (\ref{eq:TNdef}).\end{proof}
\begin{rem}
\label{integer}With Lemma \ref{lem:hNproperty} we reveal the reason
for choosing $R$ to be an integer: because we wish that the lattice
points are not too close to the boundary of $B_{NR}$ (see (\ref{eq:integerreason})).
This enables us to show, for example, that $h_{N}$ is not too small
on $U^{N}$, as in claim 1. of (\ref{eq:hNproperty}).
\end{rem}
\vspace{0.3cm}

Now, we introduce a non-negative martingale that plays an important
role in our construction of the tilted random walk. Given a real-valued
function $g$ on $\mathbb{Z}^{d}$, we denote its discrete Laplacian
by 
\begin{equation}
\Delta_{{\rm dis}}g(x)=\frac{1}{2d}\sum_{|e|=1}g(x+e)-g(x).
\end{equation}
For the finitely supported non-negative $f$ defined in (\ref{eq:hN+fdef}),
for all $x$ in $U^{N}$, we introduce under the measure $P_{x}$
the stochastic process 
\begin{equation}
M_{t}=\frac{f(X_{t\wedge T_{U^{N}}})}{f(x)}e^{\int_{0}^{t\wedge T_{U^{N}}}v(X_{s})ds},\: t\geq0,\quad P_{x}\mbox{-a.s.,}\label{eq:Mtdef}
\end{equation}
where 
\begin{equation}
v=-\frac{\Delta_{{\rm dis}}f}{f}.\label{eq:vdef}
\end{equation}
We define for all $T\geq0$, a non-negative measure $\widehat{P}_{x,T}$
(on $\widehat{W}_{+}$) with density $M_{T}$ with respect to $P_{x}$,
\begin{equation}
\widehat{P}_{x,T}=M_{T}P_{x}.\label{eq:tilteddef}
\end{equation}

In the next lemma we show that $\widehat{P}_{x,T}$ is the law of
a Markov chain and identify its infinitesimal generator.
\begin{lem}
For all $x\in U^{N}$, one has
\begin{equation}
\widehat{P}_{x,T}\mbox{ is the probability measure for a Markov chain up to time }T\mbox{ on }U^{N}.\label{eq:tiltproba}
\end{equation}
Its semi-group (acting on the finite dimensional space of functions
on $U^{N}$) admits a generator given by the bounded operator:\textup{
\begin{equation}
\widetilde{L}g(x)=\frac{1}{2d}\sum_{y\in U^{N},\, y\sim x}\frac{f(y)}{f(x)}(g(y)-g(x)).\label{eq:tiltedgenerator}
\end{equation}
}\end{lem}
\begin{proof}
To prove the claims (\ref{eq:tiltproba}) and (\ref{eq:tiltedgenerator})
we first prove that 
\begin{equation}
M_{t}\mbox{ is an }({\cal F}_{t})\mbox{-martingale under }P_{x}.\label{eq:Mtmart}
\end{equation}
For $\zeta\in(0,1)$, we define $f^{(\zeta)}=f+\zeta$ and $v^{(\zeta)}=-\frac{\Delta_{{\rm dis}}f^{(\zeta)}}{f^{(\zeta)}}=-\frac{\Delta_{{\rm dis}}f}{f^{(\zeta)}}$.
We denote by $M_{t}^{(\zeta)}$, $t\geq0$, the stochastic process
similarly defined as $M_{t}$ in (\ref{eq:Mtdef}) by
\begin{equation}
M_{t}^{(\zeta)}=\frac{f^{(\zeta)}(X_{t\wedge T_{U^{N}}})}{f^{(\zeta)}(x)}e^{\int_{0}^{t\wedge T_{U^{N}}}v^{(\zeta)}(X_{s})ds},\: t\geq0,\quad P_{x}\mbox{-a.s.}
\end{equation}
By Lemma 3.2 in Chapter 4, p. 174 of \cite{EthierKurtz}, $M_{t}^{(\zeta)}$
is an $(\mathcal{F}_{t})$-martingale under $P_{x}$. Since $N$ is
fixed, $f^{(\zeta)}$ is uniformly for $\zeta\in(0,1)$ bounded from
above and below on $U^{N}$ , $v^{(\zeta)}$ is uniformly in $\zeta$
bounded on $U^{N}$ as well. Hence, for all $t\geq0$, $M_{t}^{(\zeta)}$
is bounded above uniformly for all $\zeta\in(0,1)$. Therefore the
claim (\ref{eq:Mtmart}) follows from the dominated convergence theorem
since for all $x\in U^{N}$, $P_{x}$-a.s., $\lim_{\zeta\to0}M_{t}^{(\zeta)}=M_{t}$.
To prove the claim (\ref{eq:tiltproba}), we just note that 
\begin{equation}
E_{x}[M_{T}]=M_{0}=1.
\end{equation}
Moreover for all $x$ in $U^{N}$ by claim 1. of (\ref{eq:fproperty})
\begin{equation}
f(X_{T_{U^{N}}})=0,
\end{equation}
thus $\widehat{P}_{x,T}$ vanishes on all paths which exit $U^{N}$
before $T_{N}$. Then, the claim (\ref{eq:tiltedgenerator}) follows
by Theorem 2.5, p. 61 of \cite{DemCast}.
\end{proof}
We then denote the law of the ``tilted random walk'' by 
\begin{equation}
\widetilde{P}_{N}=\widehat{P}_{0,T_{N}}.\label{eq:PNdef}
\end{equation}

\begin{rem}
Intuitively speaking, $\widetilde{P}_{N}$ is the law of a tilted
random walk, which restrains itself up to time $T_{N}$ from exiting
$U^{N}$ and then, after the deterministic time $T_{N}$, continues
as the simple random walk. It is absolutely continuous with respect
to $P_{0}$. 
\end{rem}
It is convenient for us to define $\{\overline{P}_{x}\}_{x\in U^{N}}$,
a family of finite-space Markov chains on $U^{N}$, with generator
$\widetilde{L}$ defined in (\ref{eq:tiltedgenerator}). We will call
this Markov chain ``the confined walk'', since it is supported on
$\Gamma(U^{N})$ (see below (\ref{eq:filtration}) for the definition).
We will also tacitly regard it as a Markov chain on $\mathbb{Z}^{d}$,
when no ambiguity rises. We denote by $\overline{E}_{x}$ the expectation
with respect to $\overline{P}_{x}$, for all $x\in U^{N}$.

Thus the following corollary is immediate.
\begin{cor}
\label{cor:coincide}
\begin{equation}
\mbox{Up to time }T_{N},\ \widetilde{P}_{N}\mbox{ coincides with }\overline{P}_{0}.\label{eq:coincide}
\end{equation}
\end{cor}
\begin{proof}
It suffices to identify the finite time marginals of the two measures
with the help of the Markov property and (\ref{eq:tiltedgenerator}). \end{proof}
\begin{rem}
Since the confined walk is time-homogenous, in Sections 3, 4 and 5
we will actually perform the analysis on the confined walk instead
of the tilted walk, and transfer the result concerning the time period
$[0,T_{N})$ back to the the tilted walk thanks to the above corollary.
See for instance (\ref{eq:transfertoconfine}).
\end{rem}
We now state and prove some basic estimates about the confined walk.
\begin{lem}
\label{lem:pippy}One has 
\begin{equation}
\begin{cases}
1. & \mbox{The measure }\mbox{\ensuremath{\pi(x)=f^{2}(x),}\ \ensuremath{x\in U^{N}}, is a reversible measure for }\\
 & \mbox{the (irreducible) confined walk }\{\overline{P}_{x}\}_{x\in U^{N}};\\
2. & \mbox{The Dirichlet form associated with }\{\overline{P}_{x}\}_{x\in U^{N}}\mbox{ and }\pi\mbox{ is }\\
 & \overline{\mathcal{E}}(g,g)=(-\widetilde{L}g,g)_{l^{2}(\pi)}=\frac{1}{2}\sum_{x,y\in U^{N},x\sim y}\frac{f(x)f(y)}{2d}(g(x)-g(y))^{2}\mbox{ with }g:U^{N}\to\mathbb{R}^{+};\\
3. & \mbox{If }x,y\in U^{N},\ |x|\leq|y|\mbox{, then one has }h_{N}(x)\geq ch_{N}(y)\mbox{ and }\pi(x)\geq c'\pi(y);\\
4. & \mbox{For all }x\in U^{N},\ cN^{-d-4}\leq\pi(x)=f^{2}(x)\leq c'N^{-d}.
\end{cases}\label{eq:cfwalkproperty}
\end{equation}
\end{lem}
\begin{proof}
Claim 1. follows from claims 1.  and 2.  of (\ref{eq:fproperty})
and the observation that by (\ref{eq:tiltedgenerator}) $\widetilde{L}$
is self-adjoint in $l^{2}(\pi)$. Claim 2. follows from claim 1. and
(\ref{eq:tiltedgenerator}). Claim 3. follows from (\ref{eq:hN+fdef})
and claim 5. of (\ref{eq:thproperty}). Claim 4. follows from claims
1. and 4. of (\ref{eq:hNproperty}) and the definition of $f$ (see
(\ref{eq:hN+fdef})).
\end{proof}
In the next lemma we control the fluctuation of $v$ with a rough
lower bound and a more refined upper bound.
\begin{lem}
\label{lem:vbound}One has (recall $v$ is defined in (\ref{eq:vdef})),
for all $x$ in $U^{N}$,
\begin{equation}
-cN^{2}\leq v(x)\leq c'N^{-2}.\label{eq:vub}
\end{equation}
\end{lem}
\begin{proof}
We first record an identity for later use: 
\begin{equation}
v(x)\overset{(\ref{eq:vdef})}{=}-\frac{\Delta_{{\rm dis}}f(x)}{f(x)}\overset{(\ref{eq:hN+fdef})}{=}-\frac{\Delta_{{\rm dis}}h_{N}(x)}{h_{N}(x)}.\label{eq:vcalc-1}
\end{equation}
The inequality on the left-hand side of (\ref{eq:vub}) is very rough
and follows from
\begin{equation}
v\overset{(\ref{eq:vcalc-1})}{\underset{h_{N}\geq0}{\geq}}-\frac{\max_{x\in U^{N}}h_{N}(x)}{\min_{x\in U^{N}}h_{N}(x)}\overset{(\ref{eq:hNproperty})\,1.}{\geq}-cN^{2}.
\end{equation}
Next we prove the inequality on the right-hand side of (\ref{eq:vub}).
We split $U^{N}$ into three parts and call them by $I^{N}$, $O^{N}$,
and $S^{N}$ respectively. Before we go into details, we describe
roughly the division, and what it entails. The region $I^{N}=B_{NR/2}\cap\mathbb{Z}^{d}$
is the ``inner part'' of $U^{N}$; the region $O^{N}$ that already
appears in (\ref{eq:ONdef}) is the ``outer part'' of $U^{N}$ that
does not feel the push of the ``hard'' boundary, that is, all neighbours
of its points belong to $B_{NR}$; the region $S^{N}=\partial_{i}U^{N}\backslash O^{N}$
is a subset of the inner boundary of $U^{N}$, where all points have
a least one neighbour outside $B_{NR}\cap\mathbb{Z}^{d}$ and thus
``feel the hard push'' from outside $U^{N}$. As we will later see,
in the microscopic region that corresponds to $I^{N}$, $\widetilde{h}$
is a smooth function; in the region $O^{N}$, $h_{N}$ is at least
of order $N^{-1}$ and $\big|\Delta_{{\rm dis}}h_{N}\big|$ is at
most of order $N^{-3}$; in the region $S^{N}$, one has $\Delta_{{\rm dis}}h_{N}>0$.

We first record an estimate. Using a Taylor formula at second order
with Lagrange remainder (see Theorem 5.16, pp. 110-111 of \cite{Rudin}),
since for all $x\in U^{N}\backslash S^{N}$, all $y$ adjacent to
$x$ belongs to $B_{NR}$, we know from (\ref{eq:hN+fdef}) that 
\begin{equation}
\Delta_{\mathrm{dis}}h_{N}(x)\geq\frac{1}{N^{2}}(\frac{1}{2d}\Delta\widetilde{h}(\frac{x}{N})-cN^{-1})\mbox{ for all }x\in U^{N}\backslash S^{N}.\label{eq:DeltadishNinf}
\end{equation}

We first treat points in $I^{N}=B_{NR/2}\cap\mathbb{Z}^{d}$. On $B_{R/2}$,
we know that $\widetilde{h}\geq c$ and $\widetilde{h}$ is $C^{\infty}$
by claim 1. of (\ref{eq:thproperty}). We thus obtain that for all
$x$ in $I^{N}$,
\begin{equation}
-\frac{\Delta_{{\rm dis}}h_{N}(x)}{h_{N}(x)}\overset{(\ref{eq:hN+fdef})}{\underset{(\ref{eq:DeltadishNinf})}{\leq}}-\frac{\Delta\widetilde{h}(\frac{x}{N})-cN^{-1}}{\widetilde{h}(\frac{x}{N})N^{2}}\leq cN^{-2}.\label{eq:INub1}
\end{equation}

We then recall that $O^{N}=\big\{ U^{N}\backslash\big(\partial_{i}U^{N}\cup B_{NR/2}\big)\big\}\bigcup\big\{ x\in\partial_{i}U^{N};\,|y|=NR\mbox{ for all }y\sim x,\, y\notin U^{N}\big\}$,
as defined in (\ref{eq:ONdef}). By claim 1. of (\ref{eq:thproperty}),
we know that for all $x\in O^{N}$, $\Delta\widetilde{h}(\frac{x}{N})=0$.
Hence we find that 
\begin{equation}
v(x)\overset{(\ref{eq:DeltadishNinf})}{\leq}-\frac{\Delta\widetilde{h}(\frac{x}{N})-cN^{-1}}{h_{N}(x)N^{2}}\overset{(\ref{eq:hNproperty})\,3.}{\leq}\frac{cN^{-1}}{c'N^{-1}\cdot N^{2}}=c''N^{-2}\mbox{ for all }x\in O^{N}.\label{eq:ONub-1}
\end{equation}

We finally treat points in $S^{N}=\partial_{i}U^{N}\backslash O^{N}$.
By Lemma 6.37, p. 136 of \cite{GilTru}, $\widetilde{h}$ can be extended
to a $C^{3}$ function $w$ on $B_{(R+1)}$ such that $w=\widetilde{h}$
in $B_{R}$ and all the derivatives of $w$ up to order three are
uniformly bounded in $B_{(R+1)}$. Hence, we have for all $x\in S^{N}$,
\begin{equation}
-\Delta_{{\rm dis}}h_{N}(x)=\big(w(\frac{x}{N})-\frac{1}{2d}\sum_{y\sim x}w(\frac{y}{N})\big)+\frac{1}{2d}\sum_{y\sim x,y\notin U^{N}}(w(\frac{y}{N})-\widetilde{h}(\frac{y}{N})))={\rm I}+{\rm II}.\label{eq:SNcalc1}
\end{equation}
On the one hand, by a second-order Taylor expansion with Lagrange
remainder, and since $\Delta w=0$ in $B_{R}\backslash B_{R/2}$,
we have 
\begin{equation}
{\rm I}\leq\frac{1}{N^{2}}(\frac{1}{2d}\Delta w(\frac{x}{N})+\frac{c}{N})=\frac{c'}{N^{3}}\ \mbox{ for }x\in S^{N}.\label{eq:SNcalc2}
\end{equation}
On the other hand, we know that by claim 2. of (\ref{eq:thproperty})
\begin{equation}
\widetilde{h}(\frac{y}{N})=0\mbox{ for all }y\notin U^{N}.\label{eq:outside0}
\end{equation}
Moreover, by definition of $S^{N}$, there exists a point $y$ in
$\mathbb{Z}^{d}$, adjacent to $x$, such that $NR<|y|\leq NR+1$.
This implies that 
\begin{equation}
(NR+1)^{2}\geq|y|^{2}\geq(NR)^{2}+1\mbox{, and hence }R+\frac{1}{N}\geq\frac{|y|}{N}\geq R+c'N^{-2}.
\end{equation}
By claim 4. of (\ref{eq:thproperty}), since $\widetilde{h}$ is bounded
from above and below by two functions having (constant) negative outer
normal derivatives on $\partial B_{R}$, we find that 
\begin{equation}
\frac{\partial\widetilde{h}}{\partial n}(z)<-c\mbox{ uniformly for all }z\in\partial B_{R},\label{eq:outer}
\end{equation}
where $\frac{\partial\widetilde{h}}{\partial n}$ denotes the outer
normal derivative of $\widetilde{h}$. Thus we find that for large
$N$, 
\begin{equation}
w(\frac{y}{N})\leq-\overline{c}N^{-2}.
\end{equation}
This implies that
\begin{equation}
{\rm II}\overset{(\ref{eq:outside0})}{=}\frac{1}{2d}\sum_{y\sim x,y\notin U^{N}}w(\frac{y}{N})\leq-c''N^{-2}.\label{eq:SNcalc3}
\end{equation}
Combining (\ref{eq:SNcalc2}) and (\ref{eq:SNcalc3}), it follows
that for large $N$ and all $x\in S^{N}$, 
\begin{equation}
v(x)\overset{(\ref{eq:vcalc-1})}{=}-\frac{\Delta_{{\rm dis}}h_{N}(x)}{h_{N}(x)}\overset{(\ref{eq:SNcalc1}),(\ref{eq:SNcalc2})}{\underset{(\ref{eq:SNcalc3})}{\leq}}\frac{cN^{-3}-c''N^{-2}}{h_{N}(x)}<0.\label{eq:SNub}
\end{equation}

Since $I^{N}$, $O^{N}$ and $S^{N}$ form a partition of $U^{N}$,
the inequality in the right-hand side of (\ref{eq:vub}) follows by
collecting (\ref{eq:INub1}), (\ref{eq:ONub-1}) and (\ref{eq:SNub}).
\end{proof}
\vspace{0.3cm}

We will now derive a lower bound for the spectral gap of the confined
walk, which we denote by $\overline{\lambda}$. We use the method
introduced at the end of Section 1 and derive an upper bound for the
quantity $A$ (recall that $A$ is defined in (\ref{eq:Adef})). However,
we first need to specify our choice of paths $\gamma.$ For $x=(x_{1},\ldots,x_{d}),\, y=(y_{1},\ldots,y_{d})\in U^{N}$,
we assume, without loss of generality, that for some $l\in\{0,\ldots,d\}$
we have 
\begin{equation}
\begin{cases}
|x_{i}|\geq|y_{i}| & \mbox{ for }i=1,\ldots,l\\
|x_{i}|<|y_{i}| & \mbox{ for }i=l+1,\ldots,d,
\end{cases}\label{eq:xyrelation}
\end{equation}
($l=0$ means that $|x_{i}|<|y_{i}|$ for all $i=1,\ldots,d$, and
$l=d$ means that $|x_{i}|\geq|y_{i}|$ for all $i=1,\ldots d$).
For $p,q\in\mathbb{Z}^{d}$, which differ only in one coordinate,
we denote by $\beta(p,q)$ the straight (and shortest) path between
them. Then $\gamma(x,y)$ is defined as follows: 
\begin{equation}
\begin{array}{c}
\gamma(x,y)=\mbox{ the concatenation of the paths}\\
\beta\big((y_{1},\ldots y_{i-1},x_{i},\ldots,x_{d}),(y_{1},\ldots y_{i},x_{i+1},\ldots,x_{d})\big)\mbox{ as }i\mbox{ goes from }1\mbox{ to }d.
\end{array}\label{eq:gammadef}
\end{equation}
Loosely speaking, $\gamma(x,y)$ successively ``adjusts'' each coordinate
of $x$ with the corresponding coordinate of $y$ by first ``decreasing''
the coordinates where $|x_{i}|$ is bigger or equal to $|y_{i}|$
and then ``increasing'' the coordinates where $|y_{i}|$ is bigger
than $|x_{i}|$. It is easy to check that this path lies entirely
in $U^{N}$, since by (\ref{eq:xyrelation}), for all $\{p,q\}\in\gamma(x,y)$, one has 
\begin{equation}
\max(|p|,|q|)\leq\max(|x|,|y|).\label{eq:maxpqxy}
\end{equation}

\begin{prop}
\label{prop:spectralgap}One has, 
\begin{equation}
\overline{\lambda}\geq cN^{-2}.\label{eq:lbsg}
\end{equation}
\end{prop}
\begin{proof}
Recall that the quantity $A=\max_{e\in E}\{\frac{1}{W(e)}\sum_{x,y,\gamma(x,y)\ni e}\mathrm{leng}(\gamma(x,y))\pi(x)\pi(y)\}$
is defined in (\ref{eq:Adef}). By (\ref{eq:pathmethod}), to prove
(\ref{eq:lbsg}), it suffices to prove that 
\begin{equation}
A\leq c'N^{2}.\label{eq:Aub}
\end{equation}
On the one hand, by (\ref{eq:maxpqxy}) and claim 3. of (\ref{eq:cfwalkproperty})
one obtains that 
\begin{equation}
\min(\pi(p),\pi(q))\geq c\min(\pi(x),\pi(y)).\label{eq:zwxy}
\end{equation}
This implies that 
\begin{equation}
W(\{p,q\})\overset{(\ref{eq:tiltedgenerator})}{\underset{(\ref{eq:edgedef})}{=}}\pi(p)\frac{f(q)}{2df(p)}\overset{(\ref{eq:cfwalkproperty})\,1.}{=}\frac{1}{2d}f(p)f(q)\overset{(\ref{eq:zwxy})}{=}\frac{1}{2d}\sqrt{\pi(p)\pi(q)}\geq c'\min(\pi(x),\pi(y)).\label{eq:sg1}
\end{equation}
On the other hand, for any $x,y\in U^{N}$, one has 
\begin{equation}
\mathrm{leng}(\gamma(x,y))\leq cN.\label{eq:sg2}
\end{equation}
Now we estimate the maximal possible number of paths that could pass
through a certain edge. We claim that, for any edge $e\in E^{N}$,
where we denote by $E^{N}$ the edge set of $U^{N}$ consisting of
unordered pairs of neighbouring vertices in $U^{N}$:
\begin{equation}
E^{N}=\big\{\{x,y\};\; x,y\in U^{N},\,|x-y|=1\big\},
\end{equation}
there are at most $cN^{d+1}$ paths passing through $e$. We now prove
the claim. To fix a pair of $\{x,y\}$ such that $e=\{(a_{1},\ldots,a_{k},\ldots,a_{d}),(a_{1},\ldots,a_{k}+1,\ldots,a_{d})\}$
belongs to $\gamma(x,y)$, where $k\in\{1,\ldots d\}$, there are
$2d$ coordinates to be chosen. Actually, for $i=1,\ldots,k-1,k+1,\ldots,d$
the $i$-th coordinate of either $x$ or $y$ must be $a_{i}$. This
leaves us at most $2^{d-1}$ ways of choosing $(d-1)$ coordinates
of $x$ and $y$ to be fixed by $a_{1},\cdots a_{k-1},a_{k+1},\cdots a_{d}$.
For the other $(d+1)$ coordinates, we have no more than $cN$ choices
for each of them, since both $x$ and $y$ must lie in $U^{N}$. This
implies that there are no more that $c'N^{d+1}$ pairs of $\{x,y\}\subset U^{N}$,
such that $e\in\gamma(x,y)$ is possible.

Combining the argument in the paragraph above with (\ref{eq:sg1})
and (\ref{eq:sg2}), one has
\begin{align}
\begin{split}A & \overset{(\ref{eq:Adef})}{=}\max_{e\in E^{N}}\frac{1}{W(e)}\sum_{x,y,\gamma(x,y)\ni e}\mathrm{leng}(\gamma(x,y))\pi(x)\pi(y)\\
 & \overset{(\ref{eq:sg1})}{\underset{(\ref{eq:sg2})}{\leq}}\max_{e\in E^{N}}\sum_{x,y,\gamma(x,y)\ni e}c'N\cdot\max(\pi(x),\pi(y))\overset{(\ref{eq:cfwalkproperty})\,4.}{\leq}c''N^{d+1}\cdot N\cdot N^{-d}=c''N^{2}.
\end{split}
\end{align}
This proves (\ref{eq:Aub}) and hence (\ref{eq:lbsg}).
\end{proof}
We then define for $\{\overline{P}_{x}\}_{x\in U^{N}}$ the regeneration
time 
\begin{equation}
\overline{t_{*}}=N^{2}\log^{2}N.\label{eq:regen}
\end{equation}
In view of above proposition, $\overline{t_{*}}$ is much larger than
the relaxation time $1/\overline{\lambda}$, which is of order $O(N^{2})$.
Hence, for all $x$ in $U^{N}$, $\overline{P}_{x}[X_{t}=\cdot]$
becomes very close to the stationary distribution $\pi$, when $t\geq\overline{t_{*}}$.
More precisely, by (\ref{eq:spectralgap2}) and (\ref{eq:lbsg}) 
\begin{equation}
\sup_{x,y\in U^{N}}|\overline{P}_{x}[X_{t}=y]-\pi(y)|\leq\sup_{x,y\in U^{N}}\sqrt{\frac{\pi(y)}{\pi(x)}}e^{-\overline{\lambda}t}\overset{(\ref{eq:cfwalkproperty})\,4.}{\underset{(\ref{eq:lbsg}),(\ref{eq:regen})}{\leq}}e^{-c\log^{2}N}\mbox{ for all }t\geq\overline{t_{*}}.\label{eq:relaxation}
\end{equation}
\vspace{0.3cm}

We now relate the relative entropy between $\widetilde{P}_{N}$ (which
tacitly depends on $R$, $\eta$, $\delta$ and $\epsilon$) and $P_{0}$
to the Dirichlet form of $h_{N}$ and derive an asymptotic upper bound
for it by successively letting $N\to\infty$, $\eta\to0$, $R\to\infty$,
$\delta\to0$ and $\epsilon\to0$ in the following Propositions \ref{prop:entropyub}
and \ref{prop:limsup}. The Brownian capacity of $K$ will appear
as the limit in the above sense of the properly scaled Dirichlet form
of $h_{N}$. 
\begin{prop}
\textup{\label{prop:entropyub}One has
\begin{equation}
H(\widetilde{P}_{N}|P_{0})\leq u_{**}(1+\epsilon)\mathcal{E}_{\mathbb{Z}^{d}}(h_{N},h_{N})+o(N^{d-2}).\label{eq:entropyub}
\end{equation}
}\end{prop}
\begin{proof}
By definition of the relative entropy (see (\ref{eq:relentrodef})),
we have
\begin{eqnarray}
 &  & H(\widetilde{P}_{N}|P_{0})\text{\ensuremath{\overset{(\ref{eq:relentrodef})}{=}}}E^{\widetilde{P}_{N}}[\log\frac{d\widetilde{P}_{N}}{dP_{0}}]\overset{(\ref{eq:tilteddef})}{\underset{(\ref{eq:PNdef})}{=}}E^{\widetilde{P}_{N}}[\log M_{T_{N}}]\overset{(\ref{eq:coincide})}{=}\overline{E}_{0}[\log M_{T_{N}}]\nonumber \\
 & \overset{(\ref{eq:Mtdef})}{=} & \overline{E}_{0}[\int_{0}^{T_{N}}v(X_{s})ds+\log f(X_{T_{N}})-\log f(X_{0})]\label{eq:entropysplit}\\
 & = & \overline{E}_{0}[\int_{0}^{\overline{t_{*}}}v(X_{s})ds]+\overline{E}_{0}[\int_{\overline{t_{*}}}^{T_{N}}v(X_{s})ds]+\overline{E}_{0}[\log f(X_{T_{N}})-\log f(X_{0})]={\rm I}+{\rm II}+{\rm III}.\nonumber 
\end{eqnarray}

For an upper bound of I, by (\ref{eq:vub}) and the definition of
$\overline{t_{*}}$ (see (\ref{eq:regen})), we have 
\begin{equation}
{\rm I}\leq\overline{t_{*}}\max_{x\in U^{N}}v(x)\leq c\log^{2}N.\label{eq:esti1}
\end{equation}

For an upper bound of II, we notice that applying (\ref{eq:relaxation})
for $t\in(\overline{t_{*}},T_{N})$, 
\begin{eqnarray}
 &  & |\overline{E}_{0}[\int_{\overline{t_{*}}}^{T_{N}}v(X_{t})dt]-(T_{N}-\overline{t_{*}})\int vd\pi|\nonumber \\
 & \leq & (T_{N}-\overline{t_{*}})\sup_{t\in[\overline{t_{*}},T_{N}]}\sup_{y\in U^{N}}\Big|\overline{P}_{0}[X_{t}=y]-\int vd\pi\Big|\cdot\max_{y\in U^{N}}|v(y)|\label{eq:IIcalc}\\
 & \overset{(\ref{eq:relaxation})}{\leq} & e{}^{-c\log^{2}N}(T_{N}-\overline{t_{*}})\max_{y\in U^{N}}|v(y)|\overset{(\ref{eq:vub})}{\underset{(\ref{eq:hNproperty})\,5.}{\leq}}e^{-c'\log^{2}N}.\nonumber 
\end{eqnarray}
Since $f$ is supported on $U^{N}$ by claim 1. of $(\ref{eq:fproperty})$,
 we may enlarge the range for summation in the second equality in
the following calculation without changing the sum and see that
\begin{align}
\begin{split}\int vd\pi & \overset{(\ref{eq:vdef})}{\underset{(\ref{eq:cfwalkproperty})\,1.}{=}}\sum_{x\in U^{N}}\frac{-\Delta_{{\rm dis}}f(x)}{f(x)}f^{2}(x)\overset{(\ref{eq:TNdef})}{=}\frac{u_{**}(1+\epsilon)}{T_{N}}\sum_{x\in\mathbb{Z}^{d}}-f(x)\Delta_{{\rm dis}}f(x)||h_{N}||_{2}^{2}\\
 &\;\, \overset{(\ref{eq:hN+fdef})}{=}\,\frac{u_{**}(1+\epsilon)}{T_{N}}\sum_{x\in\mathbb{Z}^{d}}-h_{N}(x)\Delta_{{\rm dis}}h_{N}(x).
\end{split}
\label{eq:intvdpi}
\end{align}
By the discrete Green-Gauss theorem and the definition of Dirichlet form, we have 
\begin{equation}
\sum_{x\in\mathbb{Z}^{d}}-h_{N}(x)\Delta_{{\rm dis}}h_{N}(x)=\frac{1}{2}\sum_{\overset{x,x'\in\mathbb{Z}^{d}}{x\sim x'}}\frac{1}{2d}(h_{N}(x')-h_{N}(x))^{2}=\mathcal{E}_{\mathbb{Z}^{d}}(h_{N},h_{N}).\label{eq:redactionDirichlet}
\end{equation}
Hence by (\ref{eq:intvdpi}) and (\ref{eq:redactionDirichlet}) we
know that 
\begin{equation}
(T_{N}-\overline{t_{*}})\int vd\pi\leq u_{**}(1+\epsilon)\mathcal{E}_{\mathbb{Z}^{d}}(h_{N},h_{N}).\label{eq:redactionDirichlet2}
\end{equation}
Thus, we obtain from (\ref{eq:redactionDirichlet2}) and (\ref{eq:IIcalc})
that 
\begin{equation}
{\rm II}\leq u_{**}(1+\epsilon)\mathcal{E}_{\mathbb{Z}^{d}}(h_{N},h_{N})+e^{-c'\log^{2}N}.\label{eq:esti2}
\end{equation}

For the calculation of III, we know that 
\begin{equation}
\overline{E}_{0}[\log f(X_{T_{N}})-\log f(X_{0})]\leq\log\max_{x\in U^{N}}f(x)-\log\min_{x\in U^{N}}f(x)\overset{(\ref{eq:cfwalkproperty})\,4.}{\leq}c\log N.\label{eq:esti3}
\end{equation}

Combining (\ref{eq:esti1}), (\ref{eq:esti2}) and (\ref{eq:esti3}),
we obtain that 
\begin{equation}
H(\widetilde{P}_{N}|P_{0})\leq u_{**}(1+\epsilon)\mathcal{E}_{\mathbb{Z}^{d}}(h_{N},h_{N})+o(N^{d-2}),\label{eq:EntroDiri}
\end{equation}
which is (\ref{eq:entropyub}).
\end{proof}

\begin{prop}
\label{prop:limsup}One has,
\begin{equation}
\limsup_{\epsilon\to0}\limsup_{\delta\to0}\limsup_{R\to\infty}\limsup_{\eta\to0}\limsup_{N\to\infty}\frac{1}{N^{d-2}}H(\widetilde{P}_{N}|P_{0})\leq\frac{u_{**}}{d}\mathrm{cap}_{\mathbb{R}^{d}}(K).\label{eq:limsup}
\end{equation}
\end{prop}
\begin{proof}
By (\ref{eq:entropyub}), we have 
\begin{equation}
\limsup_{N\to\infty}\frac{1}{N^{d-2}}H(\widetilde{P}_{N}|P_{0})\leq u_{**}(1+\epsilon)\limsup_{N\to\infty}\frac{1}{N^{d-2}}\mathcal{E}_{\mathbb{Z}^{d}}(h_{N},h_{N}).
\end{equation}
By the definition of $h_{N}$, we have 
\begin{equation}
\frac{1}{N^{d-2}}\mathcal{E}_{\mathbb{Z}^{d}}(h_{N},h_{N})=\frac{1}{4dN^{d-2}}\sum_{x\sim y\in\mathbb{Z}^{d}}(h_{N}(y)-h_{N}(x))^{2}\overset{(\ref{eq:hN+fdef})}{=}\frac{1}{4dN^{d-2}}\sum_{x\sim y\in\mathbb{Z}^{d}}(\widetilde{h}(\frac{y}{N})-\widetilde{h}(\frac{x}{N}))^{2}.\label{eq:Dirichletcalc-1}
\end{equation}
By claim 2.~of (\ref{eq:thproperty}), the summation in the right
member of the second equality in (\ref{eq:Dirichletcalc-1}) can be
reduced to $x,y\in U^{N}\cup\partial U^{N}$. Then, we split the sum
into two parts: 
\begin{equation}
\sum_{x\sim y\in\mathbb{Z}^{d}}(\widetilde{h}(\frac{y}{N})-\widetilde{h}(\frac{x}{N}))^{2}=\Sigma_{1}+\Sigma_{2}
\end{equation}
where 
\begin{equation}
\Sigma_{1}=\sum_{x,y\in U^{N},\, x\sim y}(\widetilde{h}(\frac{y}{N})-\widetilde{h}(\frac{x}{N}))^{2}
\end{equation}
contains all summands with $x,y\in U^{N}$, and 
\begin{equation}
\Sigma_{2}=2\sum_{x\in U^{N},\, y\notin U^{N},\, x\sim y}(h_{N}(y)-h_{N}(x))^{2}
\end{equation}
contains all summands with $x$ in $U^{N}$ and $y$ in $\partial U^{N}$.
By claim 2. of (\ref{eq:thproperty}) we find that 
\begin{equation}
\lim_{N\to\infty}\frac{1}{4dN^{d-2}}\Sigma_{1}=\frac{1}{2d}\int_{\mathbb{R}^{d}}|\nabla\widetilde{h}(y)|^{2}dy
\end{equation}
by a Riemann sum argument. While by claim 2. of (\ref{eq:hNproperty}),
we obtain that 
\begin{equation}
\Sigma_{2}\leq c\sum_{x\in\partial_{i}U^{N}}h_{N}(x){}^{2}\leq c'N^{d-1}(\frac{c}{N})^{2}=c''N^{d-3}.
\end{equation}
This implies that 
\begin{equation}
\lim_{N\to\infty}\frac{1}{N^{d-2}}\Sigma_{2}=0.
\end{equation}
Therefore, we have 
\begin{equation}
\limsup_{N\to\infty}\frac{1}{N^{d-2}}\mathcal{E}_{\mathbb{Z}^{d}}(h_{N},h_{N})\leq\lim_{N\to\infty}\frac{1}{4dN^{d-2}}(\Sigma_{1}+\Sigma_{2})=\frac{1}{2d}\int_{\mathbb{R}^{d}}|\nabla\widetilde{h}(y)|^{2}dy=\frac{1}{d}\mathcal{E}_{\mathbb{R}^{d}}(\widetilde{h},\widetilde{h}).
\end{equation}
Therefore, by claim 3. of (\ref{eq:thproperty}) we see that
\begin{equation}
\limsup_{\eta\to0}\limsup_{N\to\infty}\frac{1}{N^{d-2}}H(\widetilde{P}_{N}|P_{0})\leq\limsup_{\eta\to0}\frac{u_{**}(1+\epsilon)}{d}\mathcal{E}_{\mathbb{R}^{d}}(\widetilde{h},\widetilde{h})\leq\frac{u_{**}(1+\epsilon)}{d}\mathrm{cap}_{\mathbb{R}^{d},B_{(R)}}(K^{2\delta}).
\end{equation}
As $R\to\infty$, the relative capacity converges to the usual Brownian
capacity (this follows for instance from the variational characterization
of the capacity in Theorem 2.1.5 on pp. 71 and 72 of \cite{Fukushima}):
\begin{equation}
\mathrm{cap}_{\mathbb{R}^{d},B_{(R)}}(K^{2\delta})\to\mathrm{cap}_{\mathbb{R}^{d}}(K^{2\delta})\textrm{ as }R\to\infty.
\end{equation}
Then, letting $\delta\to0$, by Proposition 1.13, p. 60 of \cite{Portstone},
we have
\begin{equation}
\mathrm{cap}_{\mathbb{R}^{d}}(K^{2\delta})\to\mathrm{cap}_{\mathbb{R}^{d}}(K)\textrm{, as }\delta\to0.
\end{equation}
Finally by letting $\epsilon\to0$ the claim then follows. \end{proof}
\begin{rem}
In this section, guided by the the heuristic strategy described below
(\ref{eq:conjecture}), we have constructed the tilted random walk.
In essence, this continuous-time walk spends up to $T_{N}$, chosen
in (\ref{eq:TNdef}), at each point $x\in K_{N}^{\delta}$ an expected
time equal to $u_{**}(1+\epsilon)h_{N}^{2}(x)=u_{**}(1+\epsilon)$,
when started with the stationary measure $\pi$ of the confined walk.
The low entropic cost of the tilted walk with respect to the simple
random walk is quantified by the above Proposition \ref{prop:limsup}.
We will now see in the subsequent sections that in the vicinity of
points of $K_{N}^{\delta},$ the geometric trace left by the tilted
walk by time $T_{N}$ stochastically donimates random interlacements
at a level ``close to $u_{**}(1+\epsilon)$''.
\end{rem}

\section{Hitting time estimates}

In this section, we relate the entrance time (of the confined walk)
into mesoscopic boxes inside $K_{N}^{\delta}$ to the capacity of
these boxes and $T_{N}$ (see (\ref{eq:TNdef})) and establish a pair
of asymptotically matching bounds in the Propositions \ref{prop:entranceinvup}
and \ref{prop:entranceinvlb}. It is a key ingredient for the construction
of couplings in Section 5. The arguments in this section are similar
to those in Section 3 and the Appendix of \cite{TeiWin}. However,
in our set-up, special care is needed due to the fact that the stationary
measure is massively non-uniform. In this section, the constants tacitly
depend on $\delta$, $\eta$, $\epsilon$ and $R$ (see (\ref{eq:deltadef})
and (\ref{eq:Rr})), $r_{1}$, $r_{2}$, $r_{3}$, $r_{4}$ and $r_{5}$
(see (\ref{eq:rchoice})). 

We start with the precise definition of objects of interest for the
current and the subsequent sections. We denote by $\Gamma^{N}=\partial K_{N}^{\delta/2}$
the boundary in $\mathbb{Z}^{d}$ of the discrete blow-up of $K^{\delta/2}$
(we recall (\ref{eq:boundarydef}) and (\ref{eq:blowupdef-1}) for
the definition of the boundary and of the discrete blow-up). The above
$\Gamma^{N}$ will serve as a set ``surrounding'' $K_{N}$. We choose
real numbers 
\begin{equation}
0<r_{1}<r_{2}<r_{3}<r_{4}<r_{5}<1.\label{eq:rchoice}
\end{equation}

We define for $x_{0}$ in $\Gamma^{N}$ six boxes centered at $x_{0}$
(when there is ambiguity we add a superscript for their center $x_{0}$):
\begin{equation}
A_{i}=B_{\infty}(x_{0},\lfloor N^{r_{i}}\rfloor),\ 1\leq i\leq5,\mbox{ and }A_{6}=B_{\infty}(x_{0},\lfloor\frac{\delta}{100}N\rfloor),\label{eq:boxdef}
\end{equation}
and we tacitly assume that $N$ is sufficiently large so that for
all $x_{0}\in\Gamma^{N}$, the following inclusions hold: 
\begin{equation}
A_{1}\subset A_{2}\subset A_{3}\subset A_{4}\subset A_{5}\subset A_{6}\subset B_{N}^{\delta}\subset\subset\mathbb{Z}^{d}.\label{3.4}
\end{equation}
In view of (\ref{3.4}) and claim 3. of (\ref{eq:fproperty}) we find
that, by (\ref{eq:tiltedgenerator}), for large $N$ and all $x$
in $U^{N}$ 
\begin{equation}
\mbox{ the stopped processes \ensuremath{X_{\cdot\wedge T_{A_{6}}}}under \ensuremath{P_{x}}\mbox{ and }\ensuremath{\overline{P}_{x}}\mbox{ have the same law.}}\label{eq:coincidestopped}
\end{equation}

\begin{rem}
Recall that the regeneration time $\overline{t_{*}}$ is defined in
(\ref{eq:regen}) as $\overline{t_{*}}=N^{2}\log^{2}N$, and for all
$k=1,\ldots,5$, $A_{k}$ are mesoscopic objects of size $O(N^{r})$
where $r\in(0,1)$. Informally, Propositions \ref{prop:entranceinvup}
and \ref{prop:entranceinvlb} will imply that for all $x$ ``far away''
from $A_{k}$, with a high $\overline{P}_{x}$-probability, 
\begin{equation}
T_{N}\gg H_{A_{k}}\gg\overline{t_{*}}.
\end{equation}

Given any $x_{0}$ in $\Gamma^{N}$, we write
\begin{equation}
D=U^{N}\backslash A_{2},\label{eq:Ddef}
\end{equation}
and let
\begin{equation}
g(x)=\overline{P}_{x}[H_{A_{1}}\leq T_{A_{2}}]\overset{(\ref{eq:coincidestopped})}{=}P_{x}[H_{A_{1}}\leq T_{A_{2}}],\ x\in U^{N},\label{eq:gdef}
\end{equation}
be the (tilted) potential function of $A_{1}$ relative to $A_{2}$.
We also let 
\begin{equation}
f_{A_{1}}(x)=1-\frac{\overline{E}_{x}[H_{A_{1}}]}{\overline{E}_{\pi}[H_{A_{1}}]}\label{eq:fA1def}
\end{equation}
be the centered fluctuation of the scaled expected entrance time of
$A_{1}$ (relative to the stationary measure). 

The following lemma shows that the inverse of $\overline{E}_{\pi}[H_{A_{1}}]$
is closely related to $\overline{\mathcal{E}}(g,g)$. (Actually we
are going to prove that they are approximately equal later in this
section, see Propositions \ref{prop:entranceinvlb} and \ref{prop:entranceinvup},
as well as Remark \ref{almost there}.) \end{rem}
\begin{lem}
\label{lem:entranceinvest}One has,
\begin{equation}
\overline{\mathcal{E}}(g,g)(1-2\sup_{x\in D}|f_{A_{1}}(x)|)\leq\frac{1}{\overline{E}_{\pi}[H_{A_{1}}]}\leq\overline{\mathcal{E}}(g,g)\frac{1}{\pi(D)^{2}}.\label{eq:entranceinvest}
\end{equation}

\end{lem}
The proof is omitted due to its similarity to the proof of Lemma 3.2
in \cite{CerTeiWin} (which further calls Proposition 3.41 in \cite{AF},
which is originally intended for Markov chain with constant jump rate).
\vspace{0.3cm}

In the next lemma we collect some properties of entrance probabilities
for later use, namely Propositions \ref{prop:entranceinvup}, \ref{prop:entranceinvlb},
\ref{prop:qsdem} and \ref{prop:enoughexcursions}. 
\begin{lem}
\label{lem:hittingtimeest}For large $N$, one has
\begin{equation}
\overline{P}_{x}[H_{A_{1}}<\overline{t_{*}}]\leq N^{-c}\quad\mbox{ for all }x\in D,\label{eq:hittingtimebig}
\end{equation}
and similarly
\begin{equation}
\overline{P}_{x}[H_{A_{2}}<\overline{t_{*}}]\leq N^{-c'}\quad\mbox{ for all }x\in U^{N}\backslash A_{3}.\label{eq:hittingtimebig1}
\end{equation}
Uniformly for all $x\in\partial_{i}A_{1}$, one has
\begin{equation}
e_{A_{1}}(x)\leq P_{x}[T_{A_{6}}<\widetilde{H}_{A_{1}}]\leq P_{x}[T_{A_{3}}<\widetilde{H}_{A_{1}}]\leq P_{x}[T_{A_{2}}<\widetilde{H}_{A_{1}}]\leq e_{A_{1}}(x)(1+N^{-c''}).\label{eq:emest}
\end{equation}
\end{lem}
\begin{proof}
We start with (\ref{eq:hittingtimebig}). First, we explain that,
to prove (\ref{eq:hittingtimebig}), it suffices to show that
\begin{equation}
\sup_{x\in\partial A_{2}}\overline{P}_{x}[H_{A_{1}}<\widetilde{t}]\leq\sup_{x\in\partial A_{2}}\overline{P}_{x}[H_{A_{1}}<\widetilde{t}-t_{\#}]+N^{-c'}\mbox{, for all }0\leq\widetilde{t}\leq\overline{t_{*}},\label{eq:induction}
\end{equation}
where we write $t_{\#}$ for $N^{2}/\log N$. Indeed, with (\ref{eq:induction}),
the claim (\ref{eq:hittingtimebig}) follows by an induction argument:
\begin{align}
\begin{split}\sup_{x\in D}\overline{P}_{x}[H_{A_{1}}<\overline{t_{*}}]\leq\sup_{x\in\partial A_{2}}\overline{P}_{x}[H_{A_{1}}<\overline{t_{*}}]\leq\sup_{x\in\partial A_{2}}\overline{P}_{x}[H_{A_{1}}<\overline{t_{*}}-t_{\#}]+N^{-c'}\leq\cdots\\
\leq\sup_{x\in\partial A_{2}}\overline{P}_{x}[H_{A_{1}}<\overline{t_{*}}-\lceil\frac{\overline{t_{*}}}{t_{\#}}\rceil t_{\#}]+\lceil\frac{\overline{t_{*}}}{t_{\#}}\rceil N^{-c'}\overset{(\ref{eq:regen})}{\leq}0+c\log^{3}N\cdot N^{-c'}\leq N^{-c''}.
\end{split}
\end{align}
Now we prove (\ref{eq:induction}). We pick $\widetilde{t}$ in $[0,\overline{t_{*}}]$.
One has 
\begin{equation}
\sup_{x\in\partial A_{2}}\overline{P}_{x}[H_{A_{1}}<\widetilde{t}]\leq\sup_{x\in\partial A_{2}}\overline{P}_{x}[H_{A_{1}}<T_{A_{6}}]+\sup_{x\in\partial A_{2}}\overline{P}_{x}[T_{A_{6}}<H_{A_{1}}<\widetilde{t}].\label{eq:tstarsep}
\end{equation}
On the one hand, by Proposition 1.5.10, p. 36 of \cite{Intersection},
one has 
\begin{equation}
\sup_{x\in\partial A_{2}}\overline{P}_{x}[H_{A_{1}}<T_{A_{6}}]\overset{(\ref{eq:coincidestopped})}{=}\sup_{x\in\partial A_{2}}P_{x}[H_{A_{1}}<T_{A_{6}}]\leq N^{-c}.\label{eq:escape1}
\end{equation}
Now we seek an upper bound for the second term in the right member
of (\ref{eq:tstarsep}). We write
\begin{equation}
\sup_{x\in\partial A_{2}}\overline{P}_{x}[T_{A_{6}}<H_{A_{1}}<\widetilde{t}]\leq\sup_{x\in\partial A_{2}}\overline{P}_{x}[t_{\#}<T_{A_{6}}<H_{A_{1}}<\widetilde{t}]+\sup_{x\in\partial A_{2}}\overline{P}_{x}[T_{A_{6}}\leq t_{\#}]={\rm I}+{\rm II}.\label{eq:rep1}
\end{equation}
To bound I we can assume that $t_{\#}<\widetilde{t}$ (otherwise ${\rm I}=0$).
Applying Markov property successively (first at time $t_{\#}$, then
at time $T_{A_{6}}$, and finally at time $H_{A_{2}}$), we find
\begin{equation}
{\rm I}\leq\sup_{y\in U^{N}}\overline{P}_{y}[T_{A_{6}}<H_{A_{1}}<\widetilde{t}-t_{\#}]\leq\sup_{x\in\partial A_{6}}\overline{P}_{x}[H_{A_{1}}<\widetilde{t}-t_{\#}]\leq\sup_{x\in\partial A_{2}}\overline{P}_{x}[H_{A_{1}}<\widetilde{t}-t_{\#}].\label{eq:telescope}
\end{equation}
Hence to prove (\ref{eq:induction}), it suffices to prove that 
\begin{equation}
{\rm II}\leq N^{-c}.\label{eq:IIbound}
\end{equation}
Recalling that $d_{\infty}(\partial A_{2},\partial A_{6})\geq cN$,
we find that
\begin{equation}
{\rm II}\overset{(\ref{eq:coincidestopped})}{=}\sup_{x\in\partial A_{2}}P_{x}[T_{A_{6}}<t_{\#}]\leq dP[T_{[-mN,mN]}\leq t_{0}],\label{eq:tooearly}
\end{equation}
where $P$ is the probability law of a one-dimensional random walk
started from $0$ (and we denote by $E$ the corresponding expectation),
$t_{0}=t_{\#}/d$, and $m=\delta/200$. We know that 
\begin{equation}
P[T_{[-mN,mN]}\leq t_{0}]=P[\max_{0\leq t\leq t_{0}}|X_{t}|\geq mN].
\end{equation}
By Doob's inequality, we have for $\lambda>0$, using symmetry 
\begin{equation}
P[\max_{0\leq t\leq t_{0}}|X_{t}|\geq mN]=2P[\max_{0\leq t\leq t_{0}}\exp(\lambda X_{t})\geq\exp(\lambda mN)]\leq\frac{2E[\exp(\lambda X_{t_{0}})]}{\exp(\lambda mN)}.\label{eq:Doob}
\end{equation}
Note that $\exp\big\{\lambda X_{t}-t(\cosh\lambda-1)\big\}$, $t\geq0$,
is a martingale under $P$, so 
\begin{equation}
E[\exp(\lambda X_{t_{0}})]=\exp\big\{ t_{0}(\cosh\lambda-1)\big\}.
\end{equation}
Hence by taking $\lambda=\frac{mN}{2t_{0}}=cN^{-1}\log^{-1}N,$ we
obtain that the right-hand term of (\ref{eq:Doob}) is bounded from
above by 
\begin{equation}
2\exp\big\{ t_{0}(\cosh\lambda-1)-\frac{m^{2}N^{2}}{2t_{0}}\big\}\leq2\exp(-c\frac{m^{2}N^{2}}{2t_{0}})\leq N^{-c'}.
\end{equation}
This implies that 
\begin{equation}
P[T_{[-mN,mN]}\leq t_{0}]\leq N^{-c}.\label{eq:tooearly2}
\end{equation}
Thus one obtains (\ref{eq:IIbound}) by collecting (\ref{eq:tooearly})
and (\ref{eq:tooearly2}). This finishes the proof of (\ref{eq:induction})
and hence of (\ref{eq:hittingtimebig}).

The claim (\ref{eq:hittingtimebig1}) follows by a similar argument.

Now we turn to (\ref{eq:emest}). All, except the rightmost inequality
of (\ref{eq:emest}), are immediate. For the rightmost inequality,
we first notice that by an estimate similar to the discussion below
(3.25) of \cite{TeiWin} we have,
\begin{equation}
P_{x}[T_{A_{2}}<\widetilde{H}_{A_{1}}<\infty]\leq N^{-c}e_{A_{1}}(x)\quad\mbox{ for all }x\in\partial_{i}A_{1}.
\end{equation}
And hence we get that 
\begin{equation}
P_{x}[T_{A_{2}}<\widetilde{H}_{A_{1}}]=P_{x}[\widetilde{H}_{A_{1}}=\infty]+P_{y}[T_{A_{2}}<\widetilde{H}_{A_{1}}<\infty]\leq(1+N^{-c})e_{A_{1}}(x)\quad\mbox{ for all }x\in\partial_{i}A_{1}.\label{eq:Dirichletub}
\end{equation}
This concludes the proof of (\ref{eq:emest}) and hence of Lemma \ref{lem:hittingtimeest}.
\end{proof}
Now we make a further calculation of the tilted Dirichlet form of
$g$ defined in (\ref{eq:gdef}).
\begin{prop}
For large $N$, one has,
\begin{equation}
\frac{\mathrm{cap}(A_{1})}{T_{N}}u_{**}(1+\epsilon)\leq\overline{\mathcal{E}}(g,g)\leq(1+N^{-c})\frac{\mathrm{cap}(A_{1})}{T_{N}}u_{**}(1+\epsilon).\label{eq:Dirichletub-1}
\end{equation}
\end{prop}
\begin{proof}
Combining the fact that $\pi=f^{2}$ (from claim 1. of (\ref{eq:cfwalkproperty})),
and the observation that $g$ is discrete harmonic in $A_{2}\backslash A_{1}$,
$g=1$ on $A_{1}$ and $g=0$ outside $A_{2}$, one has (recall that
$Z_{1}$ is the first step of the discrete chain attached to $X_{t}$,
$t\geq0$, see (\ref{eq:skeleton chain})) 
\begin{eqnarray}
\overline{\mathcal{E}}(g,g) & = & (g,-\widetilde{L}g)_{l^{2}(\pi)}=\frac{u_{**}(1+\epsilon)}{T_{N}}\sum_{y\in\partial_{i}A_{1}}g(y)(g(y)-\sum_{x\sim y}\frac{1}{2d}g(x))\nonumber \\
 & \overset{(\ref{eq:gdef})}{=} & \frac{u_{**}(1+\epsilon)}{T_{N}}\sum_{y\in\partial_{i}A_{1}}(1-\sum_{x\sim y}P_{y}[Z_{1}=x]P_{x}[H_{A_{1}}<T_{A_{2}}])\label{eq:Dirichletcalc}\\
 & \overset{\textrm{Markov}}{=} & \frac{u_{**}(1+\epsilon)}{T_{N}}\sum_{y\in\partial_{i}A_{1}}P_{y}[T_{A_{2}}<\widetilde{H}_{A_{1}}].\nonumber 
\end{eqnarray}
On the one hand, by the rightmost inequality in (\ref{eq:emest}),
one has
\begin{equation}
\sum_{y\in\partial_{i}A_{1}}P_{y}[T_{A_{2}}<\widetilde{H}_{A_{1}}]\leq(1+N^{-c})\sum_{y\in\partial_{i}A_{1}}e_{A_{1}}(y)=(1+N^{-c})\mathrm{cap}(A_{1}).\label{eq:sumub}
\end{equation}
On the other hand, one also knows that 
\begin{equation}
\mathrm{cap}(A_{1})=\sum_{y\in\partial_{i}A_{1}}e_{A_{1}}(y)\overset{(\ref{eq:emest})}{\leq}\sum_{y\in\partial_{i}A_{1}}P_{y}[T_{A_{2}}<\widetilde{H}_{A_{1}}].\label{eq:sumlb}
\end{equation}
Thus the claim (\ref{eq:Dirichletub-1}) follows by collecting (\ref{eq:Dirichletcalc}),
(\ref{eq:sumub}) and (\ref{eq:sumlb}).
\end{proof}
Next we prove the first half of the main estimate of this section,
namely the upper bound on $1/\overline{E}_{\pi}[H_{A_{1}}]$. Let
us mention that this upper bound will actually be needed in the proof
of Lemma \ref{lem:fA1lb}.
\begin{prop}
\label{prop:entranceinvup}For large $N$, one has
\begin{equation}
\frac{1}{\overline{E}_{\pi}[H_{A_{1}}]}\leq(1+N^{-c})\frac{\mathrm{cap}(A_{1})}{T_{N}}u_{**}(1+\epsilon).\label{eq:entranceinvub}
\end{equation}
As a consequence, one has 
\begin{equation}
\overline{E}_{\pi}[H_{A_{1}}]\geq cN^{2+c'}.\label{eq:entrancelb}
\end{equation}
\end{prop}
\begin{proof}
We first prove (\ref{eq:entranceinvub}). We apply the right-hand
inequality in (\ref{eq:Dirichletub-1}) to the right-hand estimate
in (\ref{eq:entranceinvest}). Note that 
\begin{equation}
\pi(D)=1-\pi(A_{2})\geq1-cN^{(r_{2}-1)d}\mbox{\quad\ for large }N,\label{eq:piA2est}
\end{equation}
for large $N$, with the help of (\ref{eq:Dirichletub-1}) we thus
find that
\begin{equation}
\frac{1}{\overline{E}_{\pi}[H_{A_{1}}]}\overset{(\ref{eq:entranceinvest})}{\leq}\frac{\overline{\mathcal{E}}(g,g)}{\pi(D)^{2}}\overset{(\ref{eq:piA2est})}{\leq}(1-cN^{(r_{2}-1)d})^{-2}\overline{\mathcal{E}}(g,g)\overset{(\ref{eq:Dirichletub-1})}{\leq}(1+N^{-c})\frac{\mathrm{cap}(A_{1})}{T_{N}}u_{**}(1+\epsilon).
\end{equation}
This yields (\ref{eq:entranceinvub}). Then the claim (\ref{eq:entrancelb})
follows by observing (\ref{eq:capacityorder}) and claim 5. of (\ref{eq:hNproperty}).
\end{proof}
In the following Lemma \ref{lem:fA1lb} and Proposition \ref{prop:entranceinvlb}
we build a corresponding lower bound by controlling the fluctuation
function $f_{A_{1}}$ defined in (\ref{eq:fA1def}).
\begin{lem}
\label{lem:fA1lb}For large $N$, one has 
\begin{equation}
f_{A_{1}}(x)\geq-N^{-c},\quad\mbox{ for all }x\in U^{N}.\label{eq:fA1lb}
\end{equation}
and in the notation of (\ref{eq:Ddef}) 
\begin{equation}
\overline{E}_{x}[H_{A_{1}}]\geq\overline{E}_{\pi}[H_{A_{1}}]-e^{-c'\log^{2}N}-\overline{P}_{x}[H_{A_{1}}\le\overline{t_{*}}](\overline{t_{*}}+\overline{E}_{\pi}[H_{A_{1}}])\quad\mbox{ for all }x\in D.\label{eq:lbcalc1}
\end{equation}
\end{lem}
\begin{proof}
As we now explain, to prove (\ref{eq:fA1lb}), it suffices to show
that 
\begin{equation}
\bigl|\overline{E}_{x}[\overline{E}_{X_{\overline{t_{*}}}}[H_{A_{1}}]]-\overline{E}_{\pi}[H_{A_{1}}]\bigr|\le e^{-c'\log^{2}N},\mbox{ for all }x\in U^{N}.\label{eq:Er2}
\end{equation}
Indeed, since $H_{A_{1}}\leq\overline{t_{*}}+H_{A_{1}}\circ\theta_{\overline{t_{*}}}$,
the simple Markov property applied at time $\overline{t_{*}}$ and
(\ref{eq:Er2}) imply that 
\begin{equation}
\sup_{x\in U^{N}}\overline{E}_{x}[H_{A_{1}}]\leq\overline{t_{*}}+e^{-c\log^{2}N}+\overline{E}_{\pi}[H_{A_{1}}].\label{Er3}
\end{equation}
It then follows that 
\begin{equation}
\begin{array}{cc}
\begin{array}{ccl}
\frac{\sup_{x\in U^{N}}\overline{E}_{x}[H_{A_{1}}]}{\overline{E}_{\pi}[H_{A_{1}}]}-1 & \overset{(\ref{Er3})}{\underset{(\ref{eq:entranceinvub})}{\leq}} & (\overline{t_{*}}+e^{-c\log^{2}N})c'\frac{\mathrm{cap}(A_{1})}{T_{N}}\\
 & \overset{(\ref{eq:capacityorder})}{\underset{(\ref{eq:TNdef})}{\leq}} & (\overline{t_{*}}+e^{-c\log^{2}N})c''N^{(d-2)r_{1}-d}\overset{(\ref{eq:regen})}{\underset{(\ref{eq:rchoice})}{\leq}}N^{-\widetilde{c}}.
\end{array}\end{array}\label{eq:tnsup}
\end{equation}
This proves (\ref{eq:fA1lb}). We now prove (\ref{eq:Er2}). Let us
consider the expectation of $H_{A_{1}}$ when started from $X_{\overline{t_{*}}}$.
We first note that for all $x\in U^{N}$, 
\begin{equation}
\begin{split}\bigl|\overline{E}_{x}[\overline{E}_{X_{\overline{t_{*}}}}[H_{A_{1}}]]-\overline{E}_{\pi}[H_{A_{1}}]\bigr|\leq\sum_{y\in U^{N}}|\overline{P}_{x}[X_{\overline{t_{*}}}=y]-\pi(y)|\sup_{y\in U^{N}}\overline{E}_{y}[H_{A_{1}}].\end{split}
\end{equation}
By the relaxation to equilibrium estimate (\ref{eq:relaxation}),
one has 
\begin{equation}
\sum_{y\in U^{N}}|\overline{P}_{x}[X_{\overline{t_{*}}}=y]-\pi(y)|\leq e^{-c\log^{2}N}\mbox{ for all }x\in U^{N}.
\end{equation}
Thus, to prove (\ref{eq:Er2}) it suffices to obtain a very crude
upper bound for the supremum of the expected entrance time in $A_{1}$
as the starting point varies in $U^{N}$: 
\begin{equation}
\overline{E}_{y}[H_{A_{1}}]\leq cN^{5+d}\quad\mbox{ for all }y\in U^{N}.\label{eq:hittingcrude-1}
\end{equation}
This follows, for example, by a corollary of the commute time identity
(see Corollary 4.28, p. 59 of \cite{Barlow}): 
\begin{equation}
\overline{E}_{y}[H_{A_{1}}]\leq r_{\mathrm{eff}}(y,A_{1})\pi(U^{N})\mbox{ for all }y\in U^{N},\label{eq:commuteiden}
\end{equation}
where $r_{{\rm eff}}(y,A_{1})$ stands for the effective resistance
between $y$ and $A_{1}$. On the one hand, by the third equality
of (\ref{eq:sg1}) and Claim 4. of (\ref{eq:cfwalkproperty}), for
all $x,y\in U^{N}$ such that $x\sim y$, we know that 
\begin{equation}
W(x,y)=\frac{1}{2d}\sqrt{\pi(x)\pi(y)}\in(cN^{-(4+d)},1],\label{eq:Wboundary}
\end{equation}
hence the resistance on $\{p,q\}$ does not exceed $cN^{4+d}$. We
know that for any $y$ in $U^{N}$, for some $x\in\partial_{i}A_{1}$,
the effective resistance between $y$ and $x$ (which we denote by
$r_{{\rm eff}}(y,x)$) is less or equal to the effective resistance
between $y$ and $x$ on the the path $\gamma(y,x)$ (which we denote
by $r_{{\rm eff}}^{\gamma}(y,x)$) defined above Proposition \ref{prop:spectralgap}
(note that $\gamma(y,x)$ is a subgraph of $U^{N}$). Since by (\ref{eq:sg2})
$\gamma(y,x)$ is of length no more than $cN$, $r_{{\rm eff}}^{\gamma}(y,x)$
does not exceed $c'N^{5+d}$ by (\ref{eq:Wboundary}). Hence, we obtain
that
\begin{equation}
r_{{\rm eff}}(y,A_{1})\leq r_{{\rm eff}}(y,x)\leq r_{{\rm eff}}^{\gamma}(y,x)\leq cN^{5+d}.\label{eq:reffub}
\end{equation}
On the other hand, one has $\pi(U^{N})=1$ (by claim 1. of (\ref{eq:cfwalkproperty})).
Thus, (\ref{eq:commuteiden}) and (\ref{eq:reffub}) yield that
\begin{equation}
\sup_{y\in U^{N}}\overline{E}_{y}[H_{A_{1}}]\leq cN^{5+d}.
\end{equation}
This completes the proof of (\ref{eq:hittingcrude-1}) and hence of
(\ref{eq:fA1lb}). 

We now turn to (\ref{eq:lbcalc1}). We consider any $x\in D$. By
the simple Markov property applied at time $\overline{t_{*}}$, we
find that 
\begin{eqnarray}
\overline{E}_{x}[H_{A_{1}}] & \geq & \overline{E}_{x}[\mathbf{1}_{\{H_{A_{1}}>\overline{t_{*}}\}}\overline{E}_{X_{\overline{t_{*}}}}[H_{A_{1}}]]=\overline{E}_{x}[\overline{E}_{X_{\overline{t_{*}}}}[H_{A_{1}}]]-E_{x}[\mathbf{1}_{\{H_{A_{1}}\le\overline{t_{*}}\}}\overline{E}_{X_{\overline{t_{*}}}}[H_{A_{1}}]]\nonumber \\
 & \stackrel{(\ref{eq:Er2})}{\geq} & \overline{E}_{\pi}[H_{A_{1}}]-e^{-c\log^{2}N}-\overline{P}_{x}[H_{A_{1}}\le\overline{t_{*}}]\sup_{y\in U^{N}}\overline{E}_{y}[H_{A_{1}}]\\
 & \stackrel{(\ref{Er3})}{\geq} & \overline{E}_{\pi}[H_{A_{1}}]-e^{-c'\log^{2}N}-\overline{P}_{x}[H_{A_{1}}\le\overline{t_{*}}](\overline{t_{*}}+\overline{E}_{\pi}[H_{A_{1}}]).\nonumber 
\end{eqnarray}
This proves (\ref{eq:lbcalc1}) and finishes Lemma \ref{lem:fA1lb}.
\end{proof}
We now prove the second main estimate.
\begin{prop}
\label{prop:entranceinvlb}For large $N$, one has that
\begin{equation}
\frac{1}{\overline{E}_{\pi}[H_{A_{1}}]}\geq(1-N^{-c})\frac{\mathrm{cap}(A_{1})}{T_{N}}u_{**}(1+\epsilon).\label{eq:entranceinvlb}
\end{equation}
\end{prop}
\begin{proof}
By applying (\ref{eq:Dirichletub-1}) and the left-hand inequality
of (\ref{eq:entranceinvest}), for large $N$, one has, 
\begin{eqnarray}
\frac{1}{\overline{E}_{\pi}[H_{A_{1}}]} & \geq & \bigl(1-2\sup_{x\in D}|f_{A_{1}}(x)|\bigr)\frac{\mathrm{cap}(A_{1})}{T_{N}}u_{**}(1+\epsilon).\label{eq:entrancelbrework}
\end{eqnarray}
Thus, with (\ref{eq:fA1lb}) in mind, to prove (\ref{eq:entranceinvlb}),
it suffices to show that for large $N$, 
\begin{equation}
\sup_{x\in D}f_{A_{1}}(x)\leq N^{-c}.\label{eq:fluctupperbd}
\end{equation}
Dividing by $\overline{E}_{\pi}[H_{A_{1}}]$ on both sides of (\ref{eq:lbcalc1})
and taking the infimum over all $x\in D$, one obtains
\begin{align}
\begin{split}\inf_{x\in D}\frac{\overline{E}_{x}[H_{A_{1}}]}{\overline{E}_{\pi}[H_{A_{1}}]} & \overset{(\ref{eq:regen})}{\geq}1-\frac{e^{-c'\log^{2}N}}{\overline{E}_{\pi}[H_{A_{1}}]}-\sup_{x\in D}\overline{P}_{x}[H_{A_{1}}\leq\overline{t_{*}}]\Bigl(\frac{N^{2}\log^{2}N}{\overline{E}_{\pi}[H_{A_{1}}]}+1\Bigr)\\
 & \overset{(\ref{eq:entrancelb})}{\underset{(\ref{eq:hittingtimebig})}{\geq}}1-e^{-\overline{c}'\log^{2}N}-N^{-\widetilde{c}'}\Bigl(c''(\log N)^{2}N^{-\widetilde{c}}+1\Bigr)\geq1-N^{-c}.
\end{split}
\end{align}
 Together with (\ref{eq:entrancelbrework}) this proves (\ref{eq:fluctupperbd})
as well as (\ref{eq:entranceinvlb}).\end{proof}
\begin{rem}
\label{almost there}The combination of Propositions \ref{prop:entranceinvup}
and \ref{prop:entranceinvlb} forms a pair of asymptotically tight
bounds on $\overline{E}_{\pi}[H_{A_{1}}]$, namely 
\begin{equation}
(1-N^{-c})\frac{\mathrm{cap}(A_{1})}{T_{N}}u_{**}(1+\epsilon)\leq\frac{1}{\overline{E}_{\pi}[H_{A_{1}}]}\leq(1+N^{-c})\frac{\mathrm{cap}(A_{1})}{T_{N}}u_{**}(1+\epsilon).
\end{equation}

\end{rem}

\section{Quasi-stationary measure}

In this section we introduce the quasi-stationary distribution (abbreviated
below as q.s.d.) induced on $D$ (recall that $D$ is defined in (\ref{eq:Ddef}))
and collect some of its properties. This will help us show in the
next section that carefully chopped sections of the confined random
walk are approximately independent, allowing us to bring into play
excursions of random walk and furthermore random interlacements. In
Proposition \ref{prop:qsdcp} we prove that the q.s.d. on $D$ is
an appropriate approximation of the stationary distribution of the
random walk conditioned to stay in $D$. In Proposition \ref{prop:qsdem}
we show that the hitting distribution of $A_{1}$ of the confined
walk starting from the q.s.d. on $D$ is very close to the normalized
equilibrium measure of $A_{1}$. In this section the constants tacitly
depend on $\delta$, $\eta$, $\epsilon$ and $R$ (see (\ref{eq:deltadef})
and (\ref{eq:Rr})), $r_{1}$, $r_{2}$, $r_{3}$, $r_{4}$ and $r_{5}$
(see (\ref{eq:rchoice})). 

We fix the choice of $A_{1}$ and $A_{2}$ as in the last section
(see (\ref{eq:boxdef})). The arguments in Lemma 4.2, Propositions
4.3 and 4.4 below are similar to those of Section 3.2 and the Appendix
of \cite{TeiWin}. However, in our set-up, special care is needed
due to the fact that the stationary measure is massively non-uniform
in the present context.

\vspace{0.3cm}

We now define the q.s.d. on $D(=U^{N}\backslash A_{2})$. We denote
by $\{H_{t}^{D}\}_{t\geq0}$ the semi-group of $\{\overline{P}_{x}\}_{x\in U^{N}}$
killed outside $D$, so that for all $f\in U^{N}\to\mathbb{R}$ 
\begin{equation}
H_{t}^{D}f(x)=\overline{E}_{x}[f(X_{t}),\, H_{A_{2}}>t].
\end{equation}
We denote by $L^{D}$ the generator of $\{H_{t}^{D}\}_{t\geq0}$.
It is classical fact that for $f:D\to\mathbb{R}$, 
\begin{equation}
L^{D}f(x)=\widetilde{L}\widetilde{f}(x)\qquad\forall x\in D,
\end{equation}
where $\widetilde{f}$ is the extension of $f$ to $U^{N}$ vanishing
outside $D$ and $\widetilde{L}$ (defined in (\ref{eq:tiltedgenerator}))
is the generator for the tilted walk. We denote by $\pi^{D}$ the
restriction of the measure $\pi$ onto $D$. So, $\{H_{t}^{D}\}_{t\geq0}$
and $L^{D}$ are self-adjoint in $l^{2}(\pi^{D})$ and 
\begin{equation}
H_{t}^{D}=e^{tL^{D}}.
\end{equation}
 We then denote by $\lambda_{i}^{D},i=1,\cdots,|D|$, with 
\begin{equation}
0\leq\lambda_{i}^{D}\leq\lambda_{i+1}^{D},\qquad i=1,\cdots,|D|-1,\label{eq:liDorder}
\end{equation}
the eigenvalues of $-L^{D}$ and by $f_{i}$, $i=1,\cdots,|D|$, an
$l^{2}(\pi^{D})$-orthonormal basis of eigenfunctions associated to
$\lambda_{i}$. Because $D$ is connected, by the Perron-Frobenius
theorem, all entries of $f_{1}$ are positive. The quasi-stationary
distribution on $D$ is the probability measure on $D$ with density
with respect to $\pi^{d}$ proportional to $f_{1}$, i.e., 
\begin{equation}
\sigma(y)=\frac{(f_{1},\delta_{y})_{l^{2}(\pi^{D})}}{(f_{1},\mathbf{1})_{l^{2}(\pi^{D})}},\qquad x\in D,\label{eq:quasidef}
\end{equation}
where, for $y\in D$, $\delta_{y}:D\to\mathbb{R}$ is the point mass
function at $y$. It is known that the q.s.d. on $D$ is the limit
distribution of the walk conditioned on never entering $A_{2}$, i.e.,
for all $x,y\in D$, one has (see (6.6.3), p. 91 of \cite{Keilson}),
\begin{equation}
\sigma(y)=\lim_{t\to\infty}\overline{P}_{x}[Y_{t}=y\big|H_{A_{2}}>t].\label{eq:qsdlim}
\end{equation}

We now prove a lemma which is useful in the proof of Proposition \ref{prop:lbsigma}
below.
\begin{lem}
\label{lem:comparison}For all $x,y\in D$, one has that
\begin{equation}
\sigma(y)\geq N^{-c}\overline{P}_{y}[H_{x}<H_{A_{2}}]\sigma(x).\label{eq:sigmacomparison}
\end{equation}
\end{lem}
\begin{proof}
By the $l^{2}(\pi^{D})$-self-adjointness of the killed semi-group
$(H_{t}^{D})_{t\geq0}$, we have that for all $x,y\in D$, $t>0$,
\begin{equation}
\overline{P}_{x}[X_{t}=y\big|H_{A_{2}}>t]=\overline{P}_{y}[X_{t}=x\big|H_{A_{2}}>t]\frac{\pi(y)}{\pi(x)}\frac{\overline{P}_{y}[H_{A_{2}}>t]}{\overline{P}_{x}[H_{A_{2}}>t]}.\label{eq:xxprimeexchange}
\end{equation}
On the one hand, by the strong Markov property applied at time $H_{x}$,
we know that for all $x,y\in D$, 
\begin{equation}
\overline{P}_{y}[H_{A_{2}}>t]\geq\overline{P}_{y}[H_{x}<H_{A_{2}}]\overline{P}_{x}[H_{A_{2}}>t],\label{eq:entrancetransfer}
\end{equation}
On the other hand, by claim 4. of (\ref{eq:cfwalkproperty}), we know
that for all $x,y\in D$, $t>0$, 
\begin{equation}
\frac{\pi(y)}{\pi(x)}\geq cN^{-4}.\label{eq:piratio}
\end{equation}
Thus, the claim (\ref{eq:sigmacomparison}) follows by taking limits
in $t$ on both sides of (\ref{eq:xxprimeexchange}) and incorporating
(\ref{eq:piratio}) and (\ref{eq:entrancetransfer}).
\end{proof}
The next lemma is also a preparation for Proposition \ref{prop:lbsigma}.
\begin{lem}
For all $x\in D\backslash A_{4}$, one has 
\begin{equation}
\max_{y\in\partial A_{3}}\overline{P}_{y}[H_{x}<H_{A_{2}}]\geq N^{-c}.\label{eq:xppentrance}
\end{equation}
\end{lem}
\begin{proof}
We fix an $x\in D\backslash A_{4}$ in the proof. Applying the Markov
property at time $T_{A_{3}}$ under $\overline{P}_{y'}$ for $y'\in\partial A_{2}$,
we see that 
\begin{equation}
\max_{y'\in\partial A_{2}}\overline{P}_{y'}[H_{x}<H_{A_{2}}]=\max_{y'\in\partial A_{2}}\overline{P}_{y'}[T_{A_{3}}<H_{x}<H_{A_{2}}]\leq\max_{y\in\partial A_{3}}\overline{P}_{y}[H_{x}<H_{A_{2}}].\label{eq:qsdlbc2}
\end{equation}
We now develop a lower bound on the left-hand side of (\ref{eq:qsdlbc2})
via effective resistance estimates. We denote by $U^{{\rm col}}$
the graph obtained by collapsing $A_{2}$ into a single vertex $a$
in $U^{N}$. With some abuse of notation we use the same symbol for
the vertices in $U^{{\rm col}}$ as in $U^{N}$ except for $a$. We
denote by $W^{{\rm col}}:U^{{\rm col}}\times U^{{\rm col}}\to\mathbb{R}^{+}$
the induced edge-weight. Let 
\begin{equation}
w_{a}=\sum_{y\in\partial A_{2}}W^{{\rm col}}(a,y)=\sum_{z\in A_{2},y\in\partial A_{2},z\sim y}W(z,y)\label{eq:wadef}
\end{equation}
be the sum of the weights of edges that touch $a$ in $U^{{\rm col}}$.
We denote by $\{P_{z}^{{\rm col}}\}_{z\in U^{{\rm col}}}$ the discrete-time
reversible Markov chain with edge-weight $W^{{\rm col}}$. The reversible
measure of this Markov chain $\pi^{{\rm col}}$ is given through 
\begin{equation}
\pi^{{\rm col}}(z)=\begin{cases}
w_{a} & z=a\\
\sum_{y\sim z}W(z,y) & \mbox{otherwise.}
\end{cases}
\end{equation}
Then we have 
\begin{equation}
\max_{y'\in\partial A_{2}}\overline{P}_{y'}[H_{x}<H_{A_{2}}]=\max_{y'\in\partial A_{2}}P_{y'}^{{\rm col}}[H_{x}<H_{a}]\geq P_{a}^{{\rm col}}[H_{x}<\widetilde{H}_{a}],\label{eq:qsdlbc3}
\end{equation}
By a classical result on electrical networks (see Proposition 3.10,
p. 69 of \cite{AF}), the escape probability in the right-hand side
of (\ref{eq:qsdlbc3}) equals 
\begin{equation}
P_{a}^{{\rm col}}[H_{x}<\widetilde{H}_{a}]=\big(w_{a}r^{{\rm col}}(a,x)\big)^{-1},\label{eq:qsdlbc35}
\end{equation}
where $r^{{\rm col}}(a,x)$ is the effective resistance between $a$
and $x$ on $U^{{\rm col}}$. We know that $r^{{\rm col}}(a,x)$ is
smaller or equal to the effective resistance between $a$ and $x$
along a path between $a$ and $x$ of length no more than $cN$ and
along this path the edge-weight is no less than $N^{-c}$ by (\ref{eq:Wboundary}).
Hence, we obtain that 
\begin{equation}
r^{{\rm col}}(a,x)\leq N^{c}.\label{eq:reffub-1}
\end{equation}
Moreover, we know that 
\begin{equation}
w_{a}=\sum_{z\in A_{2},y\in\partial A_{2},z\sim y}W(z,y)\leq N^{c}\max_{z\in D}W(z,y)\overset{(\ref{eq:Wboundary})}{\leq}N^{c}.\label{eq:waub}
\end{equation}
Therefore, we conclude from (\ref{eq:qsdlbc35}), (\ref{eq:reffub-1})
and (\ref{eq:waub}) that 
\begin{equation}
P_{a}^{{\rm col}}[H_{x}<\widetilde{H}_{a}]\geq N^{-c'}.\label{eq:qsdlbc4}
\end{equation}
The claim (\ref{eq:xppentrance}) follows by collecting (\ref{eq:qsdlbc2}),
(\ref{eq:qsdlbc3}) and (\ref{eq:qsdlbc4}). 
\end{proof}
The next proposition is a crucial estimate for us, showing that $\sigma$
is not too small at any point in $D$. This fact will be used in Proposition
\ref{prop:qsdcp}. In the proof we mainly rely on the reversibility
of the confined walk, hitting probability estimates of simple random
walk, and the Harnack principle.
\begin{prop}
\label{prop:lbsigma}For large $N$, one has the following lower bound
\begin{equation}
\inf_{x\in D}\sigma(x)\geq N^{-c},\label{eq:lowerboundsigma}
\end{equation}
and for all $x\in D$, 
\begin{equation}
N^{c'}\geq f_{1}(x)\geq N^{-c''}.\label{eq:lowerboundf}
\end{equation}
\end{prop}
\begin{proof}
We first prove (\ref{eq:lowerboundsigma}). The claim (\ref{eq:lowerboundf})
will then follow. Because $\sigma$ is a probability measure, and
\begin{equation}
\big|D\big|\leq cN^{d},\label{eq:Dsize}
\end{equation}
there must exist some $x'$ in $D$ such that 
\begin{equation}
\sigma(x')\geq cN^{-d}.\label{eq:qsmassumption}
\end{equation}
By (\ref{eq:sigmacomparison}), to prove (\ref{eq:lowerboundsigma})
it suffices to prove that for all $x\in D$, 
\begin{equation}
\overline{P}_{x}[H_{x'}<H_{A_{2}}]\geq N^{-c'}.\label{eq:goalqsd}
\end{equation}
We now prove (\ref{eq:goalqsd}).To prove (\ref{eq:lowerboundsigma}),
we distinguish between two cases according to the location of $x'$.

Case 1: $x'\in A_{4}\backslash A_{2}$ (recall the definition of $A_{4}$
in (\ref{eq:boxdef})). By (\ref{eq:coincidestopped}) and a standard
hitting estimate (see Proposition 1.5.10, p. 36 of \cite{Intersection})
for simple random walk, for all $x$ in $D$, we have that (recall
the definition of $A_{5}$ in (\ref{eq:boxdef})), 
\begin{equation}
\overline{P}_{x}[H_{\partial A_{5}}<H_{A_{2}}]\geq N^{-c},\label{eq:exitestimate}
\end{equation}
(note that the left-hand side equals $1$ if $x\notin A_{5}$). We
write 
\begin{equation}
l(x)=\overline{P}_{x}[H_{x'}<H_{A_{2}}].
\end{equation}
By the strong Markov property applied at $H_{\partial A_{5}}$,
\begin{equation}
l(x)\overset{{\rm Markov}}{\geq}\overline{P}_{x}[H_{\partial A_{5}}<H_{A_{2}}]\min_{y\in\partial A_{5}}l(y)\overset{(\ref{eq:exitestimate})}{\geq}N^{-c}\min_{y\in\partial A_{5}}l(y).\label{eq:harnackargu}
\end{equation}

We now develop a lower bound on the right-hand side of (\ref{eq:harnackargu}).
Let $S_{1}=B_{\infty}(x_{0},3N^{r_{5}})\backslash B_{\infty}(x_{0,}\frac{1}{3}N^{r_{5}})$
and $S_{2}=B_{\infty}(x_{0},2N^{r_{5}})\backslash B_{\infty}(x_{0,}\frac{1}{2}N^{r_{5}})$,
(we tacitly assume that $N$ is sufficiently large that $S_{1}\subset A_{6},$
and $S_{2}\subset D$). It is straight-forward to see that $l(x)$
is $\widetilde{L}-$harmonic in $D\backslash\{x'\}$ and that $\widetilde{L}$
coincides with $\Delta_{\mathrm{dis}}$ in $S_{1}$ (see (\ref{eq:coincidestopped})).
By the Harnack inequality (see Theorem 6.3.9, p. 131 of \cite{LawlerLimic}),
we know that (note that $\partial A_{5}\subset S_{2}$) 
\begin{equation}
\min_{y\in\partial A_{5}}l(y)\geq c'\max_{y\in\partial A_{5}}l(y).
\end{equation}
This implies by (\ref{eq:harnackargu}) that
\begin{equation}
\min_{x\in D}l(x)\geq c'N^{-c}\max_{y\in\partial A_{5}}l(y).\label{eq:exitest2}
\end{equation}
We now take any point $y'\in\partial A_{5}$ of least distance (in
the sense of $l^{\infty}-$norm) to $x'$ on $\partial A_{5}$ and
sharing $(d-1)$ common coordinates with $x'$ and fix $y'$. We set
$B=B_{\infty}(y',|y'-x'|_{\infty}-1)$. Our way of choosing $y'$
ensures that $x'\in\partial B$. Then by (\ref{eq:coincidestopped})
we have 
\begin{equation}
l(y')=\overline{P}_{y'}[H_{x'}<H_{A_{2}}]\geq\overline{P}_{y'}[X_{T_{B}}=x']\overset{(\ref{eq:coincidestopped})}{=}P_{y'}[X_{T_{B}}=x'].\label{eq:exitest3}
\end{equation}
By a classical estimate(see Lemma 6.3.7, pp. 158-159 of \cite{LawlerLimic}),
we have 
\begin{equation}
P_{y'}[X_{T_{B}}=x']\geq cN^{(1-d)r_{5}}.\label{eq:exitest4}
\end{equation}
Thus the claim (\ref{eq:goalqsd}) follows by collecting (\ref{eq:exitest2}),
(\ref{eq:exitest3}) and (\ref{eq:exitest4}).

Case 2: $x'\in D\backslash A_{4}$. Since $\partial A_{3}\subset A_{4}\backslash A_{2}$,
if we can prove that for some $y\in\partial A_{3}$, 
\begin{equation}
\sigma(y)\geq N^{-c},\label{eq:sigmax''}
\end{equation}
then we are brought back to Case 1 by taking the $y$ in (\ref{eq:sigmax''})
as the $x'$ in (\ref{eq:qsmassumption}). Now we show that we can
indeed find such $y$ that (\ref{eq:sigmax''}) holds. By (\ref{eq:sigmacomparison})
and our assumption that $\sigma(x')\geq N^{-c}$, we have
\begin{equation}
\sigma(y)\overset{(\ref{eq:sigmacomparison})}{\geq}N^{-c}\overline{P}_{y}[H_{x'}<H_{A_{2}}]\sigma(x')\overset{(\ref{eq:qsmassumption})}{\geq}N^{-c'}\overline{P}_{y}[H_{x'}<H_{A_{2}}].
\end{equation}
Hence we know that by (\ref{eq:xppentrance}), if we pick the $y$
that maximizes the probability in the left-hand side of (\ref{eq:xppentrance}),
the claim (\ref{eq:sigmax''}) is indeed true.

With these two cases we conclude the proof of (\ref{eq:lowerboundsigma}).

Now we prove (\ref{eq:lowerboundf}). By the fact that $f_{1}$ is
a unit vector in $l^{2}(\pi^{D})$ we know that 
\begin{equation}
(f_{1},f_{1})_{l^{2}(\pi^{D})}=1.\label{eq:f1norm}
\end{equation}
To prove the first inequality of (\ref{eq:lowerboundf}), we observe
that, thanks to (\ref{eq:f1norm}):
\begin{equation}
1=(f_{1},f_{1})_{l^{2}(\pi^{D})}\geq\max_{x\in D}f_{1}^{2}(x)\min_{x\in D}\pi^{D}(x)\overset{(\ref{eq:cfwalkproperty})\,4.}{\geq}N^{-c}\max_{x\in D}f_{1}^{2}(x).
\end{equation}
To prove the second inequality of (\ref{eq:lowerboundf}), we observe
that by (\ref{eq:f1norm}) 
\begin{equation}
\max_{x\in D}\pi^{D}(x)f_{1}^{2}(x)\geq\frac{1}{|D|},
\end{equation}
which implies that
\begin{equation}
\max_{x\in D}f_{1}(x)\geq\sqrt{\frac{1}{|D|\max_{x\in D}\pi^{D}(x)}}\overset{(\ref{eq:cfwalkproperty})\,4.}{\underset{(\ref{eq:Dsize})}{\geq}}N^{-c}.\label{eq:maxf1lb}
\end{equation}
This implies that for all $x\in D$,
\begin{equation}
f_{1}(x)\overset{(\ref{eq:quasidef})}{=}\frac{1}{\pi^{D}(x)}\sigma(x)(f_{1},\mathbf{1})_{l^{2}(\pi^{D})}\overset{(\ref{eq:cfwalkproperty})\,4.}{\underset{(\ref{eq:lowerboundsigma})}{\geq}}N^{-c}\max_{x\in D}f_{1}(x)\min_{x\in D}\pi^{D}(x)\overset{(\ref{eq:cfwalkproperty})\,4.}{\underset{(\ref{eq:maxf1lb})}{\geq}}N^{-c'}.
\end{equation}
This finishes the proof of (\ref{eq:lowerboundf}), and concludes
the proof of Proposition \ref{prop:lbsigma}.
\end{proof}
In the following proposition we show that the spectral gap of $L^{D}$
is at least of order $N^{-2}$.
\begin{lem}
\label{lem:quasispectrum}One has that for large N
\begin{equation}
\lambda_{2}^{D}-\lambda_{1}^{D}\geq cN^{-2}.\label{eq:qell}
\end{equation}
\end{lem}
\begin{proof}
Recall that $\overline{\lambda}_{2}$ stand for the second smallest
eigenvalue of $-\widetilde{L}$. By the eigenvalue interlacing inequality
(see Theorem 2.1 of \cite{Interlacing}), we have 
\begin{equation}
\lambda_{2}^{D}\geq\overline{\lambda}_{2}.\label{eq:l2dgel2}
\end{equation}
While by the paragraph below equation (12) of \cite{AB} we have 
\begin{equation}
\lambda_{1}^{D}=\frac{1}{\overline{E}_{\sigma}[H_{A_{2}}]}.\label{eq:l1a2lb}
\end{equation}
By Lemma 10 a) of \cite{AB}, we have
\begin{equation}
\overline{E}_{\sigma}[H_{A_{2}}]\geq\overline{E}_{\pi}[H_{A_{2}}],\mbox{ or equivalently }\frac{1}{\overline{E}_{\sigma}[H_{A_{2}}]}\leq\frac{1}{\overline{E}_{\pi}[H_{A_{2}}]}.\label{eq:sigmapi}
\end{equation}
By an argument similar to the proof of Proposition \ref{prop:entranceinvup}
(by replacing $A_{1}$, $A_{2}$ by $A_{2}$, $A_{3}$), we find that
\begin{equation}
\frac{1}{\overline{E}_{\pi}[H_{A_{2}}]}\leq cN^{-d+(d-2)r_{2}}.\label{eq:h2hit}
\end{equation}
This implies by (\ref{eq:l1a2lb}) and (\ref{eq:sigmapi}) that
\begin{equation}
\lambda_{1}^{D}\leq cN^{-d+(d-2)r_{2}}.\label{eq:l1barub}
\end{equation}
Hence, we obtain that for large $N$ 
\begin{equation}
\lambda_{2}^{D}-\lambda_{1}^{D}\overset{(\ref{eq:l1barub})}{\underset{(\ref{eq:l2dgel2})}{\geq}}\overline{\lambda}_{2}-cN^{-d+(d-2)r_{2}}\overset{(\ref{eq:lbsg})}{\geq}c'N^{-2}.
\end{equation}

\end{proof}
The next Proposition shows that the q.s.d. on $D$ is very close to
the distribution of the confined walk at time $\overline{t_{*}}$
conditioned on not hitting $A_{2}$ (see (\ref{eq:regen}) for the
definition of $\overline{t_{*}}$).
\begin{prop}
\label{prop:qsdcp}One has that for large $N$,

\begin{equation}
\sup_{x,y\in D}|\overline{P}_{x}[Y_{\overline{t_{*}}}=y\big|H_{A_{2}}>\overline{t_{*}}]-\sigma(y)|\leq e^{-c\log^{2}N}.\label{eq:qsdcp}
\end{equation}
\end{prop}
\begin{proof}
The conditional probability in (\ref{eq:qsdcp}) is expressed through
$H_{\overline{t_{*}}}^{D}$ as 
\begin{equation}
P_{x}[X_{\overline{t_{*}}}=y|H_{A_{2}}>\overline{t_{*}}]=\frac{H_{\overline{t_{*}}}^{D}\delta_{y}(x)}{(H_{\overline{t_{*}}}^{D}\mathbf{1})(x)}.\label{quasi3}
\end{equation}

Now we calculate the numerator in the right-hand side of (\ref{quasi3}).
We decompose $\delta_{y}$ in the $l^{2}(\pi^{D})$ base $\{f_{i}\}_{i=1,\cdots,|D|}$:
\begin{equation}
\delta_{y}=\sum_{i=1}^{|D|}a_{i}f_{i},\mbox{ where }\label{eq:deltaydecomp}
\end{equation}
\begin{equation}
a_{i}=(\delta_{y},f_{i})_{l^{2}(\pi^{D})}=f_{i}(y)\pi^{D}(y)\mbox{, for }1\leq i\leq|D|.\label{eq:aicalc}
\end{equation}
Hence, one has 
\begin{equation}
H_{\overline{t_{*}}}^{D}\delta_{y}=\sum_{i=1}^{|D|}e^{-\lambda_{i}^{D}\overline{t_{*}}}a_{i}f_{i}=e^{-\lambda_{1}^{D}\overline{t_{*}}}\big(a_{1}f_{1}+\sum_{i=2}^{|D|}e^{(\lambda_{1}^{D}-\lambda_{i}^{D})\overline{t_{*}}}a_{i}f_{i}\big),
\end{equation}
so that 
\begin{equation}
H_{\overline{t_{*}}}^{D}\delta_{y}(x)=e^{-\lambda_{1}^{D}\overline{t_{*}}}\big(a_{1}f_{1}(x)+\sum_{i=2}^{|D|}e^{(\lambda_{1}^{D}-\lambda_{i}^{D})\overline{t_{*}}}a_{i}f_{i}(x)\big).\label{eq:qsdcalc1}
\end{equation}
By Proposition \ref{prop:lbsigma}, one has 
\begin{equation}
a_{1}f_{1}(x)\overset{(\ref{eq:aicalc})}{=}\pi^{D}(y)f_{1}(y)f_{1}(x)\overset{(\ref{eq:lowerboundf})}{\underset{(\ref{eq:cfwalkproperty})\,4.}{\geq}}N^{-c}.\label{eq:qsdest3}
\end{equation}
For large $N$, thanks to Lemma \ref{lem:quasispectrum}, the reminder
term inside the bracket of (\ref{eq:qsdcalc1}) is bounded by
\begin{align}
\begin{split}\Big|\sum_{i=2}^{|D|}e^{(\lambda_{1}^{D}-\lambda_{i}^{D})\overline{t_{*}}}a_{i}f_{i}(x)\Big|\leq\sum_{i=2}^{|D|}e^{(\lambda_{1}^{D}-\lambda_{i}^{D})\overline{t_{*}}}\Big|a_{i}f_{i}(x)\Big|\overset{(\ref{eq:liDorder})}{\text{\ensuremath{\underset{(\ref{eq:aicalc})}{\leq}}}}\sum_{i=2}^{|D|}e^{(\lambda_{1}^{D}-\lambda_{2}^{D})\overline{t_{*}}}\Big|\pi^{D}(y)f_{i}(y)f_{i}(x)\Big|\\
\overset{(\ref{eq:qell})}{\underset{(\ref{eq:regen})}{\leq}}|D|e^{-c\log^{2}N}\Big|\pi^{D}(y)f_{i}(y)f_{i}(x)\Big|\overset{(\ref{eq:Dsize})}{\leq}e^{-c'\log^{2}N}\sqrt{\pi^{D}(y)/\pi^{D}(x)}|f_{i}|_{l^{2}(\pi^{D})}^{2}\overset{(\ref{eq:cfwalkproperty})\,4.}{\leq}e^{-c''\log^{2}N}.
\end{split}
\label{eq:sgcontrol}
\end{align}
This implies that 
\begin{equation}
\bigg|\frac{H_{\overline{t_{*}}}^{D}\delta_{y}(x)}{e^{-\lambda_{1}^{D}\overline{t_{*}}}a_{1}f_{1}(x)}-1\bigg|=\bigg|\frac{H_{\overline{t_{*}}}^{D}\delta_{y}(x)}{e^{-\lambda_{1}^{D}\overline{t_{*}}}\pi^{D}(y)f_{1}(y)f_{1}(x)}-1\bigg|\overset{(\ref{eq:qsdcalc1})\textrm{-}(\ref{eq:sgcontrol})}{\leq}e^{-c\log^{2}N}.\label{eq:qsdfinal1}
\end{equation}
We now treat the denominator of the right-hand side of (\ref{quasi3}).
Since $f_{1}\geq0$, we have that
\begin{align}
\begin{split}(H_{\overline{t_{*}}}^{D}1)(x) & \qquad\ =\quad\quad\ \sum_{x'\in D}H_{\overline{t_{*}}}^{D}\delta_{x'}(x)=\sum_{x'\in D}\frac{1}{\pi^{D}(x)}(H_{\overline{t_{*}}}^{D}\delta_{x'},\delta_{x})_{l^{2}(\pi^{D})}\\
 & \overset{\mathrm{self-adjointness}}{=}\sum_{x'\in D}\frac{1}{\pi^{D}(x)}(\delta_{x'},H_{\overline{t_{*}}}^{D}\delta_{x})_{l^{2}(\pi^{D})}=\sum_{x'\in D}\frac{\pi^{D}(x')}{\pi^{D}(x)}H_{\overline{t_{*}}}^{D}\delta_{x}(x').
\end{split}
\label{eq:qsdcalc15}
\end{align}
Similar to (\ref{eq:deltaydecomp}) we decompose $\delta_{x}$ in
the $l^{2}(\pi^{D})$ orthonormal base $\{f_{i}\}_{i=1,\cdots,|D|}$:
\begin{equation}
\delta_{y}=\sum_{i=1}^{|D|}b_{i}f_{i},\mbox{ where }b_{i}=\pi^{D}(x)f_{i}(x)\mbox{ for all }i=1,\cdots,|D|.\label{eq:deltaydec}
\end{equation}
As in (\ref{eq:qsdcalc1}) we have for all $x'\in D$, 
\begin{equation}
H_{\overline{t_{*}}}^{D}\delta_{x}(x')=e^{-\lambda_{1}^{D}\overline{t_{*}}}\big(b_{1}f_{1}(x')+\sum_{i=2}^{|D|}e^{(\lambda_{1}^{D}-\lambda_{i}^{D})\overline{t_{*}}}b_{i}f_{i}(x')\big).\label{eq:deltaxcalc}
\end{equation}
So, by (\ref{eq:qsdcalc15}) and the above equality we find that 
\begin{equation}
\begin{array}{cc}
\begin{array}{rcl}
(H_{\overline{t_{*}}}^{D}1)(x) & = & \sum_{x'\in D}\frac{\pi^{D}(x')}{\pi^{D}(x)}e^{-\lambda_{1}^{D}\overline{t_{*}}}\big(b_{1}f_{1}(x')+\sum_{i=2}^{|D|}e^{(\lambda_{1}^{D}-\lambda_{i}^{D})\overline{t_{*}}}b_{i}f_{i}(x')\big)\\
 & \overset{(\ref{eq:deltaydec})}{=} & e^{-\lambda_{1}^{D}\overline{t_{*}}}\big(f_{1}(x)(f_{1},\mathbf{1})_{l^{2}(\pi^{D})}+\sum_{i=2}^{|D|}e^{(\lambda_{1}^{D}-\lambda_{i}^{D})\overline{t_{*}}}f_{i}(x)(f_{i},\mathbf{1})_{l^{2}(\pi^{D})}\big).
\end{array}\end{array}\label{eq:qsdcalc3}
\end{equation}
Again, by Proposition \ref{prop:lbsigma}, one has that 
\begin{equation}
(f_{1},\mathbf{1})_{l^{2}(\pi^{D})}\geq N^{-c}.\label{eq:f11norm}
\end{equation}
For large $N$, with the observation that $|{\bf 1}|{}_{l^{2}(\pi^{D})}\leq1$
and $|f_{i}|_{l^{2}(\pi^{D})}=1$, by the definition of $\pi^{D}$,
the reminder term inside the bracket of (\ref{eq:qsdcalc3}) is bounded
by (we apply Cauchy-Schwarz inequality at $(*)$) 
\begin{eqnarray}
 &  & \sum_{i=2}^{|D|}e^{(\lambda_{1}^{D}-\lambda_{i}^{D})\overline{t_{*}}}\Big|f_{i}(x)(f_{i},\mathbf{1})_{l^{2}(\pi^{D})}\Big|\,\quad\overset{(\ref{eq:liDorder})}{\leq}\sum_{i=2}^{|D|}e^{(\lambda_{1}^{D}-\lambda_{2}^{D})\overline{t_{*}}}\Big|f_{i}(x)(f_{i},\mathbf{1})_{l^{2}(\pi^{D})}\Big|\nonumber \\
 & \overset{(\ref{eq:qell}),(\ref{eq:regen})}{\underset{(*)}{\leq}} & \sum_{i=2}^{|D|}e^{-c\log^{2}N}\Big|f_{i}(x)\Big||f_{i}|{}_{l^{2}(\pi^{D})}|{\bf 1}|{}_{l^{2}(\pi^{D})}\leq\ \sum_{i=2}^{|D|}e^{-c\log^{2}N}\sqrt{\frac{1}{\pi^{D}(x)}}\Big|\sqrt{\pi^{D}(x)}f_{i}(x)\Big|\label{eq:qsdest2}\\
 & \overset{(\ref{eq:cfwalkproperty})\,4.}{\leq} & \sum_{i=2}^{|D|}e^{-c\log^{2}N}N^{c}|f_{i}|{}_{l^{2}(\pi^{D})}\leq e^{-c'\log^{2}N}.\nonumber 
\end{eqnarray}
Along with (\ref{eq:f11norm}), this shows that 
\begin{equation}
\bigg|\frac{(H_{\overline{t_{*}}}^{D}\mathbf{1})(x)}{e^{-\lambda_{1}^{D}\overline{t_{*}}}f_{1}(x)(f_{1},\mathbf{1})_{l^{2}(\pi^{D})}}-1\bigg|\leq e^{-c\log^{2}N}.\label{eq:qsdfinal2}
\end{equation}
Combining (\ref{eq:qsdfinal1}) and (\ref{eq:qsdfinal2}) one has
that for large $N$ and uniformly for all $x,y\in D$, (note that
we define $\sigma(\cdot)$ in (\ref{eq:quasidef})) 
\begin{equation}
\bigg|P_{x}[X_{\overline{t_{*}}}=y|H_{A_{2}}>\overline{t_{*}}]-\sigma(y)\bigg|\overset{(\ref{quasi3})}{=}\bigg|\frac{H_{\overline{t_{*}}}^{D}\delta_{y}(x)}{(H_{\overline{t_{*}}}^{D}\mathbf{1})(x)}-\sigma(y)\bigg|\leq e^{-c\log^{2}N}\sigma(y)\leq e^{-c\log^{2}N}.
\end{equation}
which is exactly the claim (\ref{eq:qsdcp}).
\end{proof}
We define the stopping time $V$ as the first time when the confined
random walk has stayed outside $A_{2}$ for a consecutive duration
of $\overline{t_{*}}$:
\begin{equation}
V=\inf\{t\geq\overline{t_{*}}:X_{[t-\overline{t_{*}},t]}\cap A_{2}=\emptyset\}.\label{eq:Vdef}
\end{equation}
The next lemma is a preparatory result for Proposition \ref{prop:qsdem}
below. This lemma shows that the probability $\overline{P}_{x}[V<\widetilde{H}_{A_{1}}]$,
when normalized by the sum of such probabilities as $x$ varies in
the inner boundary of $A_{1}$, is approximately equal to $\widetilde{e}_{A_{1}}(x)$,
the normalized equilibrium measure of $A_{1}$.
\begin{lem}
For large $N$, one has that
\begin{equation}
\bigg|\frac{\overline{P}_{x}[V<\widetilde{H}_{A_{1}}]}{\sum_{y\in\partial_{i}A_{1}}\overline{P}_{y}[V<\widetilde{H}_{A_{1}}]\widetilde{e}_{A_{1}}(x)}-1\bigg|\leq N^{-c}.\label{quni4}
\end{equation}
\end{lem}
\begin{proof}
For any $y\in\partial_{i}A_{1}$, by (\ref{eq:hittingtimebig1}) and
the strong Markov property applied at time $T_{A_{3}}$, we obtain
that
\begin{align}
\begin{split}\overline{P}_{y}[V<\widetilde{H}_{A_{1}}] & \overset{{\rm Markov}}{\underset{(\ref{eq:Vdef})}{\geq}}\;\,\overline{P}_{y}[T_{A_{3}}<\widetilde{H}_{A_{1}}]\inf_{x\in U^{N}\backslash A_{3}}\overline{P}_{x}[H_{A_{2}}>\overline{t_{*}}]\\
 & \overset{(\ref{eq:hittingtimebig1})}{\ \,\underset{(\ref{eq:coincidestopped})}{\geq}}\quad P_{y}[T_{A_{3}}<\widetilde{H}_{A_{1}}](1-N^{-c})\overset{(\ref{eq:emest})}{\geq}e_{A_{1}}(y)(1-N^{-c}).
\end{split}
\label{eq:quni3}
\end{align}
On the other hand, $\overline{P}_{y}[V<\widetilde{H}_{A_{1}}]$ is
bounded from above by
\begin{equation}
\overline{P}_{y}[V<\widetilde{H}_{A_{1}}]\leq\overline{P}_{y}[T_{A_{2}}<\widetilde{H}_{A_{1}}]\overset{(\ref{eq:coincidestopped})}{=}P_{y}[T_{A_{2}}<\widetilde{H}_{A_{1}}]\overset{(\ref{eq:emest})}{\leq}e_{A_{1}}(y)(1+N^{-c}).
\end{equation}
Together with (\ref{eq:quni3}), we find that 
\begin{equation}
(1-N^{-c})e_{A_{1}}(y)\leq\overline{P}_{y}[V<\widetilde{H}_{A_{1}}]\leq(1+N^{-c'})e_{A_{1}}(y),\mbox{ for any }y\in\partial_{i}A.\label{eq:capest1}
\end{equation}
Summing over $y\in\partial_{i}A_{1}$ we obtain that 
\begin{equation}
(1-N^{-c})\sum_{y\in\partial_{i}A_{1}}\overline{P}_{y}[V<\widetilde{H}_{A_{1}}]\leq\mathrm{cap}(A_{1})\leq(1+N^{-c'})\sum_{y\in\partial_{i}A_{1}}\overline{P}_{y}[V<\widetilde{H}_{A_{1}}].\label{eq:capest2}
\end{equation}
The claim (\ref{quni4}) follows by combining (\ref{eq:capest1})
and (\ref{eq:capest2}), recalling that by the definition of normalized
equilibrium measure, $\widetilde{e}_{A_{1}}(x)=e_{A_{1}(x)}\big/{\rm cap}(A_{1})$.
\end{proof}
The following proposition shows that the hitting distribution of the
confined walk on $A_{1}$ started from the q.s.d. on $D$ is very
close to the normalized equilibrium measure of $A_{1}$. The proof
of the next proposition is close to the proof of Lemma 3.10 of \cite{TeiWin},
and can be found in the Appendix at the end of this article.
\begin{prop}
\label{prop:qsdem}For large $N$ and any $x_{0}\in\Gamma^{N}$ (recall
that $A_{1}$ tacitly depends on $x_{0}$) one has
\begin{equation}
\sup_{x\in\partial_{i}A_{1}}\left|\frac{\overline{P}_{\sigma}[X_{H_{A_{1}}}=x]}{\widetilde{e}_{A_{1}}(x)}-1\right|\leq N^{-c}.\label{eq:quasi}
\end{equation}

\end{prop}

\section{{\normalsize{\label{sec:coupling}}}Chain coupling of excursions}

In this section we prove in Theorem \ref{thm:rwdisconnection} that
the tilted random walk disconnects $K_{N}$ from infinity with a probability,
which tends to $1$ as $N$ tends to infinity. For this purpose, we
show that the confined random walk visits the mesoscopic boxes $A_{1}$
centered at $\Gamma^{N}$(defined in (\ref{eq:boxdef})) sufficiently
often so that at time $T_{N}$ the trace of the walk ``locally'' dominates
(via a chain of couplings) random interlacements with intensity higher
than $u_{**}$. Hence, it disconnects in each such box the center
from its boundary with very high probability. Some arguments in this
section are based on Section 4 of \cite{TeiWin}, with necessary adaptions.
In this section the constants tacitly depend on $\delta$, $\eta$,
$\epsilon$ and $R$ (see (\ref{eq:deltadef}) and (\ref{eq:Rr})),
$r_{1}$, $r_{2}$, $r_{3}$, $r_{4}$ and $r_{5}$ (see (\ref{eq:rchoice})). 

Throughout this section we fix $x_{0}\in\Gamma^{N}$, the center of
the boxes $A_{1}$ through $A_{6}$, except in Proposition \ref{prop:tiltedblocking}
and Theorem \ref{thm:rwdisconnection}. 

\vspace{0.3cm}

We recall the definition of $V$ in (\ref{eq:Vdef}). For a path in
$\Gamma(U^{N})$, we denote by $R_{k}$ and $V_{k}$ the successive
entrance times $H_{A_{1}}$ and stopping times $V$ :
\[
\begin{array}{rclrcl}
R_{1} & = & H_{A_{1}}; & V_{1} & = & R_{1}+V\circ\theta_{R_{1}};\qquad\textrm{and for }i\geq2,\\
R_{i} & = & V_{i-1}+H_{A_{1}}\circ\theta_{V_{i-1}}; & V_{i} & = & R_{i}+V\circ\theta_{R_{i}}.
\end{array}
\]
Colloquially, we call such sections $X_{[R_{i},V_{i})}$ ``long excursions''
in contrast to the ``short excursions'' we will later define (see
above (\ref{eq:I2'def})). We set 
\begin{equation}
J=\lfloor(1+\epsilon/2)u_{**}\mathrm{cap}(A_{1})\rfloor.\label{eq:defofk}
\end{equation}
The next proposition shows that, with high probability, the confined
random walk has already made at least $J$ ``long excursions'' before
time $T_{N}$.
\begin{prop}
\label{prop:enoughexcursions}For large $N$, one has
\begin{equation}
\overline{P}_{0}[R_{J}\geq T_{N}]\leq e^{-N^{c}}.\label{eq:enoughexcursion}
\end{equation}

\end{prop}
The proof is deferred to the Appendix at the end of this article because
it is rather technical and similar to the proof of Lemma 4.3 of \cite{TeiWin}.

\vspace{0.3cm}

Next we construct a chain of couplings. Simply speaking, it is a sequence
of couplings involving multiple random sets, in which the preceding
set stochastically dominate the following set  with probability close
(or sometimes equal) to 1. 

We start with the first coupling. The following proposition shows
that one can construct a probability space where $(J-1)$ long excursions
(counted from the second excursion) coincide with high probability
with $(J-1)$ independent long excursions started from the q.s.d.
We write $\overline{P_{2}^{J}}=\otimes_{i=2}^{J}\overline{P}_{\sigma}^{i}$
for the product of $(J-1)$ independent copies of $\overline{P}_{\sigma}$.
We denote by $\mathcal{A}_{i}$ the random set $X_{[R_{i},V_{i})}\cap A_{1}$
and set $\mathcal{A}=\cup_{i=2}^{J}\mathcal{A}_{i}$. 
\begin{prop}
\label{prop:coupling0}For large $N$, there exists a probability
space $(\Omega_{0},\mathcal{B}_{0},Q_{0})$, endowed with a random
set $\mathcal{A}$ with the same law as $\mathcal{A}$ under $\overline{P}_{0}$
and random sets $\check{\mathcal{A}_{i}}$, $i=2,\ldots J$, distributed
as $\check{X}_{[0,V_{1})}^{i}\cap A_{1}$ where for $i\geq2$, $\check{X}^{i}$'s
are i.i.d. distributed as $X$ under $\overline{P}_{\sigma}$, such
that 
\begin{equation}
Q_{0}[\mathcal{A}\neq\check{\mathcal{A}}]\leq e^{-c''\log^{2}N}\label{eq:coupling0}
\end{equation}
where $\check{\mathcal{A}}=\cup_{i=2}^{J}\check{\mathcal{A}}_{i}$. \end{prop}
\begin{proof}
For each $x\in D$, we use Proposition 4.7, p. 50 in \cite{MCMT}
and Proposition \ref{prop:qsdcp} to construct a coupling $q_{x}$
of random variables $\Xi$ with the law of $X_{\overline{t_{*}}}$
under $\overline{P}_{x}[\cdot\big|H_{A_{2}}>\overline{t_{*}}]$ and
$\Sigma$ with the law of $\sigma$ such that 
\begin{equation}
\max_{x\in D}q_{x}[\Xi\neq\Sigma]\leq|D|e^{-c\log^{2}N}\leq e^{-c'\log^{2}N}.\label{eq:dec00}
\end{equation}

As in (\ref{eq:ldef}) we consider $L$ defined by 
\begin{equation}
L=\sup\{l\geq0:\:\tau_{l}\leq V,\, Z_{l}\in A_{2}\}.
\end{equation}
We then introduce $L_{i}=L\circ\theta_{R_{i}}+l_{i}$, where $l_{i}$
satisfies$\tau_{l_{i}}=R_{i}$ for $i\geq1$ as the last step at which
the $i$-th excursion is in $A_{2}$. 

We now construct $Q_{0}$ with the help of (\ref{eq:dec00}) in a
similar fashion to to the proof of Lemma 4.2 in \cite{TeiWin}. The
procedure goes inductively. We start by choosing $x_{1}^{+}\in\partial A_{2}$
according to $\overline{P}_{0}[Z_{L_{1}+1}=\cdot]$. For $i\geq1$,
if $x_{i}^{+}$ is chosen, we choose $x_{i+1}$ and $\check{x}_{i+1}$
points in $D=U^{N}\backslash A_{2}$ according to $q_{x_{i}^{+}}[\Xi=\cdot,\,\Sigma=\cdot]$.
If $x_{i+1}$ and $\check{x}_{i+1}$ coincide (which is the typical
case, i.e., if the coupling is successful at step $i+1$), we choose
$\mathcal{A}_{i+1}=\check{\mathcal{A}}_{i+1}$ subsets of $A_{1}$
and $x_{i+1}^{+}=\check{x}_{i+1}^{+}$ points in $\partial A_{2}$
according to $\overline{P}_{x_{i+1}}[\mathcal{A}_{1}=\cdot,\, Z_{L_{1}+1}=\cdot]$.
Otherwise, if $x_{i+1}$ differs from $\check{x}_{i+1}$ (which means
that the coupling fails at step $i+1$), then we choose independently
$\mathcal{A}_{i+1},\, x_{i+1}^{+}$ according to $\overline{P}_{x_{i+1}}[\mathcal{A}_{1}=\cdot,\, Z_{L_{1}+1}=\cdot]$
and $\check{\mathcal{A}}_{i+1},\,\check{x}_{i+1}^{+}$ according to
$\overline{P}_{\check{x}_{i+1}}[\mathcal{A}_{1}=\cdot,\, Z_{L_{1}+1}=\cdot]$.
In both cases, we repeat the above procedure until step $J$. Then
we write $\mathcal{A}=\cup_{i=2}^{J}\mathcal{A}_{i}$ and $\check{\mathcal{A}}=\cup_{i=2}^{J}\check{\mathcal{A}}_{i}$.

By a procedure as in the proof of Lemma 4.2 in \cite{TeiWin}, (we
replace $A$ by $A_{1}$, $B$ by $A_{2}$, $t_{*}$ by $\overline{t_{*}}$,
$\mathbb{T}$ by $U^{N}$, $X_{i}$ by $Z_{i}$, $Y_{t}$ by $X_{t}$,
$k$ by $J$, $U_{1}$ by $V_{1}$, $\bar{x}_{i}$ and $\bar{x}_{i}^{+}$
by $\check{x}_{i}$ and $\check{x}_{i}^{+}$), we can check that $Q_{0}$
is a coupling of $\mathcal{A}$ and $\check{\mathcal{A}}$, and the
probability that the coupling fails has an upper bound 
\begin{equation}
Q_{0}[\mathcal{A}\neq\check{\mathcal{A}}]\leq(J-1)\max_{x\in D}q_{x}[\Xi\neq\Sigma]\overset{(\ref{eq:defofk})}{\underset{(\ref{eq:dec00})}{\leq}}c'N^{d-2}e^{-c\log^{2}N}\leq e^{-c''\log^{2}N},
\end{equation}
which is exactly what we want.
\end{proof}
On an auxiliary probability space $(\mathcal{O}_{1},\mathcal{F}_{1},\mathcal{P}^{\mathcal{I}_{1}}$),
we denote by $\eta_{1}$ the Poisson point process on $\Gamma(U^{N})$
with intensity $(1+\epsilon/3)u_{**}\mathrm{cap}(A_{1})\kappa_{1}$,
where $\kappa_{1}$ is defined as the law of the stopped process $X_{(H_{A_{1}}+\cdot)\wedge V_{1}}$
under $\overline{P}_{\sigma}$. In other words, $\kappa_{1}$ is the
law of ``long excursions'' started from $\sigma$ and recorded from
the first time it enters $A_{1}$. We denote by 
\begin{equation}
\mathcal{I}_{1}=\cup_{\gamma\in\mathrm{supp}(\eta_{1})}\mathrm{Range}(\gamma)\cap A_{1}
\end{equation}
the trace of $\eta_{1}$ on $A_{1}$. In the next proposition we construct
a second coupling such that $\check{\mathcal{A}}$ dominates $\mathcal{I}_{1}$
with high probability.
\begin{prop}
\label{prop:coupling1}There exists a probability space $(\Omega_{1},\mathcal{B}_{1},Q_{1})$,
endowed with random sets $\mathcal{I}_{1}$ with the same law as $\mathcal{I}_{1}$
under $\mathcal{P}^{\mathcal{I}_{1}}$ and $\check{\mathcal{A}}$
with the same law as $\check{\mathcal{A}}$ under $\overline{P_{2}^{J}}$,
such that
\begin{equation}
Q_{1}[\check{\mathcal{A}}\supseteq\mathcal{I}_{1}]\geq1-e^{-N^{c}}.\label{eq:coupling1}
\end{equation}
\end{prop}
\begin{proof}
We pick a Poisson random variable $\xi$ with parameter $(1+\epsilon/3)u_{**}\mathrm{cap}(A_{1})$.
Then we generate (independently from $\xi$) an infinite sequence
$\{\check{X}^{i}\}_{i\geq1}$ of $\mbox{i.i.d.}$ confined walks under
$\overline{P}_{\sigma}$. We then let $\mathrm{\mathcal{I}}_{1}\sim\cup_{i=2}^{\xi+1}\check{X}_{[0,V_{1})}^{i}\cap A_{1}$
and $\check{\mathcal{A}}=\cup_{i=2}^{J}\check{X}_{[0,V_{1})}^{i}\cap A_{1}$,
both having the respective required laws. Moreover $\{\check{\mathcal{A}}\supseteq\mathcal{I}_{1}\}=\{J\geq\xi+1\}$,
by the definition of $J$ (see (\ref{eq:defofk})) and a standard
estimate on the deviation of Poisson random variables, we have
\begin{equation}
Q_{1}[\check{\mathcal{A}}\supseteq\mathcal{I}_{1}]=Q_{1}[J\geq\xi+1]\geq1-e^{-N^{c}},
\end{equation}
which is exactly (\ref{eq:coupling1}).
\end{proof}
Now on another auxiliary probability space $(\mathcal{O}_{2},\mathcal{F}_{2},\mathcal{P}^{\mathcal{I}_{2}}$),
we denote by $\eta_{2}$ the Poisson point process on $\Gamma(U^{N})$
with intensity $(1+\epsilon/4)u_{**}\mathrm{cap}(A_{1})\kappa_{2}$,
where $\kappa_{2}$ is defined as the law of the stopped process $X_{\cdot\wedge V_{1}}$
under $\overline{P}_{\widetilde{e}_{A_{1}}}$. In other words, it
is the law of ``long excursions'' started from the normalized equilibrium
measure of $A_{1}$ (note that, since in this case the excursions
start from inside $A_{1}$, we start recording directly from time
$0$). We denote by 
\begin{equation}
\mathcal{I}_{2}=\cup_{\gamma\in\mathrm{supp}(\eta_{2})}\mathrm{Range}(\gamma)\cap A_{1}
\end{equation}
the trace of $\eta_{2}$ on $A_{1}$. The next proposition and corollary
construct the third coupling so that $\mathcal{I}_{1}$ dominates
$\mathcal{I}_{2}$ almost surely. This is shown by proving that the
intensity measure of $\mathcal{I}_{1}$ is bigger than that of $\mathcal{I}_{2}$
with the help of Proposition \ref{prop:qsdem}.
\begin{prop}
\label{prop:coupling2prep}For large $N$, one has 
\begin{equation}
(1+\frac{\epsilon}{3})\kappa_{1}\geq(1+\frac{\epsilon}{4})\kappa_{2}.\label{eq:k1overk2}
\end{equation}
\end{prop}
\begin{proof}
By the definition of $\kappa_{1}$ and $\kappa_{2}$, and the strong
Markov property applied at time $H_{A_{1}}$, we can represent the United-Nikokym derivative of $\kappa_1$ and $\kappa_2$ through a function of the starting point of the trajectory
\begin{equation}
\frac{d\mathrm{\kappa_{1}}}{d\kappa_{2}}=\phi(X_0)\mbox{ where }\phi(x)=\frac{\overline{P}_{\sigma}[X_{H_{A_{1}}}=x]}{\widetilde{e}_{A_{1}}(x)}\mbox{ for all }x\in\partial_i A_1\mbox{ and 0 otherwise.}
\end{equation}
Hence we obtain, via (\ref{eq:quasi}), that for large $N$,
\begin{equation}
\frac{d(\mathrm{\kappa_{1}-\kappa_{2})}}{d\kappa_{2}}=\phi(X_0)-1\geq-N^{-c}\geq\frac{-\epsilon/12}{(1+\epsilon/3)}\quad\kappa_2\mbox{-a.s.}
\end{equation}
This implies (\ref{eq:k1overk2}) after rearrangement.
\end{proof}
As a consequence, we have the following corollary.
\begin{cor}
\label{cor:coupling2}For large $N$, there exists a probability space
$(\Sigma_{2},\mathcal{B}_{2},Q_{2})$ endowed with random sets $\mathcal{I}_{1}$
with the same law as $\mathcal{I}_{1}$ under $P^{\mathcal{I}_{1}}$
and $\mathcal{I}_{2}$ with the same law as $\mathcal{I}_{2}$ under
$P^{\mathcal{I}_{2}}$, such that 
\begin{equation}
\mathcal{I}_{1}\supseteq\mathcal{I}_{2}\qquad Q_{2}\textrm{-a.s.}\label{eq:coupling2}
\end{equation}
\end{cor}
\begin{proof}
This follows immediately from the domination of measures. Indeed,
we first construct $\mathcal{I}_{2}$ on some probability space. Then
we consider the positive measure on $\Gamma(U^{N})$ 
\begin{equation}
\alpha=(1+\epsilon/3)\kappa_{1}-(1+\epsilon/4)\kappa_{2},
\end{equation}
and construct (independently from $\mathcal{I}_{2}$) a Poisson point
process $\widehat{\eta}$ on $\Gamma(U^{N})$ with intensity measure
$\alpha$. Then $\mathcal{I}_{1}=(\cup_{\gamma\in\mathrm{supp}(\widehat{\eta})}\mathrm{Range}(\gamma)\cap A_{1})\cup\mathcal{I}_{2}$
has the required law.
\end{proof}
On another auxiliary probability space $(\mathcal{O}'_{2},\mathcal{F}'_{2},\mathcal{P}^{\mathcal{I}'_{2}})$,
we denote by $\eta'_{2}$ the law of the Poisson point process on
$\Gamma(U^{N})$ with intensity $(1+\epsilon/4)u_{**}\mathrm{cap}(A_{1})\kappa'_{2}$,
where $\kappa'_{2}$ is defined as the stopped process $X_{\cdot\wedge T_{A_{2}}}$
under $P_{\widetilde{e}_{A_{1}}}$, or equivalently $\overline{P}_{\widetilde{e}_{A_{1}}}$.
Contrary to the definition of a ``long excursion'', we would like
to call $X_{[H_{A_{1}},T_{A_{2}})}$ a ``short excursion'', since
we stop the excursion earlier than a ``long excursion'' (this is because
$T_{A_{2}}<V_{1}$). In other words, $\kappa'_{2}$ is the measure
of ``short excursions'' started from the normalized equilibrium measure
of $A_{1}$. We denote by 
\begin{equation}
\mathcal{I}'_{2}=\cup_{\gamma\in\mathrm{supp}(\eta'_{2})}\mathrm{Range}(\gamma)\cap A_{1}\label{eq:I2'def}
\end{equation}
the trace of $\eta'_{2}$ in $A_{1}$. Hence we can naturally construct
the fourth coupling such that $\mathcal{I}_{2}$ dominates $\mathcal{I}'_{2}$
almost surely, which is stated in the corollary below.
\begin{cor}
\label{cor:coupling2'}When $N$ is large, there exists a probability
space \textup{$(\Sigma'_{2},\mathcal{B}'_{2},Q'_{2})$}, endowed with
random sets $\mathcal{I}'_{2}$ with the same law as $\mathcal{I}_{2}'$
under $\mathcal{P}^{\mathcal{I}'_{2}}$, and $\mathcal{I}_{2}$ with
the same law as $\mathcal{I}_{2}$ under $\mathcal{P}^{\mathcal{I}{}_{2}}$
such that 
\begin{equation}
\mathcal{I}_{2}\supseteq\mathcal{I}'_{2}\qquad Q'_{2}\textrm{-a.s.}\label{eq:coupling2'}
\end{equation}

\end{cor}
The fifth coupling establishes the stochastic domination of $\mathcal{I}'_{2}$
on the trace of $\mathcal{I}^{(1+\epsilon/8)u_{**}}$ in $A_{1}$.
It is reproduced from \cite{Gumbel}.
\begin{prop}
\label{prop:coupling3} When $N$ is large, there exists a probability
space $(\Sigma_{3},\mathcal{B}_{3},Q_{3})$ endowed with random sets
$\mathcal{I}$ with the same law as \textup{$\mathcal{I}^{u_{**}(1+\frac{\epsilon}{8})}\cap A_{1}$
under $\mathbb{P}$} and $\mathcal{I}'_{2}$ with the same law as
$\mathcal{I}'_{2}$ under $P^{\mathcal{I}'_{2}}$\textup{, such that
\begin{equation}
Q_{3}[\mathcal{I}'_{2}\supseteq\mathcal{I}]\geq1-e^{-N^{c}}.\label{eq:coupling3}
\end{equation}
}
\end{prop}
We refer the readers to Proposition 5.4 of \cite{Gumbel} and to Section
9 of \cite{Gumbel} for its proof. 

\vspace{0.3cm}

The next proposition links together the above couplings from Propositions
\ref{prop:coupling0}, \ref{prop:coupling1}, Corollaries \ref{cor:coupling2},
\ref{cor:coupling2'}, and Proposition \ref{prop:coupling3}. We prove
that for any $x_{0}$ in the ``strip'' $\Gamma^{N}$, the probability
that it is connected in $\mathcal{V}$ (i.e., the vacant set of the
random walk, see below (\ref{eq:filtration})) to the (inner) boundary
of $A_{1}^{x_{0}}$ is small.
\begin{prop}
\label{prop:tiltedblocking}For large $N$ and all $x_{0}\in\Gamma^{N}$,
one has 
\begin{equation}
\widetilde{P}_{N}[x_{0}\overset{\mathcal{V}}{\longleftrightarrow}\partial_{i}A_{1}^{x_{0}}]\leq e^{-c\log^{2}N}.\label{eq:tiltedblocking}
\end{equation}
\end{prop}
\begin{proof}
First, by Corollary \ref{cor:coincide}, for all $x_{0}\in\Gamma^{N}$,
one can replace $\widetilde{P}_{N}$ (the tilted walk) with $\overline{P}_{0}$
(the confined walk) stopped at time $T_{N}$ so as to benefit from
the various results we obtained for $\overline{P}_{0}$:
\begin{equation}
\widetilde{P}_{N}[x_{0}\overset{\mathcal{V}}{\longleftrightarrow}\partial_{i}A_{1}^{x_{0}}]\leq\widetilde{P}_{N}[x_{0}\overset{(X_{[R_{2},T_{N})})^{c}}{\longleftrightarrow}\partial_{i}A_{1}^{x_{0}}]\overset{(\ref{eq:coincide})}{=}\overline{P}_{0}[x_{0}\overset{(X_{[R_{2},T_{N})})^{c}}{\longleftrightarrow}\partial_{i}A_{1}^{x_{0}}].\label{eq:transfertoconfine}
\end{equation}
By Proposition \ref{prop:enoughexcursions} and the first coupling,
namely Proposition \ref{prop:coupling0}, (see ibid. for notation),
for large $N$, the right-hand quantity in (\ref{eq:transfertoconfine})
is bounded above by the probability that there is a connection in
$\check{\mathcal{A}}^{c}$ (from $x$ to $\partial_{i}A_{1}^{x_{0}}$)
plus a small correction:
\begin{align}
\begin{array}{cc}
\begin{array}{c}
\overline{P}_{0}[x_{0}\overset{(X_{[R_{2},T_{N})})^{c}}{\longleftrightarrow}\partial_{i}A_{1}^{x_{0}}]\leq\overline{P}_{0}[x_{0}\overset{\mathcal{A}^{c}}{\longleftrightarrow}\partial_{i}A_{1}^{x_{0}},\ R_{J}<T_{N}]+\overline{P}_{0}[R_{J}\geq T_{N}]\\
\overset{(\ref{eq:enoughexcursion})}{\underset{(\ref{eq:coupling0})}{\leq}}Q_{0}[x_{0}\overset{\check{\mathcal{A}}^{c}}{\longleftrightarrow}\partial_{i}A_{1}^{x_{0}}]+e^{-c\log^{2}N}=\overline{P_{2}^{J}}[x_{0}\overset{\check{\mathcal{A}}^{c}}{\longleftrightarrow}\partial_{i}A_{1}^{x_{0}}]+e^{-c\log^{2}N}.
\end{array}\end{array}\label{eq:ccineq1}
\end{align}
Then, by the second coupling, namely Corollary \ref{prop:coupling1},
one has that the first term of the last equation in (\ref{eq:ccineq1})
is bounded above by the probability that there is a connection in
$\mathcal{I}_{1}^{c}$ plus a small correction:
\begin{align}
\begin{split}\overline{P_{2}^{J}}[x_{0}\overset{\check{\mathcal{A}}^{c}}{\longleftrightarrow}\partial_{i}A_{1}^{x_{0}}] & \ \,\leq\ Q_{1}[x_{0}\overset{\check{\mathcal{A}}^{c}}{\longleftrightarrow}\partial_{i}A_{1}^{x_{0}},\:\check{\mathcal{A}}\supseteq\mathcal{I}_{1}]+Q_{1}[\check{\mathcal{A}}\nsupseteq\mathcal{I}_{1}]\\
 & \overset{(\ref{eq:coupling1})}{\leq}Q_{1}[x_{0}\overset{\mathcal{I}_{1}^{c}}{\longleftrightarrow}\partial_{i}A_{1}^{x_{0}}]+e^{-N^{c}}=P^{\mathcal{I}_{1}}[x_{0}\overset{\mathcal{I}_{1}^{c}}{\longleftrightarrow}\partial_{i}A_{1}^{x_{0}}]+e^{-N^{c}}.
\end{split}
\label{eq:ccineq2}
\end{align}
Next, by the third coupling, namely Corollary \ref{cor:coupling2},
one has that the first term in the last equation of (\ref{eq:ccineq2})
is smaller or equal to the probability that there is a connection
in $\mathcal{I}_{2}^{c}$:
\begin{equation}
\mathcal{P}^{\mathcal{I}_{1}}[x_{0}\overset{\mathcal{I}_{1}^{c}}{\longleftrightarrow}\partial_{i}A_{1}^{x_{0}}]=Q_{2}[x_{0}\overset{\mathcal{I}_{1}^{c}}{\longleftrightarrow}\partial_{i}A_{1}^{x_{0}}]\overset{(\ref{eq:coupling2})}{\leq}Q_{2}[x_{0}\overset{\mathcal{I}_{2}^{c}}{\longleftrightarrow}\partial_{i}A_{1}^{x_{0}}]=\mathcal{P}{}^{\mathcal{I}_{2}}[x_{0}\overset{\mathcal{I}_{2}^{c}}{\longleftrightarrow}\partial_{i}A_{1}^{x_{0}}].\label{eq:ccineq3}
\end{equation}
Next, by the fourth coupling, namely Corollary \ref{cor:coupling2'},
one has that the first term in the last equation of (\ref{eq:ccineq3})
is smaller or equal to the probability that there is a connection
in $\mathcal{I}_{2}^{'c}$:
\begin{equation}
\mathcal{P}^{\mathcal{I}_{2}}[x_{0}\overset{\mathcal{I}_{2}^{c}}{\longleftrightarrow}\partial_{i}A_{1}^{x_{0}}]=Q'_{2}[x_{0}\overset{\mathcal{I}_{2}^{c}}{\longleftrightarrow}\partial_{i}A_{1}^{x_{0}}]\overset{(\ref{eq:coupling2'})}{\leq}Q'_{2}[x_{0}\overset{\mathcal{I}_{2}^{'c}}{\longleftrightarrow}\partial_{i}A_{1}^{x_{0}}]=\mathcal{P}{}^{\mathcal{I}'_{2}}[x_{0}\overset{\mathcal{I}_{2}^{'c}}{\longleftrightarrow}\partial_{i}A_{1}^{x_{0}}].\label{eq:ccineq4}
\end{equation}
Finally, by the fifth coupling, namely Proposition \ref{prop:coupling3},
and the strong super-criticality of random interlacements (see (\ref{eq:supercrit})),
one has
\begin{align}
\begin{split}\mathcal{P}^{\mathcal{I}'_{2}}[x_{0}\overset{\mathcal{I}_{2}^{'c}}{\longleftrightarrow}\partial_{i}A_{1}^{x_{0}}]=Q_{3}[x_{0}\overset{\mathcal{I}_{2}^{'c}}{\longleftrightarrow}\partial_{i}A_{1}^{x_{0}}]\leq Q_{3}[\mathcal{I}'_{2}\supseteq\mathcal{I},\, x_{0}\overset{\mathcal{I}^{c}}{\longleftrightarrow}\partial_{i}A_{1}^{x_{0}}]+Q_{3}[\mathcal{I}'_{2}\nsupseteq\mathcal{I}]\\
\overset{(\ref{eq:coupling3})}{\leq}Q_{3}[x_{0}\overset{\mathcal{I}^{c}}{\longleftrightarrow}\partial_{i}A_{1}^{x_{0}}]+e^{-N^{c}}=\mathbb{P}[x_{0}\overset{\mathcal{V}^{u_{**}(1+\epsilon/8)}}{\longleftrightarrow}\partial_{i}A_{1}^{x_{0}}]+e^{-N^{c}}\overset{(\ref{eq:supercrit})}{\leq}e^{-N^{c'}}.
\end{split}
\label{eq:ccineq5}
\end{align}

The claim (\ref{eq:tiltedblocking}) hence follows by collecting (\ref{eq:ccineq1})-(\ref{eq:ccineq5}).
\end{proof}
We are ready now to state and prove the main result of this section,
namely that the tilted disconnection probability tends to $1$ as
$N$ tends to infinity.
\begin{thm}
\label{thm:rwdisconnection}
\begin{equation}
\lim_{N\to\infty}\widetilde{P}_{N}[K_{N}\overset{\mathcal{V}}{\nleftrightarrow}\infty]=1.
\end{equation}
\end{thm}
\begin{proof}
Note that for large $N$, if a nearest-neighbour path connects $K_{N}$
and infinity, it must go through the set $\Gamma^{N}$ at some point
$x_{0}$ (see above (\ref{eq:rchoice}) for the definition of $\Gamma^{N}$).
Hence, it connects $x_{0}$ to the inner boundary of $A_{1}^{x_{0}}$,
so that
\begin{equation}
\{K_{N}\overset{\mathcal{V}}{\nleftrightarrow}\infty\}^{c}\subset\cup_{x_{0}\in\Gamma^{N}}\{x_{0}\overset{\mathcal{V}}{\longleftrightarrow}\partial_{i}A_{1}^{x_{0}}\}.
\end{equation}
Thus, we see that for large $N$,
\begin{equation}
\widetilde{P}_{N}[\{K_{N}\overset{\mathcal{V}}{\nleftrightarrow}\infty\}^{c}]\leq\sum_{x_{0}\in\Gamma^{N}}\widetilde{P}_{N}[x_{0}\overset{\mathcal{V}}{\longleftrightarrow}\partial_{i}A_{1}^{x_{0}}].\label{eq:probdecomp}
\end{equation}
By Proposition \ref{prop:tiltedblocking}, we find that for large
$N$, uniformly for each $x_{0}\in\Gamma^{N}$, we can bound each
term on right-hand side of (\ref{eq:probdecomp}), and find
\begin{equation}
\widetilde{P}_{N}[\{K_{N}\overset{\mathcal{V}}{\nleftrightarrow}\infty\}^{c}]\leq|\Gamma^{N}|e^{-c\log^{2}N}\underset{_{N\to\infty}}{\longrightarrow}0.
\end{equation}
This completes the proof of Theorem \ref{thm:rwdisconnection}.
\end{proof}

\section{Denouement and Epilogue}

In this section we combine the main ingredients, namely Theorem \ref{thm:rwdisconnection}
and Proposition \ref{prop:limsup} and prove Theorem \ref{thm:mainTheorem}. 
\begin{proof}[Proof of Theorem \ref{thm:mainTheorem}]
We recall the entropy inequality (see (\ref{eq:Entropychange})),
and apply it to $P_{0}$ and $\widetilde{P}_{N}$ (which is defined
in Section 2). By Theorem \ref{thm:rwdisconnection}, one has 
\begin{equation}
\lim_{N\to\infty}\widetilde{P}_{N}[K_{N}\overset{\mathcal{V}}{\nleftrightarrow}\infty]=1,
\end{equation}
thus the relative entropy inequality (\ref{eq:Entropychange}) yields
that
\begin{equation}
\liminf_{N\to\infty}\frac{1}{N^{d-2}}\log(P_{0}[K_{N}\overset{\mathcal{V}}{\nleftrightarrow}\infty])\geq-\limsup_{N\to\infty}\frac{1}{N^{d-2}}H(\widetilde{P}_{N}|P_{0}).\label{eq:entropychangeappli}
\end{equation}
Then, as in the proof of Proposition \ref{prop:limsup}, taking consecutively
the limsup as $\eta\to0$, $R\to\infty$, $\delta\to0$ and $\epsilon\to0$,
one has
\begin{equation}
\limsup_{\epsilon\to0}\limsup_{\delta\to0}\limsup_{R\to\infty}\limsup_{\eta\to0}\limsup_{N\to\infty}\frac{1}{N^{d-2}}H(\widetilde{P}_{N}|P_{0})\leq\frac{u_{**}}{d}\mathrm{cap}_{\mathbb{R}^{d}}(K),
\end{equation}
proving Theorem \ref{thm:mainTheorem}.\end{proof}
\begin{rem}
\label{finalremark}1) The proof of Theorem \ref{thm:mainTheorem}
not only confirms the conjecture proposed in Remark 5.1 2) of \cite{Li-Szn lb},
but also shows, after minor changes, that for any $M>1$,
\begin{equation}
\liminf_{N\to\infty}\frac{1}{N^{d-2}}\log(P_{0}[B_{N}\overset{\mathcal{V}}{\nleftrightarrow}S_{N}])\geq-\frac{u_{**}}{d}\mathrm{cap}_{\mathbb{R}^{d}}([-1,1]^{d}),
\end{equation}
where (in the notation of \cite{dscn2-Szn}) $B_{N}=\{x\in\mathbb{Z}^{d};\,|x|_{\infty}\leq N\}$
and $S_{N}=\{x\in\mathbb{Z}^{d};\,|x|_{\infty}=[MN]\}$. This asymptotic
lower bound possibly matches the asymptotic upper bound (\ref{eq:mainubnew})
recently obtained in \cite{dscn2-Szn}.

2) Assume for simplicity that the compact $K$ is regular. Notice
that unlike what happens for $d\geq5$, when $d=3,4$, the function
$h$ defined in (\ref{eq:Dirichletproblem}) is not in $L^{2}(\mathbb{R}^{d})$,
and $h_{N}(x)=h(\frac{x}{N})$ is not in $l^{2}(\mathbb{Z}^{d})$.
This fact affects $T_{N}$ defined in (\ref{eq:TNdef}) (which diverges
if $R\to\infty$ when $d=3,4$, but not when $d\ge5$). One can wonder
whether this feature reflects different qualitative behaviours of
the random walk path under the conditional measure $P_{0}[\cdot|K_{N}\overset{{\cal V}}{\nleftrightarrow}\infty]$
when $N$ becomes large?
\end{rem}
\appendix

\section{Appendix}

{\small{In the appendix we include the proof of Propositions \ref{prop:qsdem}
and \ref{prop:enoughexcursions}.}}{\small \par}
\begin{proof}[Proof of Proposition \ref{prop:qsdem}]
{\small{We first prove that for $x\in\partial_{i}A_{1}$ 
\begin{equation}
\bigg|\overline{P}_{x}[V<\widetilde{H}_{A_{1}}]-\overline{P}_{\sigma}[X_{H_{A_{1}}}=x]\sum_{y\in\partial_{i}A_{1}}\overline{P}_{y}[V<\widetilde{H}_{A_{1}}]\bigg|\leq e^{-c\log^{2}N},\label{quni2}
\end{equation}
and, as we will see, the claim (\ref{eq:quasi}) will then follow.
We consider in the left-hand side of (\ref{quni1}) the probability
that the random walk started from $x\in\partial_{i}A_{1}$ stays in
$D$ for a time interval of length $\overline{t_{*}}$ before returning
to $A_{1}$, and then returns to $A_{1}$ through some vertex other
than $x$. By reversibility of the confined walk, and the fact that
by claim 3. of (\ref{eq:fproperty}) and claim 1. of (\ref{eq:cfwalkproperty}),
$\pi(y)=\pi(x)$ for all $y\in\partial_{i}A_{1}$, this probability
can be written as 
\begin{equation}
\sum_{y\in\partial_{i}A_{1}\setminus\{x\}}\overline{P}_{x}[V<\widetilde{H}_{A_{1}},\, X_{H_{A_{1}}}=y]=\sum_{y\in\partial_{i}A_{1}\setminus\{x\}}\overline{P}_{y}[V<\widetilde{H}_{A_{1}},\, X_{H_{A_{1}}}=x].\label{quni1}
\end{equation}
We introduce $L$, the index of last ``step'' of the path in $A_{2}$
before time $V$ (see (\ref{eq:taudef}) and the paragraph above (\ref{eq:skeleton chain})
for the definition of $\tau_{l}$ and $Z_{l}$ respectively):
\begin{equation}
L=\sup\{l:\:\tau_{l}\leq V,\, Z_{l}\in A_{2}\}.\label{eq:ldef}
\end{equation}
We consider the summands from (\ref{quni1}): for all $x,y\in\partial_{i}A_{1}$,
we sum over all possible values of $L$ and $X_{\tau_{L}}=Z_{L}$
(recall the definition of $\tau_{l}$ in (\ref{eq:taudef}) and the
relation between $X_{\tau_{l}}$ and $Z_{l}$ in (\ref{eq:skeleton chain})),
and apply Markov property at the times $\tau_{l+1}$ and $\tau_{l+1}+\overline{t_{*}}$:
\begin{eqnarray}
 &  & \overline{P}_{x}[V<\widetilde{H}_{A_{1}},X_{\widetilde{H}_{A_{1}}}=y]=\sum_{l\geq0,x'\in\partial_{i}A_{2}}\overline{P}_{x}[L=l,\, Z_{l}=x',\, V<\widetilde{H}_{A_{1}},X_{\widetilde{H}_{A_{1}}}=y]\nonumber \\
 & = & \sum_{l\geq0,x'\in\partial_{i}A_{2}}\overline{P}_{x}[Z_{l}=x',\,\tau_{l}<\widetilde{H}_{A_{1}}\wedge V,\, H_{A_{2}}\circ\theta_{\tau_{l+1}}>\overline{t_{*}},\, X_{\widetilde{H}_{A_{1}}}=y]\label{eq:qsdcp1}\\
 & = & \sum_{\overset{l\geq0,x''\in D}{x'\in\partial_{i}A_{2}}}\overline{E}_{x}\Big[Z_{l}=x',\,\tau_{l}<\widetilde{H}_{A_{1}}\wedge V,\,\overline{P}_{Z_{l+1}}[H_{A_{2}}>\overline{t_{*}}]\overline{P}_{Z_{l+1}}[X_{\overline{t_{*}}}=x''|H_{A_{2}}>\overline{t_{*}}]\Big]\overline{P}_{x''}[X_{H_{A_{1}}}=y],\nonumber 
\end{eqnarray}
(we will soon use the fact that the conditioned probability in the
last expression is close to $\sigma(x'')$ by Proposition \ref{prop:qsdcp}).
Similarly we have
\begin{equation}
\overline{P}_{x}[V<\widetilde{H}_{A_{1}}]=\sum_{l\geq0,x'\in\partial_{i}A_{2}}\overline{E}_{x}\Big[Z_{l}=x',\,\tau_{l}<\widetilde{H}_{A_{1}}\wedge V,\,\overline{P}_{Z_{l+1}}[H_{A_{2}}>\overline{t_{*}}]\Big].
\end{equation}
This implies that
\begin{equation}
\overline{P}_{x}[V<\widetilde{H}_{A_{1}}]\overline{P}_{\sigma}[X_{H_{A_{1}}}=y]=\sum_{\overset{l\geq0,x''\in D}{x'\in\partial_{i}A_{2}}}\overline{E}_{x}\Big[Z_{l}=x',\,\tau_{l}<\widetilde{H}_{A_{1}}\wedge V,\,\overline{P}_{Z_{l+1}}[H_{A_{2}}>\overline{t_{*}}]\Big]\sigma(x'')\overline{P}_{x''}[X_{H_{A_{1}}}=y].\label{eq:qsdcp2}
\end{equation}
Hence, by combining (\ref{eq:qsdcp1}) and (\ref{eq:qsdcp2}) we have
\begin{equation}
\Big|\overline{P}_{x}[V<\widetilde{H}_{A_{1}},X_{\widetilde{H}_{A_{1}}}=y]-\overline{P}_{x}[V<\widetilde{H}_{A_{1}}]\overline{P}_{\sigma}[X_{H_{A_{1}}}=y]\Big|\overset{(\ref{eq:qsdcp})}{\leq}e^{-c\log^{2}N}.
\end{equation}
Applying this estimate in both sides in (\ref{quni1}), we obtain
that 
\begin{equation}
\Big|\overline{P}_{x}[V<\widetilde{H}_{A_{1}}]\overline{P}_{\sigma}[X_{H_{A_{1}}}\neq x]-\sum_{y\in\partial_{i}A_{1}\setminus\{x\}}\overline{P_{y}}[V<\widetilde{H}_{A_{1}}]\overline{P}_{\sigma}[X_{H_{A_{1}}}=x]\Big|\leq e^{-c\log^{2}N}.
\end{equation}
Finally, by adding and subtracting $\overline{P}_{x}[V<\widetilde{H}_{A_{1}}]\overline{P}_{\sigma}[X_{H_{A_{1}}}=x]$,
we obtain (\ref{quni2}) as desired.}}{\small \par}

{\small{Now we prove (\ref{eq:quasi}). By (\ref{eq:emboxlb}) and
(\ref{eq:capest2}) one has that 
\begin{equation}
\sum_{y\in\partial_{i}A_{1}}\overline{P}_{y}[V<\widetilde{H}_{A_{1}}]\widetilde{e}_{A_{1}}(x)\geq N^{-c'}.\label{eq:quni5}
\end{equation}
Hence dividing (\ref{quni2}) by the left-hand term of (\ref{eq:quni5}),
one obtains 
\begin{equation}
\bigg|\frac{\overline{P}_{x}[V<\widetilde{H}_{A_{1}}]}{\sum_{y\in\partial_{i}A_{1}}\overline{P}_{y}[V<\widetilde{H}_{A_{1}}]\widetilde{e}_{A_{1}}(x)}-\frac{\overline{P}_{\sigma}[X_{H_{A_{1}}}=x]}{\widetilde{e}_{A_{1}}(x)}\bigg|\leq e^{-c'\log^{2}N},
\end{equation}
and together with (\ref{quni4}) the proof of (\ref{eq:quasi}) is
complete.}}{\small \par}
\end{proof}

\begin{proof}[Proof of Proposition \ref{prop:enoughexcursions}]
{\small{In this proof we always assume that $N$ is sufficiently
large. We recall the definition of $T_{N}$ in (\ref{eq:TNdef}) and
the choice of $\epsilon$ in (\ref{eq:deltadef}). In order to prove
(\ref{eq:enoughexcursion}), we observe that, $\overline{P}_{0}$-a.s.,
\begin{equation}
\begin{split}\{R_{J}\geq T_{N}\}\subseteq & \left\{ H_{A_{1}}+H_{A_{1}}\circ\theta_{V_{1}}+\cdots+H_{A_{1}}\circ\theta_{V_{J-1}}\geq(1-\frac{\epsilon}{100})T_{N}\right\} \\
 & \cup\left\{ V\circ\theta_{R_{1}}+\cdots+V\circ\theta_{R_{J-1}}\geq\frac{\epsilon}{100}T_{N}\right\} ,
\end{split}
\label{eld3}
\end{equation}
that is, the (unlikely) event $\{R_{J}\geq T_{N}\}$ happens only
when either the sum of $H_{A_{1}}$'s exceeds a quantity close to
$T_{N}$ or the sum of shifted $V$'s exceeds a small quantity (but
still of order $T_{N}$). Now we give an upper bound to their respective
probabilities. We define 
\begin{equation}
t_{N}=\sup_{y\in U^{N}}\overline{E}_{y}[H_{A_{1}}],\label{eq:tNdef}
\end{equation}
which is the maximum of the expected entrance time in $A_{1}$ starting
from an arbitrary point in $U^{N}$ (it is not much bigger than $\overline{E}_{\pi}[H_{A_{1}}]$
by (\ref{eq:tnsup})). By the exponential Chebychev inequality and
the strong Markov property applied inductively at $V_{1},\cdots,V_{J-1}$
and $R_{1},\cdots,R_{J-1}$, we deduce from (\ref{eld3}) that, for
any $\theta>0$, 
\begin{equation}
\overline{P}_{0}[R_{J}\geq T_{N}]\leq\exp\bigg(-\theta(1-\frac{\epsilon}{100})\frac{T_{N}}{t_{N}}\bigg)\bigg(\sup_{x\in U^{N}}\overline{E}_{x}\Big[\exp(\theta\frac{H_{A_{1}}}{t_{N}})\Big]\bigg)^{J}+\exp(-\frac{\epsilon}{100}\frac{T_{N}}{t_{N}})\bigg(\sup_{x\in A_{1}}\overline{E}_{x}\Big[e^{\frac{V}{t_{N}}}\Big]\bigg)^{J}.\label{poisson0}
\end{equation}
We now treat the first term on the right-hand side of (\ref{poisson0}).
Khasminskii's Lemma (see (4) and (6) in \cite{Khasminskii}) states
that for all $B$ subset of $U^{N}$ and $n\geq1$, 
\begin{equation}
\sup_{x\in U^{N}}\overline{E}_{x}[H_{B}^{n}]\leq n!\sup_{y\in U^{N}}\overline{E}_{y}[H_{B}]^{n}.\label{eq:Khasminskii}
\end{equation}
Hence we have 
\begin{equation}
\sup_{x\in U^{N}}\overline{E}_{x}\left[\exp(\theta\frac{H_{A_{1}}}{t_{N}})\right]\leq\sum_{j=0}^{\infty}\frac{\theta^{j}}{j\text{!}t_{N}^{j}}\sup_{x\in U^{N}}\overline{E}_{x}[H_{A_{1}}^{j}]\overset{(\ref{eq:Khasminskii})}{\underset{(\ref{eq:tNdef})}{\leq}}\sum_{j=0}^{\infty}\theta^{j}=\frac{1}{1-\theta}\mbox{ for }\theta\in(0,\frac{1}{2}).\label{poisson1}
\end{equation}
Now, we derive an upper bound for $\sup_{x\in A_{2}}\overline{E}_{x}[\exp(\frac{V}{t_{N}})]$
and treat the second term on the right-hand side of (\ref{poisson0}).
We first note that, $\overline{P}_{x}$-a.s.~for any $x\in A_{2}$,
\begin{equation}
\begin{array}{cc}
\begin{array}{ccl}
V & \leq & (T_{A_{3}}+\overline{t_{*}})1_{\{H_{A_{2}}\circ\theta_{T_{A_{3}}}>\overline{t_{*}}\}}+\big(T_{A_{3}}+\overline{t_{*}}+V\circ\theta_{H_{A_{2}}}\circ\theta_{T_{A_{3}}}\big)1_{\{H_{A_{2}}\circ\theta_{T_{A_{3}}}\leq\overline{t_{*}}\}}\\
 & = & T_{A_{3}}+\overline{t_{*}}+V\circ\theta_{H_{A_{2}}}\circ\theta_{T_{A_{3}}}1_{\{H_{A_{2}}\circ\theta_{T_{A_{3}}}\leq\overline{t_{*}}\}}.
\end{array}\end{array}\label{eq:Vub}
\end{equation}
By the strong Markov property applied at $H_{A_{2}}\circ\theta_{T_{A_{3}}}+T_{A_{3}}$
and $T_{A_{3}}$, we have
\begin{align}
\begin{split}\sup_{x\in A_{2}}\overline{E}_{x}[e^{\frac{V}{t_{N}}}] & \,\text{\ensuremath{\overset{(\ref{eq:Vub})}{\leq}}}\sup_{x\in A_{2}}\overline{E}_{x}[e^{\frac{T_{A_{3}}+\overline{t_{*}}}{t_{N}}}]\bigg(1+\sup_{y\in U^{N}\backslash A_{3}}\overline{P}_{y}[H_{A_{2}}\leq\overline{t_{*}}]\sup_{x\in A_{2}}\overline{E}_{x}[e^{\frac{V}{t_{N}}}]\bigg)\\
 & \overset{(\ref{eq:hittingtimebig1})}{\leq}\sup_{x\in A_{2}}\overline{E}_{x}[e^{\frac{T_{A_{3}}+\overline{t_{*}}}{t_{N}}}]\bigg(1+N^{-c}\sup_{x\in A_{2}}\overline{E}_{x}[e^{\frac{V}{t_{N}}}]\bigg).
\end{split}
\label{eq:expub}
\end{align}
By Proposition \ref{prop:entranceinvup} we have
\begin{equation}
\frac{1}{t_{N}}\overset{(\ref{eq:tNdef})}{\leq}\frac{1}{\overline{E}_{\pi}[H_{A_{1}}]}\overset{(\ref{eq:entranceinvub})}{\leq}(1+N^{-c})\frac{\mathrm{cap}(A_{1})}{T_{N}}u_{**}(1+\epsilon)\overset{(\ref{eq:hNproperty})\,5.}{\underset{(\ref{eq:capacityorder})}{\leq}}cN^{-d+r_{1}(d-2)}.\label{eq:tNinvub}
\end{equation}
By an elementary estimate on simple random walk and the observation
that the diameter of $A_{3}$ is smaller than $cN^{r_{3}}$, we have
\begin{equation}
\overline{E}_{x}[T_{A_{3}}]\overset{(\ref{eq:coincidestopped})}{=}E_{x}[T_{A_{3}}]\leq cN^{2r_{3}}\quad\mbox{for all }x\in A_{3},\label{eq:TA3estimate}
\end{equation}
therefore we obtain that 
\begin{equation}
\frac{\sup_{x\in A_{3}}\overline{E}_{x}[T_{A_{3}}]}{t_{N}}\leq cN^{-d+2r_{3}+(d-2)r_{1}}\leq N^{-c'}.\label{poisson3}
\end{equation}
By an argument like (\ref{poisson1}), again with the help of Khasminskii's
Lemma (see (\ref{eq:Khasminskii})), we obtain that 
\begin{equation}
\sup_{x\in A_{2}}\overline{E}_{x}\big[\exp(\frac{T_{A_{3}}}{t_{N}})\big]\leq\frac{1}{1-N^{-c}}\leq e^{N^{-c'}}\mbox{ for large }N.\label{eq:TA3ub}
\end{equation}
Moreover, we obtain from (\ref{eq:tNinvub}) that 
\begin{equation}
\frac{\overline{t_{*}}}{t_{N}}\overset{(\ref{eq:regen})}{\leq}cN^{-c'}.\label{eq:tstnub}
\end{equation}
We apply (\ref{eq:TA3ub}) and (\ref{eq:tstnub}) to the right-hand
side of (\ref{eq:expub}), and conclude after rearrangement (and with
an implicit truncation argument where $V$ in (\ref{eq:Vub}) and
(\ref{eq:expub}) is replaced by $V\wedge M$) that 
\begin{equation}
\sup_{x\in A_{2}}\overline{E}_{x}[e^{V/t_{N}}]\leq e^{N^{-c}}.\label{poisson4}
\end{equation}
}}{\small \par}

{\small{We now return to (\ref{poisson0}). Substituting (\ref{poisson1})
and (\ref{poisson4}) into (\ref{poisson0}) and using the fact that
for $0\leq\theta\leq\frac{1}{2}$, 
\begin{equation}
(1-\theta)^{-1}\leq1+\theta+2\theta^{2}\leq e^{\theta+2\theta^{2}},\label{eq:thetacalc}
\end{equation}
we deduce that
\begin{eqnarray}
\overline{P}_{0}[R_{J}\geq T_{N}] & \leq & \exp\Big(-\theta\Big(1-\frac{\epsilon}{100}\Big)\frac{T_{N}}{t_{N}}+(\theta+2\theta^{2})J\Big)+\exp\Big(-\frac{\epsilon}{100}\frac{T_{N}}{t_{N}}+N^{-c}J\Big)\nonumber \\
 & \stackrel{(\ref{eq:defofk})}{\underset{(\ref{eq:thetacalc})}{\leq}} & \exp\Big(-\theta\Big(1-\frac{\epsilon}{100}\Big)\frac{T_{N}}{t_{N}}+(\theta+2\theta^{2})\lfloor(1+\epsilon/2)u_{**}\mathrm{cap}(A_{1})\rfloor\Big)\label{poisson5}\\
 &  & \qquad+\exp\Big(-\frac{\epsilon}{100}\frac{T_{N}}{t_{N}}+N^{-c}\lfloor(1+\epsilon/2)u_{**}\mathrm{cap}(A_{1})\rfloor\Big).\nonumber 
\end{eqnarray}
Recall the definition of $f_{A_{1}}$ in (\ref{eq:fA1def}). Using
Lemma \ref{lem:fA1lb}, we know that
\begin{equation}
\frac{\overline{E}_{x}[H_{A_{1}}]}{\overline{E}_{\pi}[H_{A_{1}}]}=1-f_{A_{1}}(x)\overset{(\ref{eq:fA1lb})}{\leq}1+N^{-c}\leq(1-\frac{\epsilon}{100})^{-1}\mbox{ for all }x\in U^{N}.
\end{equation}
Hence, by Proposition \ref{prop:entranceinvlb} we obtain that 
\begin{equation}
\frac{T_{N}}{t_{N}}\geq(1-\frac{\epsilon}{100})\frac{T_{N}}{\overline{E}_{\pi}[H_{A_{1}}]}\overset{(\ref{eq:entranceinvlb})}{\geq}(1-\frac{\epsilon}{50})(1+\epsilon)u_{**}\mathrm{cap}(A_{1}).\label{eq:TNtNlb}
\end{equation}
By choosing an appropriately small $\theta$ and applying (\ref{eq:TNtNlb})
we know that
\begin{equation}
-\theta\Big(1-\frac{\epsilon}{100}\Big)\frac{T_{N}}{t_{N}}+(\theta+2\theta^{2})\lfloor(1+\epsilon/2)u_{**}\mathrm{cap}(A_{1})\rfloor\leq-N^{c}\mbox{ for large }N,\label{eq:poisson6}
\end{equation}
moreover, we also know that 
\begin{equation}
-\frac{\epsilon}{100}\frac{T_{N}}{t_{N}}+N^{-c}\lfloor(1+\epsilon/2)u_{**}\mathrm{cap}(A_{1})\rfloor\leq-N^{c'}\mbox{ for large }N.\label{eq:poisson7}
\end{equation}
Inserting (\ref{eq:poisson6}) and (\ref{eq:poisson7}) into (\ref{poisson5}),
we obtain (\ref{eq:enoughexcursion}) as desired.}}{\small \par}\end{proof}

\end{document}